\documentclass[a4paper, 10pt, twoside, notitlepage]{amsart}
\usepackage{amsmath,amscd}
\usepackage{amssymb}
\usepackage{comment}
\usepackage{graphicx, xcolor}

\usepackage{mathrsfs}
\usepackage[ocgcolorlinks, linkcolor=blue]{hyperref}

\usepackage{bm}
\usepackage{bbm}
\usepackage{url}

\usepackage{geometry}
\usepackage[utf8]{inputenc}
\usepackage{mathtools,amssymb,amsmath}
\usepackage{esint}
\usepackage{tikz}
\usepackage{dsfont}
\usepackage{relsize}
\usepackage{url}
\usepackage{xcolor}
\usepackage{graphicx}
\usepackage{mathrsfs}
\usepackage[shortlabels]{enumitem}
\usepackage{lineno}
\usepackage{amsmath}
\usepackage{enumitem}
\usepackage{amsthm} 
\usepackage{verbatim}
\usepackage{dsfont}
\usepackage{booktabs}
\numberwithin{equation}{section}

\usepackage{slashed}

\allowdisplaybreaks

\mathtoolsset{showonlyrefs}

\graphicspath{{images}}

\newtheorem{theorem}{Theorem}[section]
\newtheorem{lemma}[theorem]{Lemma}
\newtheorem{definition}{Definition}[section]

\newtheorem{proposition}[theorem]{Proposition}
\newtheorem{example}{Example}[section]

\newtheorem{remark}{Remark}[section]

\title[Inverse Kirchhoof plate equations]{Inverse problems for nonlinear Kirchhoff plate equations with multiple unknown parameters}

\author[S. Fu]{Song-Ren Fu}
\address{School of Mathematics and Statistics, Northwestern Polytechnical University, Xi’an 710129, Shaanxi, PR China.}
\email{songrenfu@nwpu.edu.cn}

\author[H. Liu]{Hongyu Liu$^*$}
\address{Department of Mathematics, City University of Hong Kong, Kowloon, Hong Kong SAR, PR China.}  
\email{hongyliu@cityu.edu.hk}

\author[Y. Yu]{Yongyi Yu}
\address{School of Mathematics, Jilin University, People's Republic of China.}
\email{yuyy122@jlu.edu.cn}

\author[T. Zheng]{Tianyi Zheng}
\address{School of Mathematical Sciences and LPMC, Nankai University, Tianjin 300071, People's Republic of China.}
\email{9820250139@nankai.edu.cn}

\date{August 6, 2026. $^*$Corresponding author}

\DeclareMathOperator{\supp}{supp} 

\begin{document}

\begin{abstract}
This paper provides a comprehensive treatment of inverse boundary value problems for (nonlinear) Kirchhoff plate equations under diverse general settings. We begin by establishing the global well-posedness of the nonlinear forward equations, which not only underpins the subsequent inverse analysis but also holds independent theoretical significance. The inverse problems are then examined for both passive and active measurement regimes. With a single passive boundary measurement, we establish the stable recovery of the unknown initial data. In the active regime with infinitely many boundary measurements, our results are twofold. For linear equations featuring generic time-dependent potentials—allowing for spatial unboundedness—we demonstrate the simultaneous recovery of both initial data and coefficients. For nonlinear equations, where both the nonlinearity and initial data are unknown, we develop a novel Runge approximation approach, together with carefully constructed geometric optics solutions and higher-order linearization around nonzero solutions, to prove their simultaneous determination. Furthermore, we introduce a delicate cut-off technique that provides an alternative means of addressing the scenario of vanishing initial data. Notably, the methodologies and results developed herein are readily generalizable to other boundary conditions and plate models, including the classical Euler–Bernoulli equation.
\medskip


		
\end{abstract}

\maketitle

\setlength{\parskip}{0.8ex} 

\section{Introduction}
Elastic (thin) plate theory is a fundamental model in structural engineering and solid mechanics, and also finds applications in flow-structure interaction problems. For instance, aircraft wings of small aspect ratio are not amenable to beam theory, as their two principal dimensions greatly exceed their thickness. Nevertheless, such structures are usually well approximated by plate models. In the same vein, aircraft fuselages are composed of thin-walled members stiffened by ribs and longerons, and the thin‑walled portions lying between the stiffeners may be treated as thin plates. The corresponding models that consider tangential deflection are known as shell equations. Within linear elasticity, the Kirchhoff plate theory provides a classical framework for analyzing thin plates. The assumptions of Kirchhoff plate theory concern the kinematics of a material line that is initially normal to the mid-plane. For a comprehensive account of plate theory and the associated mathematical equations, the reader is referred to, e.g., the references \cite{1988Modelling,doi:10.1137/1.9781611970821,bauchau2009kirchhoff,mittelstedt2023theory}.

The study of Kirchhoff plate equations has received significant interest in the context of forward problems, such as well-posedness, controllability, long time behaviour, and scattering problems, see \cite{avalos1998exponential,ALSICON00,bourgeois2020well,EllerLasieckaTriggiani,LASIECKA199162,lasiecka2020sharp,lasiecka2000structural,peng2020long,zhang2006sharp,zhang2008optimality} and the references therein. However, the literature on inverse problems for \emph{time-domain} Kirchhoff plate equations remains relatively limited. Motivated by both practical and theoretical considerations, we investigate in this paper inverse problems of recovering unknown parameters for Krichhoff plate equations from appropriate  measurements taken on the boundary of the space-time domain.

Throughout this paper, let $\Omega\subset\mathbb R^n$ ($n\ge 1$) be a bounded domain with smooth boundary $\Gamma\vcentcolon=\partial\Omega$. For any set $A\subset\mathbb R^n$, denote by $A_T\vcentcolon=A\times (0,T)$ for $T>0$.
In this paper, mathematically, we study the following generalized nonlinear Kirchhoff plate equation:
\begin{equation}\label{eq:intro-non-lin-plate}
\begin{cases}
(I-\gamma \Delta)\partial_t^2u+ \Delta^2 u=F(x,t,u) & {\rm in}\,\, \Omega_T,\\
u=h_1,\, \Delta u=h_2 & {\rm on}\,\, \Gamma_T,\\
u(0)=\eta_1,\,\partial_tu(0)=\eta_2 & {\rm in}\,\, \Omega,
\end{cases}
\end{equation}
where $I$ is the identity operator, $u=u(x,t)$ denotes the vertical displacement of the plate, \(\gamma>0\) is a parameter that is proportional to the square of the plate thickness, and the term $\gamma\Delta \partial_t^2u$ represents rotational forces of the plate. In particular, when $\gamma=0$, it becomes the classical Euler-Bernoulli equation with principal part $\partial_t^2u+\Delta^2u$. 
The boundary pair $(u,\Delta u)|_{\Gamma_T}=(h_1,h_2)$ denotes the non-homogeneous hinged boundary condition, and 
the nonlinearity $F$ stands for some semilinear terms that arise from e.g.,  nonlinear perturbations of the system. In particular, it reduces to some usual external source term provided that $F=F(x,t)$ is independent of $u$. 


Physically, the triple $(\eta_1,\eta_2,F)$ may be regarded as unknown sources, and all of them generate the solutions $u$ of the equation \eqref{eq:intro-non-lin-plate}. With zero boundary conditions ($h_1=h_2=0$), the solutions and thus the boundary measurements are uniquely generated by the sources $(\eta_1,\eta_2,F)$. In this context, the measurements are called \emph{passive}. In contrast, \emph{active measurements} involve introducing a known source into the system to generate observable outputs, which together with the input constitute a typical dataset for inverse problems.

\subsection{Formulation of the associated inverse problems}
Before formulating the inverse problems considered in this paper, we present some necessary notions.

Let $\langle\cdot,\cdot\rangle_0$ be the standard inner product of $\mathbb R^n$. We introduce the following (partial) boundary set
\begin{equation}\label{cond-for-Gamma-0}
\Gamma_0\vcentcolon=\{x\in\Gamma: \langle x-x_0,\nu\rangle_0>0  \},
\end{equation}
where $x_0\in\mathbb R^n\backslash\overline\Omega$ is fixed, and $\nu$ denotes the unit outward normal vector field along $\Gamma$ of $\Omega$.  For this fixed $x_0$, denote by
\begin{equation}
0<R_0\vcentcolon=\gamma^{\frac12}\min_{\overline\Omega}|x-x_0|<R_1\vcentcolon=\gamma^{\frac12}\max_{\overline\Omega}|x-x_0|.
\end{equation}
Due to the presence of $\gamma\in (0,1)$ in our setting, we note that the definitions of $R_0,R_1$ are defined by suitably modifying those in \cite[equation (20)]{zhang2006sharp}, where $\gamma$ was set to $1$. Indeed, based on the discussions in Remark \ref{rem:plate-wave-gamma}, we can introduce a  metric $g=\gamma ds^2$, where $ds^2$ is the standard Euclidean metric. Thus, we have $|x-x_0|_g=\gamma^{\frac12}|x-x_0|$.
Again by \cite{zhang2006sharp}, we also choose a constant $c\in (0,1)$ and $T>2R_1$ such that
\begin{equation}\label{cond-R0-R1-T}
\frac{(4+5c)R_0^2}{9c}>R_1^2,\quad  c^2T^2>4R_1^2.
\end{equation}
The observation time $T_0$ is given by
\begin{equation}\label{observe-time-T0}
T_0\vcentcolon=\frac{2\gamma^{\frac12}}{c}\max_{\overline\Omega}|x-x_0|=\frac{2}{c}R_1.
\end{equation}
\begin{remark}
Assume that $\Omega\subset B_L(0)\vcentcolon=\{x\in\mathbb R^n: |x|<L\}$ for some sufficiently large $L>0$. Choose $x_0\in\mathbb R^n\backslash\overline\Omega$ far away from $\Omega$ such that $R_0= C_0L$ and $R_1=(C_0+2)L$ for some $C_0\ge 1$. From \eqref{cond-R0-R1-T}, we can find that
\begin{equation}
0<c<\frac{9(1+2C_0^{-1})^2-5}{4}<1.
\end{equation}
Hence, if $C_0$ is large enough, then $c$ is close to $1$, and thus $T_0$ is close to $2R_1$.
\end{remark}

\subsubsection{Parameter recovery via passive measurements}
Let us first define a function space $\mathcal H^0_\Omega\vcentcolon=\mathcal H_1\times\mathcal H_2$, where
\begin{equation}\label{def:H1-H2}
\mathcal H_1=\{v\in H^3(\Omega): v|_\Gamma=\Delta v|_\Gamma=0\},\quad \mathcal H_2=H^2(\Omega)\cap H_0^1(\Omega).
\end{equation}
We first consider the inverse problem of recovering the initial data of the following nonlinear Kirchhoff plate equation with homogeneous boundary conditions,
\begin{equation}\label{eq:nonlin-plate-homo-boundary}
\begin{cases}
(I-\gamma\Delta)\partial_t^2u+\Delta^2u=f(x,u) & {\rm in}\,\Omega_T,\\
u=\Delta u=0 & {\rm on}\, \Gamma_T,\\
u(0)=\eta_1,\, \partial_tu(0)=\eta_2 & {\rm in}\, \Omega.
\end{cases}
\end{equation}
Under suitable conditions on nonlinearities $f$ and initial data $(\eta_0,\eta_1)$, the global well-posedness of \eqref{eq:nonlin-plate-homo-boundary} will be established in Section 2.3 (see Theorem \ref{thm:global-well-posed-H3}). Hence, for the moment, we formally define the map
\begin{equation}\label{passive-meas-initial-data}
\Lambda^{0,\Gamma_0}_{\eta_1,\eta_2,f}=(\partial_\nu u,\partial_\nu\Delta u)|_{\Gamma_{0T}},
\end{equation}
where $u$ denotes the solution of \eqref{eq:nonlin-plate-homo-boundary} associated to $\eta_1,\eta_2$ and $f$, $\Gamma_0\subset\Gamma$ was defined in \eqref{cond-for-Gamma-0}, $\Gamma_{0T}\vcentcolon=\Gamma_0\times (0,T)$ is part of the lateral boundary $\Gamma_T$. For given Banach spaces $X(\Gamma_{0T})$ and $Y(\Gamma_{0T})$ on $\Gamma_{0T}$, we denote
\begin{equation}
\|\Lambda^{0,\Gamma_0}_{\eta_1,\eta_2,f}\|_{X(\Gamma_{0T})\times Y(\Gamma_{0T})}=\bigl(\|\partial_\nu u\|^2_{X(\Gamma_{0T})} +\|\partial_\nu\Delta u\|^2_{Y(\Gamma_{0T})} \bigr)^{\frac12}.
\end{equation}
We note that $\Lambda^{0,\Gamma_0}_{\eta_1,\eta_2,f}$ generates the so-called \emph{passive measurement}, since the solution $u$ of the equation \eqref{eq:nonlin-plate-homo-boundary} is uniquely determined by the sources $(\eta_1,\eta_2,f)$, and there are no prescribed boundary data serving as inputs.

We are concerned with the following inverse problem:

{\bf (IP-1)} Can one recover the unknown initial data $(\eta_1,\eta_2)$ of \eqref{eq:nonlin-plate-homo-boundary} by using the passive boundary measurement $\Lambda^{0,\Gamma_0}_{\eta_1,\eta_2,f}$?

For this question, we will present a theorem (see Theorem \ref{thm:deter-initial-data}) that gives an affirmative answer to {\bf (IP-1)}.

\subsubsection{Parameter recovery via active measurements}

To avoid confusion of notations, we replace $f$ by $g$ and consider the following nonlinear Kirchhoff plate equation
\begin{equation}\label{eq:intro-non-lin-plate-g}
\begin{cases}
(I-\gamma \Delta)\partial_t^2u+ \Delta^2 u=g(x,t,u) & {\rm in}\,\, \Omega_T,\\
u=h_1,\, \Delta u=h_2 & {\rm on}\,\, \Gamma_T,\\
u(0)=\eta_1,\,\partial_tu(0)=\eta_2 & {\rm in}\,\, \Omega.
\end{cases}
\end{equation}
For the moment, we assume that the nonlinear equation \eqref{eq:intro-non-lin-plate-g} is well-posed. Thus, we formally
define the input-to-output map associated to  \eqref{eq:intro-non-lin-plate-g} with respect to the parameters $g,\eta_1,\eta_2$ by
\begin{equation}
\Lambda^{\Gamma,T}_{g,\eta_1,\eta_2}(h_1,h_2)=\bigl(\partial_\nu u|_{\Gamma_T}, \partial_\nu\Delta u|_{\Gamma_T}, u|_{t=T}\bigr),\quad (h_1,h_2)\in \mathcal S_\epsilon(\Gamma_T),
\end{equation}
where the function space $\mathcal S_\epsilon(\Gamma_T)$ will be given later (see \eqref{set:admi-nonlin-g}). We note that the (local) well-posedness of \eqref{eq:intro-non-lin-plate-g} will be established in Section \ref{subsec:well-posedness-non-small}.

In particular, we are also interested in the linear case where $g(x,t,u)=-q(x,t)u$. If $q$ is unknown, and the initial values $\eta_1,\eta_2$ are all known functions, then we define a map related to $q$ by
\begin{equation}
\Lambda^{\Gamma,T}_q(h_1,h_2,\eta_1,\eta_2)=\bigl(\partial_\nu u|_{\Gamma_T}, \partial_\nu\Delta u|_{\Gamma_T}, u|_{t=T}\bigr),\,\, (h_1,h_2,\eta_1,\eta_2)\in\mathcal H^T_{\Gamma,\Omega},
\end{equation}
where the function space $\mathcal H_{\Gamma,\Omega}^T$ is given by
\begin{equation}\label{func-space-initial-boundary-comp}
\begin{split}
\mathcal H_{\Gamma,\Omega}^T\vcentcolon=&\{(h_1,h_2,\eta_1,\eta_2)\in H^3(\Omega)\times H^2(\Omega)\times H^2(0,T;H^3(\Gamma))\\
&\times H^2(0,T;H^1(\Gamma)):
 \eta_1|_\Gamma=h_1(0),\, \Delta\eta_1|_\Gamma=h_2(0),\, \eta_2|_\Gamma=h_2(0)\}.
\end{split}
\end{equation}
If $q$ and the initial pair $(\eta_1,\eta_2)\in\mathcal H^0_\Omega=\mathcal H_1\times\mathcal H_2$ are unknown functions, we define two maps related to $q,\eta_1,\eta_2$ by
\begin{equation}
\Lambda^{\Gamma,T}_{q,\eta_1,\eta_2}(h_1,h_2)=\bigl(\partial_\nu u|_{\Gamma_T}, \partial_\nu\Delta u|_{\Gamma_T}, u|_{t=T}\bigr),\quad (h_1,h_2)\in \mathcal H^T_{\Gamma},
\end{equation}
\begin{equation}
\tilde \Lambda^{\Gamma,T}_{q,\eta_1,\eta_2}(h_1,h_2)=\bigl(\partial_\nu u|_{\Gamma_T}, \partial_\nu\Delta u|_{\Gamma_T}\bigr),\quad (h_1,h_2)\in \mathcal H^T_{\Gamma}.
\end{equation}
Here, we have used the following function spaces
$$\mathcal H_1(\Gamma_T)\vcentcolon=H_0^2(0,T;H^3(\Gamma)),\,\, \mathcal H_2(\Gamma_T)\vcentcolon=H^2(0,T;H^1(\Gamma))\cap H_0^1(0,T;H^1(\Gamma)),$$ and
\begin{equation}\label{function-space-initial-boundary}
\mathcal H^T_{\Gamma}\vcentcolon=\mathcal H_1(\Gamma_T)\times \mathcal H_2(\Gamma_T).
\end{equation}
More function spaces will be explicitly introduced in Section \ref{subsec:well-posed-prelimi}.
The map $\tilde \Lambda^{\Gamma,T}_{q,\eta_1,\eta_2}$ can be regarded as a variant of the classical Dirichlet-to-Neumann (DN) map for Kirchhoff plate equations. By Theorem \ref{proposition-well-posedness-lin-u} in Section \ref{subsec:well-posedness-lin}, we know that the maps $\Lambda^{\Gamma,T}_q$, $\Lambda^{\Gamma,T}_{q,\eta_1,\eta_2}$ and $\tilde \Lambda^{\Gamma,T}_{q,\eta_1,\eta_2}$ are well-defined provided that $q$ belongs to some suitable function space. 

We emphasize that, unlike the passive measurement $\Lambda^{0,\Gamma_0}_{\eta_1,\eta_2,f}$, the maps $\Lambda^{\Gamma,T}_{\eta_1,\eta_2,g}$, $\Lambda^{\Gamma,T}_q$, $\Lambda^{\Gamma,T}_{q,\eta_1,\eta_2}$ and $\tilde \Lambda^{\Gamma,T}_{q,\eta_1,\eta_2}$ are called \emph{active measurements}, since the inputs are not identically zero and originate from known initial boundary conditions $h_1,h_2,\eta_1,\eta_2$, or only from some known boundary pair $(h_1,h_2)$. 

For the nonlinear Kirchhoff plate equation \eqref{eq:intro-non-lin-plate-g} (including the linear case $g(u)=-qu$), we are concerned with the following inverse problems:

{\bf (IP-2.1)} For the linear case $g(x,t,u)=-qu$, can one recover the time-dependent linear coefficient $q$ by using the active  measurement $\Lambda_q^{\Gamma,T}$?

{\bf (IP-2.2)} For the linear case $g(x,t,u)=-qu$, can one recover the time-dependent linear coefficient $q$ and the initial data $\eta_1,\eta_2$ simultaneously by using the active measurement $\Lambda^{\Gamma,T}_{q,\eta_1,\eta_2}$ or $\tilde \Lambda^{\Gamma,T}_{q,\eta_1,\eta_2}$?

{\bf (IP-3)} Can one recover some nonlinearities $g$ and the initial data $(\eta_1,\eta_2)$ simultaneously  by using the active measurement $\Lambda^{\Gamma,T}_{g,\eta_1,\eta_2}$?

Denote by
\begin{equation}
\mathcal L\vcentcolon=(I-\gamma\Delta)\partial_t^2+\Delta^2.
\end{equation}
For the sake of better illustrating the inverse problems studied in this paper, we collect them in the following table.
{\small
\begin{table}[h]
\centering
\small
\caption{Kirchhoff plate equations and associated inverse problems}
\label{tableII}
\begin{tabular}{c @{\quad} c  @{\quad} c}
\toprule
\textbf{Equations} 
& \textbf{Recovered parameters} 
& \textbf{Measurement data} \\
\midrule
$\mathcal Lu=f(x,u)$ & $\eta_1(x),\eta_2(x)$ & $\Lambda^{0,\Gamma_0}_{\eta_1,\eta_2,f}$ \\
\midrule
$\mathcal Lu+qu=0$  & $q(x,t)$  & $\Lambda_q^{\Gamma,T}(h_1,h_2,\eta_1,\eta_2)$  \\
\midrule
$\mathcal Lu+qu=0$ & $\eta_1(x),\eta_2(x),q(x,t)$ & $\Lambda^{\Gamma,T}_{q,\eta_1,\eta_2}(h_1,h_2)$ or $\tilde \Lambda^{\Gamma,T}_{q,\eta_1,\eta_2}(h_1,h_2)$ \\
\midrule
$\mathcal Lu=g(x,t,u)$  & $\eta_1(x),\eta_2(x),g(x,t,z)$ & $\Lambda^{\Gamma,T}_{g,\eta_1,\eta_2}(h_1,h_2)$ \\
\bottomrule
\end{tabular}
\end{table}}

\subsection{Main assumptions and theorems}
This subsection is devoted to introducing assumptions on the nonlinearities $f,g$, and presenting the main theorems.

\subsubsection{Main assumptions}
Let us begin with recalling the definition of Carath\'eodory functions.
\begin{definition}
Let $U\subset\mathbb R^n$ be an open set. A function $f:U\times \mathbb R\to \mathbb R$ is called a \emph{Carath\'eodory function} if it satisfies

$(1)$ $\tau\mapsto f(x,\tau)$ is continuous for a.e. $x\in U$,

$(2)$ $x\mapsto f(x,\tau)$ is measurable for all $\tau\in \mathbb R$.
\end{definition}
We will sometimes use the symbol $\lesssim$ to denote some inequality holds up to a positive constant whose value is irrelevant for our arguments.

\noindent{\bf Assumptions on the nonlinearity $f$.}
Let $\mathfrak m=\frac{5n}{2}$ or satisfy
\begin{equation}\label{cond-for-m}
\begin{cases}
\mathfrak{m}\ge \frac n2 & \text{if}\, \, n>4,\\
\mathfrak{m} >2 & \text{if}\,\, n=4,\\
\mathfrak{m}\ge 2 & \text{if}\,\, 1\le n< 4.
\end{cases}
\end{equation}
For the nonlinearity $f:\Omega_T\times \mathbb R\to\mathbb R$, we impose the following assumptions:

{\bf (A.1)} $f(\cdot,0)\in L^2(\Omega)$, and $f$ has partial derivative $\partial_\tau f$, which is a Carath\'eodory function, and there exist a function $a\in L^{\mathfrak m}(\Omega)$ and $\mathfrak r\in\mathbb R_{\ge 0}$ such that
\begin{equation}\label{cond:for-partial-tau-f}
|\partial_\tau f(x,\tau)|\lesssim |a(x)|+|\tau|^{\mathfrak r}
\end{equation}
for all $\tau\in\mathbb R$ and a.e. $x\in\Omega$. 

{\bf (A.2)} Denote $F:\Omega\times \mathbb R\to \mathbb R$ by
\begin{equation}
F(x,\tau)=\int_0^\tau f(x,s)\, ds.
\end{equation}
There is a positive constant $M_0$ such that $F$ fulfills $F(x,\tau)\le M_0$ for all $\tau\in \mathbb R$ and a.e. $x\in\Omega$.

Let $T>0$ be given. We define a set for nonlinearities $f$ by
\begin{equation}
\mathcal M_{\mathfrak r,\mathfrak m}^T\vcentcolon=\{f:\overline\Omega\times \mathbb R\to\mathbb R:\, f\, \text{satisfies {\bf (A.1)} and {\bf (A.2)}}\},
\end{equation}
for some $\mathfrak r\in\mathbb R_{\ge 0}$.
In order to establish the global well-posedness for the nonlinear problem \eqref{eq:nonlin-plate-homo-boundary}, we introduce an exponent $r$ that satisfies
\begin{equation}\label{cond-for-r}
\begin{cases}
0\le r\le \frac{4}{n-4} & \text{if}\, \, n>4,\\
0\le r<\infty  & \text{if}\,\, 1\le n\le 4.
\end{cases}
\end{equation}

\begin{example}\label{EXL:ASSUMPTIONS}
We next give some nontrivial examples of $f\in \mathcal M_{r,\mathfrak m}^T$ with $r$ satisfying the conditions \eqref{cond-for-r}.

{\rm (1)} Let $f(x,\tau)=\alpha(x)|\tau|^r\tau$ be the power type nonlinearity with $r$ satisfying \eqref{cond-for-r}. It is not difficult to see that 
$$|\partial_\tau f(x,\tau)|=|\alpha(x)|(r+1)|\tau|^r\lesssim |\tau|^r,\quad F(x,\tau)=\alpha(x)\frac{|\tau|^{r+2}}{r+2}\le 0$$ whenever $\alpha\in L^{\infty}(\Omega)$ is non-positive. Clearly, $f$ verifies {\bf (A.1)} and {\bf (A.2)}.

{\rm (2)} Let $\alpha\in L^\infty(\Omega)$, and $f(x,\tau)=\alpha(x)\cos\tau$ be the super-linear nonlinearity of $\cos$-type. We see that 
$$|\partial_\tau f(x,\tau)|=|\alpha(x)\sin\tau|\le |\alpha(x)|,\quad F(x,\tau)=\alpha(x)\sin\tau\le |\alpha(x)|\le C$$
for any $\tau\in\mathbb R$, and a.e. $x\in\Omega$. Hence $f$ satisfies the assumptions {\bf (A.1)}-{\bf (A.2)} by taking $p_0=\infty$ and $r=0$.

{\rm (3)} Let $1\le n\le 6$, and let $f(x,\tau)=\alpha(x)\tau e^{-\tau^2}$ be the Gaussian type nonlinearity with $\alpha\in L^\infty(\Omega)$. It is clear that the assumptions {\bf (A.1)} and {\bf (A.2)} are satisfied by checking the following
$$|\partial_\tau f(x,\tau)|=|\alpha(x)e^{-\tau^2}(1-2\tau^2)|\lesssim |\alpha(x)|+|\tau|^2,$$
and 
\begin{equation}
F(x,\tau)=\frac{\alpha(x)}{2}(1-e^{-\tau^2})\le C.
\end{equation}

{\rm (4)} Let $1\le n\le 6$, and let $f(x,\tau)=\alpha(x)\tau\ln (1+\tau^2)+\beta(x)\tau^3$ be the mixed logarithmic-power type nonlinearity with $\alpha\in L^{\infty}(\Omega)$ being non-positive and $\beta\in L^\infty(\Omega)$ being strictly negative. We can check that 
\begin{equation}
\begin{split}
|\partial_\tau f(x,\tau)|&=\Big|\frac{\alpha(x)}{1+\tau^2}[(1+\tau^2)\ln(1+\tau^2)+2\tau^2]+3\beta(x)\tau^2\Big|\\
&\lesssim |\alpha(x)|+|\beta(x)|+|\tau|^2,
\end{split}
\end{equation}
and
\begin{equation}
\begin{split}
F(x,\tau)&=\frac12\alpha(x)(1+\tau^2)\ln (1+\tau^2)+\frac14\tau^2(\beta(x)\tau^2-2\alpha(x))\\
&\le \frac{\beta}{4}(\tau^2-\alpha\beta{-1})^2-\frac{\alpha^2}{4\beta}\le -\frac{\alpha^2}{4\beta},
\end{split}
\end{equation}
for any $\tau\in\mathbb R$ and a.e. $x\in\Omega$.
Hence, $f$ satisfies the assumptions {\bf (A.1)}-{\bf (A.2)}.
\end{example}
In order to study the inverse {\bf (IP-1)} for the nonlinear problem \eqref{eq:nonlin-plate-homo-boundary}, we also introduce an exponent $r_0$ that satisfies
\begin{equation}\label{cond-for-r0-IP}
\begin{cases}
0\le r_0\le \frac{4}{5(n-6)} & \text{if}\, \, n>6,\\
0\le r_0\le \frac{4}{n-4}  & \text{if}\,\, n=5,6,\\
0\le r_0<\infty  & \text{if}\,\, 1\le n\le 4,
\end{cases}
\end{equation}

\begin{remark}
The assumption that $f\in\mathcal M_{r,\mathfrak m}^T$ with $r$ and $\mathfrak m$ respectively satisfying \eqref{cond-for-r} and \eqref{cond-for-m} is sufficient to guarantee the global well-posedness of the nonlinear equation \eqref{eq:nonlin-plate-homo-boundary}. In contrast, $f\in\mathcal M_{r_0,5n/2}^T$ with $r_0$ satisfying \eqref{cond-for-r0-IP} and $\mathfrak m=\frac{5n}{2}$, is mainly used in studying {\bf (IP-1)} and is somewhat more restrictive, arising from the regularity requirement of the linear potential in the controllability inequality for linear Kirchhoff plate equations (see Lemma \ref{lem:controllability-ineq}).

We mention that, in \cite[condition (2.4)]{Lin-JLMS}, the authors assumed that the time-dependent nonlinearity $f:\Omega_T\times\mathbb R\to \mathbb R$ satisfies the following increasing condition: 
\begin{equation}\label{cond:LLL-24-f}
\limsup_{s\to\infty}\frac{\partial_sf(x,t,s)}{\ln|s|}=0,
\end{equation}
uniformly for $(x,t)\in \overline\Omega_T$. In particular, if we choose $f$ as given by (1) of Example \ref{EXL:ASSUMPTIONS}, then \eqref{cond:LLL-24-f} does not hold for such $f$. Hence, the increasing conditions imposed  on $f\in\mathcal M_{r,\mathfrak m}^T$  (and even $f\in\mathcal M_{r_0,5n/2}^T$) are less restrictive. Moreover, in the most practical cases $n=1,2,3$, there are no restrictions on $r$ and $r_0$. However, in this paper we restrict ourselves to considering time-independent nonlinearities $f$.
\end{remark}


\noindent{\bf Assumptions on the nonlinearity $g$.} We impose the following conditions on $g$.

{\bf (B.1)} For given $T>T_0$, $g(x,t,z):\overline\Omega\times[0,T]\times\mathbb C\to\mathbb C$ is analytic on $z\in\mathbb C$ with $g(\cdot,\cdot,0)=0$ in $\overline\Omega_T$, and $\partial_z^kg(\cdot,\cdot,z_0)\in L^\infty(0,T;L^{\frac{5n}{2}}(\Omega))$ for any $k\in\mathbb N$ and $z_0\in\mathbb C$.

Let $0<t_1<t_2<T$, and let $g_0$ be some known complex-valued $C^\infty$-function. We define an admissible set for $g$ by
\begin{equation}\label{set:admi-nonlin-g}
\mathcal G_{g_0}^{t_1,t_2}\vcentcolon=\{g\,\,\text{satisfies {\bf (B.1)}}: g=g_0\,\, \text{in}\,\, \overline\Omega\times ([0,t_1)\cup (t_2,T])\times\mathbb C\}.
\end{equation}
It is clear that, if $g_1,g_2\in\mathcal G_{g_0}^{t_1,t_2}$ for some $t_1,t_2$, then ${\rm supp}\, (g_1-g_2)(\cdot,\cdot,z)\subset \overline\Omega\times[t_1,t_2]$ for any $z\in\mathbb C$.


Furthermore, for given constant $\epsilon>0$, we define an admissible set for the boundary conditions $(h_1,h_2)$ by
\begin{equation}\label{set:admi-boundary-data}
\mathcal S_\epsilon(\Gamma_T)\vcentcolon=\{(h_1,h_2)\in \mathcal H_\Gamma^T: \|(h_1,h_2)\|_{\mathcal H_\Gamma^T}\le\epsilon/2 \},
\end{equation}
where the function space $\mathcal H_\Gamma^T$ is given by \eqref{function-space-initial-boundary} with norm 
\begin{equation}
\|(h_1,h_2)\|_{\mathcal H_\Gamma^T}=\bigl(\|h_1\|_{H^2(0,T;H^3(\Gamma))}^2+\|h_2\|^2_{H^2(0,T;H^1(\Gamma))}  \bigr)^{\frac12}.
\end{equation}
We also define an admissible set for initial data by
\begin{equation}\label{set:initial-data-small}
\mathcal S_\epsilon(\Omega)\vcentcolon=\{(\eta_1,\eta_2)\in \mathcal H^0_\Omega=\mathcal H_1\times\mathcal H_2: \|(\eta_1,\eta_2)\|_{\mathcal H_\Omega^0}\le\epsilon/2\},
\end{equation}
where the function space $\mathcal H_\Omega^0$ is given by \eqref{def:H1-H2} with norm 
\begin{equation}
\|(\eta_1,\eta_2)\|_{\mathcal H_\Omega^0}=\bigl(\|\eta_1\|^2_{H^3(\Omega)}+\|\eta_2\|^2_{H^2(\Omega)} \bigr)^{\frac12}.
\end{equation}

\begin{remark}
For the local well-posedness of the nonlinear equation \eqref{eq:intro-non-lin-plate-g}, we restrict ourselves to consider the most practical cases $n=1,2,3$ only. The reason is simply that we concentrate on recovering coefficients $\partial_z^kg$ $(k\ge 1)$ with low regularity. Moreover, we can directly use the $E_3^T$-regularity of solutions (see \eqref{def:function-space-Em-T} for the definition of $E_3^T$), which can be continuously embedded into $C(\overline\Omega_T)$, for linear Kirchhoff plate equations established in Theorem \ref{proposition-well-posedness-lin-u}. 

\end{remark}

\subsubsection{Main results} 
The first theorem concerning {\bf (IP-1)} is stated as follows. 

\begin{theorem}\label{thm:deter-initial-data}
Let $T>T_0$ be given, where $T_0$ is given by \eqref{observe-time-T0}. Assume that the nonlinearity $f\in\mathcal M^T_{r_0,5n/2}$ with $r_0$ satisfying \eqref{cond-for-r0-IP}, and the initial data $(\eta_{1j},\eta_{2j})\in\mathcal H^0_\Omega$ (see \eqref{def:H1-H2}) for $j=1,2$. Let $\Lambda^{0,\Gamma_0}_{\eta_{1j},\eta_{2j},f}$ be the passive measurements of the nonlinear Kirchhoff plate equation \eqref{eq:nonlin-plate-homo-boundary} with respect to $f$ and $\eta_{1j},\eta_{2j}$ for $j=1,2$. Then, 
\begin{equation}\label{est:thm1-deter-initial-data}
\begin{split}
&\|\eta_{11}-\eta_{12}\|_{H^3(\Omega)}+\|\eta_{21}-\eta_{22}\|_{H^2(\Omega)}\\
&\le C\|\Lambda_{\eta_{11},\eta_{21},f}^{0,\Gamma_0}-\Lambda_{\eta_{12},\eta_{22},f}^{0,\Gamma_0}\|_{H^1(0,T;L^2(\Gamma_{0}))\times L^2(\Gamma_{0T}) }.
\end{split}
\end{equation}
\end{theorem}
\begin{remark}
The above theorem shows that the initial data can be stably recovered in the presence of some nonlinear perturbation that is known a priori. However, if the nonlinearity $f$ is unknown, then the passive measurement $\Lambda^{0,\Gamma_0}_{\eta_1,\eta_2,f}$ fails to recover the initial values. See the second part of  Section \ref{subsection-deter-initial} for more details. 
\end{remark}

We define an admissible set for $q$ by
\begin{equation}\label{set:for-q-q-0}
\mathcal U_{q_0}^{t_1,t_2}\vcentcolon=\{q\in L^\infty(0,T;L^{\frac{5n}{2}}(\Omega)):\, q=q_0\,\, \text{in}\,\, \Omega\times ([0,t_1)\cup (t_2,T])\}
\end{equation}
where $0<t_1<t_2<T$, and $q_0\in L^\infty(0,T;L^{\frac{5n}{2}}(\Omega))$ is some known function. This implies, if $q_1,q_2\in \mathcal U_{q_0}^{t_1,t_2}$ for some $t_1,t_2$, then ${\rm supp}\, (q_1-q_2)\subset\Omega\times [t_1,t_2]$.
Now, we state the second theorem that gives answers to {\bf (IP-2.1)} and {\bf (IP-2.2)}.
\begin{theorem}\label{thm:deter-potential-coeffi}
Suppose that for each $j=1,2$, $u_j$ is the solution to the equation
\begin{equation}
\begin{cases}
(I-\gamma \Delta)\partial_t^2u_j+ \Delta^2 u_j+q_ju_j=0 & {\rm in}\,\, \Omega_T,\\
u_j=h_1,\, \Delta u_j=h_2 & {\rm on}\,\, \Gamma_T,\\
u_j(0)=\eta_{1j},\,\partial_tu_j(0)=\eta_{2j} & {\rm in}\,\, \Omega,
\end{cases}
\end{equation}
with respect to the coefficient $q_j$ and the initial data $(\eta_{1j},\eta_{2j})$.
Let $T_0>0$ be as in \eqref{observe-time-T0}, and let the function spaces $\mathcal H_\Omega^0,\mathcal H_{\Gamma,\Omega}^T,\mathcal H_\Gamma^T$ be given by \eqref{def:H1-H2}, \eqref{func-space-initial-boundary-comp} and \eqref{function-space-initial-boundary}, respectively.  The following three assertions hold.

(1) Assume that $T>0$, the initial data $\eta_{1j}=\eta_1,\eta_{2j}=\eta_2$ are known functions, and the coefficients $q_j\in L^\infty(0,T;L^{p_0}(\Omega))$ with $p_0$ satisfying \eqref{cond-for-tilde-q} for $j=1,2$. If 
\begin{equation}
\Lambda^{\Gamma,T}_{q_1}(h_1,h_2,\eta_1,\eta_2)=\Lambda^{\Gamma,T}_{q_2}(h_1,h_2,\eta_1,\eta_2),\quad \text{for all}\,\, (h_1,h_2,\eta_1,\eta_2)\in \mathcal H^T_{\Gamma,\Omega},
\end{equation}
then $q_1=q_2$ in $\Omega_T$.

(2) Assume that $T>t_2>t_1>T_0$, the initial data $(\eta_{1j},\eta_{2j})\in\mathcal H_\Omega^0$, and the coefficients $q_j\in \mathcal U^{t_1,t_2}_{q_0}$ for $j=1,2$.
If 
\begin{equation}
\Lambda^{\Gamma,T}_{q_1,\eta_{11},\eta_{21}}(h_1,h_2)=\Lambda^{\Gamma,T}_{q_2,\eta_{12},\eta_{22}}(h_1,h_2),\quad \text{for all}\,\, (h_1,h_2)\in \mathcal H_\Gamma^T,
\end{equation}
then $\eta_{11}=\eta_{12}$, $\eta_{21}=\eta_{22}$ in $\Omega$, and $q_1=q_2$ in $\Omega_T$.

(3) Assume that $T-T_0>t_2>t_1>T_0$, the initial data $(\eta_{1j},\eta_{2j})\in\mathcal H_\Omega^0$, and the coefficients $q_j\in \mathcal U^{t_1,t_2}_{q_0}$ for $j=1,2$.
If 
\begin{equation}
\tilde \Lambda^{\Gamma,T}_{q_1,\eta_{11},\eta_{21}}(h_1,h_2)=\tilde \Lambda^{\Gamma,T}_{q_2,\eta_{12},\eta_{22}}(h_1,h_2),\quad \text{for all}\,\, (h_1,h_2)\in \mathcal H_\Gamma^T,
\end{equation}
then $\eta_{11}=\eta_{12}$, $\eta_{21}=\eta_{22}$ in $\Omega$, and $q_1=q_2$ in $\Omega_T$.
\end{theorem}

We now present the third theorem for (\textbf{IP-3}), which deals with the simultaneous recovery of initial data and nonlinearities $g$ from the active measurement $\Lambda^{\Gamma,T}_{g,\eta_1,\eta_2}$.
\begin{theorem}\label{thm:deter-ini-da-nonlin}
Let the dimension $n\in\{1,2,3\}$, and let $T>t_2>t_1>T_0$, where $T_0$ is given by \eqref{observe-time-T0}. Assume that $g_j\in \mathcal G_{g_0}^{t_1,t_2}$ for $j=1,2$. Then there is a small constant $\epsilon>0$, such that for any initial data $(\eta_{1j},\eta_{2j})\in\mathcal S_\epsilon(\Omega)$ $(j=1,2)$, if 
\begin{equation}
\Lambda^{\Gamma,T}_{g_1,\eta_{11},\eta_{21}}(h_1,h_2)=\Lambda^{\Gamma,T}_{g_2,\eta_{12},\eta_{22}}(h_1,h_2),\quad \text{for all}\,\, (h_1,h_2)\in\mathcal S_\epsilon(\Gamma_T),
\end{equation}
then it holds that $g_1=g_2$ in $\Omega_T\times\mathbb C$ and $(\eta_{11},\eta_{21})=(\eta_{12},\eta_{22})$ in $\Omega$.
\end{theorem}
\begin{remark}
We note that, the smallness conditions for the nonlinear Kirchhoff plate equation \eqref{eq:intro-non-lin-plate-g} is needed only for the local well-posedness, but not in the study of the associated inverse problem {\bf (IP-3)}.

Let $T>2T_0$ and let $T_0<t_1<t_2<T-T_0$. Similar to the third assertion (3) in Theorem \ref{thm:deter-potential-coeffi},  we define the following (DN-type) map
\begin{equation}
\tilde \Lambda^{\Gamma,T}_{g,\eta_1,\eta_2}(h_1,h_2)=\bigl(\partial_\nu u|_{\Gamma_T}, \partial_\nu\Delta u|_{\Gamma_T}\bigr),\quad (h_1,h_2)\in \mathcal S_\epsilon(\Gamma_T),
\end{equation}
and assume that $g_j\in \mathcal G_{g_0}^{t_1,t_2}$ for $j=1,2$. We can similarly prove that, if 
$$\tilde \Lambda^{\Gamma,T}_{g_1,\eta_{11},\eta_{21}}(h_1,h_2)=\tilde \Lambda^{\Gamma,T}_{g_2,\eta_{12},\eta_{22}}(h_1,h_2)$$ for all $(h_1,h_2)\in\mathcal S_\epsilon(\Gamma_T)$, then $g_1=g_2$ in $\Omega_T\times\mathbb C$ and $(\eta_{11},\eta_{21})=(\eta_{12},\eta_{22})$ in $\Omega$.
\end{remark}

\subsection{Connections to existing studies and novelties}
Inverse problems, which aim to recover unknown parameters (such as coefficients or initial data) and the location or shape of objects from indirect measurements, constitute a central area of research in partial differential equations (PDEs). A vast body of literature has addressed inverse problems for a variety of PDEs, including wave equations, heat equations, and elliptic equations.  Among these works, an inverse boundary problem was studied for nonlinear parabolic equations in \cite{isakov1993uniqueness}, in which the first-order linearization of the Dirichlet-to-Neumann (DN) map coincides with the DN map of the linearized equation. The linearization and the generalized higher-order linearization approach have been widely applied to solve inverse problems for various kinds of nonlinear PDEs, see an incomplete list \cite{isakov1994global,kurylev2018inverse,carstea2021calderon,lin2022simultaneous,feizmohammadi2023inverse,LaiSIMA2024,UhlmannMA24,lassas2025APDE,fu2026calderon} and the references therein.


Compared to inverse second-order PDEs that have been extensively studied, the literature on fourth-order plate equations, especially for those in time-domain, remains comparatively limited. The higher order of spatial derivatives adds to the complexity of the analysis of plate equations. Moreover, they do not possess apparent geometric structures, in contrast to the principal part of wave equations, for which the Lorentz metric provides a natural geometric framework. Existing work has primarily addressed the stable recovery of coefficients, sources, initial data, and unique continuation appearing in time-dependent beam and plate equations \cite{EllerLasieckaTriggiani2001,yuan2007lipschitz,wang2007global,osses2013potential,hasanov2021unique,yu2022carleman,fu2023stability} by primarily using Carleman estimates together with the method initiated by Bukhgeim and Klibanov \cite{bukhgeim1981global}. We also mention the work \cite{doi:10.1137/080725635} that introduced an abstract framework based on exact observability for Euler-Bernoulli plate equations. Under clamped boundary conditions, a global Carleman estimate was established in \cite{imba2025carleman} without reducing the Euler-Bernoulli plate equation to a system of two coupled Schrödinger equations. This finally yields Lipschitz stability in the recovery of the linear coefficient. We also refer the interested readers to \cite{bao2026uniqueness,bhattacharyya2025inverse,chang2024analysis,LiIP2021,li2024inverse} for inverse problems for linear and nonlinear biharmonic operators, which can be regarded as the steady-state counterparts corresponding to the plate equations in time domains. However, there are few results available in literature for inverse (nonlinear) Kirchhoff plate equations.

In recent years, simultaneously determining an unknown source and its surrounding medium from passive measurements has attracted considerable attention in the literature, owing to its practical relevance in emerging applications.
This is exemplified by inverse problems in photo-acoustic and thermo-acoustic tomography \cite{liu2015determining,feizmohammadi2025reconstruction,kian2025uniqueness,kian2025determination,suhonen2025reconstructing}. The practical importance of such problems has stimulated extensive research across diverse fields, including brain imaging \cite{deng2019inverse}, geomagnetic anomaly detection, quantum mechanics \cite{deng2019identifying,deng2020identifying,li2019determining,li2021determining}, and inverse obstacle scattering problems
\cite{liu2026inverse}. Carleman-based approaches require prior knowledge of either the source or the coefficient, and
existing Fourier-based simultaneous recovery results discard high-frequency information. An effective strategy for simultaneous recovery is to combine passive and active measurements, an approach that has been successfully applied to several inverse problems. Nevertheless, inverse problems that rely solely on passive data present distinct and substantial theoretical challenges; they remain relatively underexplored and represent an emerging and timely frontier in the field.

Most recently, the work \cite{liu2025determining} studied inverse boundary problems for evolutionary PDEs using only a single passive boundary observation, providing a systematic resolution for a
broad class of such inverse problems for second-order hyperbolic, parabolic, and Schr\"odinger equations within a single coherent approach. We also mention the work \cite{Lin-JLMS}, which addressed some inverse boundary problems associated with a time-dependent semilinear hyperbolic equation, where both nonlinearity and sources (including initial displacement and initial velocity) are unknown. The observability inequality for wave equations with potentials has been utilized for the stable recovery of initial values in nonlinear wave equations with known nonlinearity. Furthermore, by employing higher-order linearization around non-zero solutions, suitable geometric optics (GO) solutions, and a Runge approximation for linear wave equations, the authors demonstrated that in various generic settings, the nonlinearity and/or sources are uniquely recoverable from passive or active boundary measurements. The coefficients considered in this paper are bounded.
For the inverse (time domain) Euler-Bernoulli plate equation $\rho(x)\partial_t^2u+\Delta^2u=0$ with initial data $(u(0),\partial_tu(0))=(f(x),g(x))$, under some a priori assumptions on $\rho,f,g$, the work \cite{gao2023inverse} established the unique identifying results in simultaneously determining both the unknown density and the internal sources from the passive boundary measurement. For Euler-Bernoulli plate equations with variable coefficients (on Riemannian manifolds) and nonlinearities, the first three authors of this paper studied in \cite{fu2026stablyrecovering} the inverse problems of stably determining initial data, coefficients by a single pair passive measurement, under the hinged and clamped boundary conditions, respectively,

These related works on inverse plate equations have mainly focused on recovering bounded or time-independent coefficients, and some rather restrictive a priori conditions on initial data and coefficients are required. 
To the authors’ knowledge, inverse problems of simultaneously recovering coefficients and initial data for Kirchhoff plate equations remain unexplored, and even the recovery of linear coefficients has not been studied yet. In this paper, we aim to investigate such inverse problems for nonlinear Kirchhoff plate equations. The main findings are as follows: 

(1) With a wide range of nonlinearities, we establish global well-posedness for semilinear Kirchhoff plate equations and prove that, when the nonlinearities are known a priori, the initial data can be stably recovered from a single pair of passive measurement. Some counterexamples are given to show that, if the nonlinearity (or external source) is unknown, then the passive measurement is insufficient to recover the initial data. Furthermore, unknown coefficients render passive measurement incapable of recovering the initial values, even if the external source is known, thereby precluding the simultaneous recovery of both the sources and coefficients unless some \emph{a priori} or technical assumptions are imposed.

(2) The coefficients are time-dependent and not required to be bounded. We study the simultaneous recovery of initial data and coefficients using two kinds of active measurements, both without requiring the extra measurement $\partial_tu|_{t=T}$.

(3) By applying the well-posedness results for linear equations, we establish a refined Runge approximation in a smaller function space and as an alternative way, we introduce a simple cut-off technique to address the issue of vanishing initial data arising in the application of GO solutions.

The rest of the paper is organized  as follows. In Section \ref{sec:well-posedness}, we establish the well-posedness of linear Kirchhoff plate equations, which then serves as the foundation for the global and local well-posedness results for  nonlinear equations with certain nonlinearities. Section \ref{Construction-GO-solution} is dedicated to constructing GO solutions for linear Kirchhoff plate equations with potentials. The proofs of the main theorems are provided in Section \ref{sec:proof-thms}. Finally, Section \ref{sec:conclusions} contains some concluding remarks along with a precise description of a simple cut-off technique.

\section{Well-posedness of the forward Kirchhoff plate equation}\label{sec:well-posedness}
In this section, we are devoted to studying the well-posedness  of the nonlinear plate equation \eqref{eq:intro-non-lin-plate}. Let us begin with introducing some function spaces and a useful lemma that will be frequently used in our study.

\subsection{Preliminaries}\label{subsec:well-posed-prelimi}
Let $X$ be some given Banach space with norm $\|\cdot\|_X$, and let $C^k([a,b];X)$ and $L^p(a,b;X)$ ($k\in\mathbb N$, $1\le p\le\infty$ and $(a,b)\subset\mathbb R$) stand for the space of $k$-times continuously differentiable functions and the space of measurable functions $\varphi:(a,b)\to X$ such that $\|\varphi(t)\|_X\in L^p(a,b)$. The norms are given by
\begin{equation}\label{def:norm}
\|\varphi\|_{L^p(a,b);X)}\vcentcolon=\left(\int_a^b\|\varphi(t)\|_X^pdt\right)^{\frac 1p}<\infty,
\end{equation}
\begin{equation}
\|\varphi\|_{C^k([a,b];X)}\vcentcolon=\max_{0\le l\le k}\max_{t\in [a,b]}\|\partial_t^l\varphi\|_{X}=\max_{0\le l\le k}\|\partial_t^l\varphi\|_{L^\infty(a,b;X)}.
\end{equation}
The function space $H^m_0(0,T;X)$ is defined by 
\begin{equation}\label{set-function-space-H0m-X}
\begin{split}
H^m_0(0,T;X)\vcentcolon=&\{\varphi\in L^2(0,T;X): \partial_t^k\varphi\in L^2(0,T;X)\,\, \text{for}\,\, k=0,1,\cdots,m,\\
&\quad \partial_t^j\varphi(\cdot,l)=0\,\, \text{for}\,\, j=0,1,\cdots,m-1\,\, \text{and}\,\, l=0,T\}.
\end{split}
\end{equation}

We follow the same lines of \cite[Section 2.2]{lin2024well} to introduce the following contents. 

Whenever $\varphi\in L^1_{\rm loc}(a,b;X)$ with $X$ being a space of functions over a subset of some euclidean space, such as $L^2(\Omega)$ or $H^s(\Omega)$ for $s\in\mathbb R$, then $\varphi$ is identified with a function $\varphi(x,t)$, and $\varphi(t)$ denotes the function $x\mapsto \varphi(x,t)$ for almost all $t\in (a,b)$. This is justified by
the fact that any $\varphi\in L^r(0,T;L^s(\Omega))$ with $1\le r,s\le \infty$ can be seen as a measurable function $\varphi:\Omega\times (a,b)\to \mathbb R$, such that the norm $\|\varphi\|_{L^r(a,b;L^s(\Omega))}$, as defined in \eqref{def:norm}, is finite. Particularly, we write $L^p(a,b;L^p(\Omega))=L^p(\Omega\times (a,b))$ for $1\le p\le \infty$.

Let us also define two function spaces that will be served as solution spaces for the plate equation. For any given $T>0$ and $m\in\mathbb N_{\ge 1}$, we define 
\begin{equation}\label{def:function-space-Em-T}
E_m^T\vcentcolon=\bigcap_{k=0}^1C^k([0,T];H^{m-k}(\Omega)),
\end{equation}
with the equivalent norm
\begin{equation}
\|\varphi\|_{E^T_m}\vcentcolon=\left( \max_{t\in [0,T]}\sum_{k=0}^1\|\partial_t^k\varphi\|^2_{H^{m-k}(\Omega)} \right)^{\frac12}.
\end{equation}



The following lemma offers two useful inequalities that play an important role in analyzing the well-posedness of \eqref{eq:nonlin-plate-homo-boundary} and its linearized problem.
\begin{lemma}\label{lem:ineq-qu-H-1L2}
The following assertions hold.

(1) Suppose that $u\in H^2(\Omega)$, and  $q\in L^p(\Omega)$ with $p$ satisfying the conditions 
\begin{equation}\label{cond-for-q}
\begin{cases}
p\ge \frac n3, & {\rm if}\,\, n>4,\\
p>\frac 43, & {\rm if}\,\, n=4,\\
p\ge \frac 65, & {\rm if}\,\, n=3,\\
p>1, & {\rm if}\,\, n=2,\\
p\ge 1, & {\rm if}\,\, n=1.
\end{cases}
\end{equation}
Then
\begin{equation}\label{ineq:qu-H1}
\|q\varphi\|_{H^{-1}(\Omega)}\lesssim \|q\|_{L^p(\Omega)}\|u\|_{H^2(\Omega)}.
\end{equation}
Moreover, if $q\in L^{\mathfrak m}(\Omega)$ with $\mathfrak m$ satisfying \eqref{cond-for-m}, then
\begin{equation}\label{ineq:qu-L2-H2}
\|qu\|_{L^2(\Omega)}\lesssim \|q\|_{L^{\mathfrak m}(\Omega)}\|u\|_{H^2(\Omega)}.
\end{equation}

(2) Suppose that $u\in H^3(\Omega)$, and $q\in L^{p_0}(\Omega)$ with $p_0$ satisfying the conditions
\begin{equation}\label{cond-for-tilde-q}
\begin{cases}
p_0\ge \frac n3, & {\rm if}\,\, n>6,\\
p_0>2, & {\rm if}\,\, n=6,\\
p_0\ge 2, & {\rm if}\,\, 1\le n<6,
\end{cases}
\end{equation}
Then
\begin{equation}\label{ineq:qu-H2}
\|qu\|_{L^2(\Omega)}\lesssim \|q\|_{L^{p_0}(\Omega)}\|u\|_{H^3(\Omega)}.
\end{equation}
\end{lemma}
\begin{proof}
For the range $n>4$, by the Sobolev embedding theorem that $H^2(\Omega)\hookrightarrow L^d(\Omega)$ for $2\le d\le \frac{2n}{n-4}$, and H\"older's inequality, we can obtain
\begin{equation}\label{est:qu-H-1}
\begin{split}
\|qu\|_{H^{-1}(\Omega)}&=\sup_{\varphi\in H_0^1(\Omega),\|\varphi\|_{H^1(\Omega)}=1}\big|\langle qu,\varphi\rangle_{H^{-1}(\Omega),H_0^1(\Omega)}\big|\\
&\le \sup_{\varphi\in H_0^1(\Omega),\|\varphi\|_{H^1(\Omega)}=1}\|\varphi\|_{L^{\frac{2n}{n-2}}(\Omega)}\|qu\|_{L^{\frac{2n}{n+2}}(\Omega)}\\
&\le \|q\|_{L^{\frac n3}(\Omega)}\|u\|_{L^{\frac{2n}{n-4}}(\Omega)}\lesssim \|q\|_{L^{\frac n3}(\Omega)}\|u\|_{H^2(\Omega)}
\end{split}
\end{equation}
For $ n=4$, since $H^2(\Omega)\hookrightarrow L^d(\Omega)$ for any $2\le d<\infty$, we have
\begin{equation}
\|qu\|_{L^{\frac{2n}{n+2}}(\Omega)}\le \|q\|_{L^{\frac 43+\varepsilon}(\Omega)} \|u\|_{L^{d_\varepsilon}(\Omega)}\lesssim \|q\|_{L^{\frac 43+\varepsilon}(\Omega)} \|u\|_{H^2(\Omega)},\,\, d_\varepsilon=\frac{16}{9\varepsilon}+\frac43,\, \forall \varepsilon>0.
\end{equation}
For $1\le n<4$, we have $H^2(\Omega)\hookrightarrow L^\infty(\Omega)$. If $2<n<4$ (i.e., $n=3$), then using \eqref{est:qu-H-1}, we have
\begin{equation}
\|qu\|_{H^{-1}(\Omega)}\le \|qu\|_{L^{\frac 65}(\Omega)}\le \|q\|_{L^{\frac 65}(\Omega)}\|u\|_{L^\infty(\Omega)}\lesssim \|q\|_{L^{\frac 65}(\Omega)}\|u\|_{H^2(\Omega)}.
\end{equation}
If $n=2$, then $H^1(\Omega)\hookrightarrow L^d(\Omega)$ for $2\le d<\infty$. Thus, using the boundedness of $\Omega$ and H\"older's inequality, we have
\begin{equation}
\begin{split}
\|qu\|_{H^{-1}(\Omega)}&=\sup_{\varphi\in H_0^1(\Omega),\|\varphi\|_{H^1(\Omega)}=1}\big|\langle qu,\varphi\rangle_{H^{-1}(\Omega),H_0^1(\Omega)}\big|\le \|\varphi\|_{L^{\frac{1+\varepsilon}{\varepsilon}}(\Omega)}\|qu\|_{L^{1+\varepsilon}(\Omega)}\\
&\le \|\varphi\|_{H^1(\Omega)}\|q\|_{L^{1+\varepsilon}(\Omega)}\|u\|_{L^\infty(\Omega)}\lesssim \|q\|_{L^{1+\varepsilon}(\Omega)}\|u\|_{H^2(\Omega)},\quad \forall \varepsilon>0.
\end{split}
\end{equation}
If $n=1$, then $H^1(\Omega)\hookrightarrow L^\infty(\Omega)$. We have
\begin{equation}
\|qu\|_{H^{-1}(\Omega)}\lesssim \|q\|_{L^1(\Omega)}\|u\|_{H^2(\Omega)}.
\end{equation}
Hence, collecting all cases, we conclude that
\begin{equation}
\|qu\|_{H^{-1}(\Omega)}\lesssim  \|q\|_{L^p(\Omega)}\|u\|_{H^2(\Omega)},
\end{equation}
where the exponent $p$ satisfies conditions \eqref{cond-for-q}, yielding  the inequality \eqref{ineq:qu-H1}.

Now we turn our attention to proving \eqref{ineq:qu-H2}; the remaining inequality \eqref{ineq:qu-L2-H2} can be proved in a similar way.  If $u\in H^3(\Omega)$ and $n>6$, then the Sobolev embedding theorem that $H^3(\Omega)\hookrightarrow L^d(\Omega)$ for $2\le d\le \frac{2n}{n-6}$, implies 
\begin{equation}
\|qu\|_{L^2(\Omega)}\le \|q\|_{L^{\frac n3}(\Omega)}\|u\|_{L^{\frac{2n}{n-6}}(\Omega)}\lesssim\|q\|_{L^{\frac n3}(\Omega)}\|u\|_{H^3(\Omega)}.
\end{equation}
If $n=6$, then again by the Sobolev embedding theorem $H^3(\Omega)\hookrightarrow L^d(\Omega)$ for any $2\le d<\infty$ implies 
\begin{equation}
\|qu\|_{L^2(\Omega)}\le \|q\|_{L^{2+\varepsilon}(\Omega)}\|u\|_{L^{\frac{2(2+\varepsilon)}{\varepsilon}}(\Omega)}\lesssim\|q\|_{L^{2+\varepsilon}(\Omega)}\|u\|_{H^3(\Omega)},\quad \forall\varepsilon>0.
\end{equation}
The subcritical case $1\le n<6$ is clear since $H^3(\Omega)\hookrightarrow L^\infty(\Omega)$. Hence we can conclude that
\begin{equation}
\|qu\|_{L^2(\Omega)}\lesssim \|q\|_{L^{p_0}(\Omega)}\|u\|_{H^3(\Omega)},
\end{equation}
where $p_0$ satisfies \eqref{cond-for-tilde-q}. Therefore, the proof of Lemma \ref{lem:ineq-qu-H-1L2} is complete.
\end{proof}
\subsection{Well-posedness of linear Kirchhoff plate equations}\label{subsec:well-posedness-lin}
In order to establish well-posedness result of the nonlinear equation \eqref{eq:intro-non-lin-plate}, we first study the well-posedness for the following linearized plate equation:
\begin{equation}\label{eq:lin-u-plate}
\begin{cases}
(I-\gamma \Delta)\partial_t^2u+ \Delta^2 u+qu=h & {\rm in}\,\, \Omega_T,\\
u=h_1,\,\Delta u=h_2 & {\rm on}\,\, \Gamma_T,\\
u(0)=\eta_1,\,\partial_tu(0)=\eta_2 & {\rm in}\,\, \Omega.
\end{cases}
\end{equation}
The well-posedness of the linear equation \eqref{eq:lin-u-plate} is stated as follows. 
\begin{theorem}\label{proposition-well-posedness-lin-u}
Let $T>0$ be given. Suppose that $q\in L^\infty(0,T;L^p(\Omega))$ with $p$ satisfying conditions \eqref{cond-for-q}, the external source term $h\in L^2(0,T;H^{-1}(\Omega))$, and the initial boundary data satisfy 
\begin{equation}\label{cond1-for-h1-h2-h}
(h,h_1,h_2)\in L^2(0,T;H^{-1}(\Omega))\times H^2(0,T;H^2(\Gamma))\times H^2(0,T;L^2(\Gamma)),
\end{equation}
and 
\begin{equation}\label{cond1-for-u0-u1}
(\eta_1,\eta_2)\in H^2(\Omega)\times H^1(\Omega),\quad \eta_1|_\Gamma=h_1(0),\, \eta_2|_\Gamma=\partial_th_1(0).
\end{equation}
Then the equation \eqref{eq:lin-u-plate} admits a unique solution $u\in E_2^T\cap H^2(0,T;L^2(\Omega))$ with $\partial_\nu\partial_tu|_{\Gamma_T}\in L^2(\Gamma_T)$ and $\partial_\nu\Delta u|_{\Gamma_T}\in H^{-2}(0,T;H^{-2}(\Gamma))$. Moreover, there is a positive constant $C$ depending on $\Omega,\gamma,n,p$ and the norm $\|q\|_{L^\infty(0,T;L^p(\Omega))}$, such that for any $t\in [0,T]$, it holds that
\begin{equation}\label{est:u-H2-H1}
\begin{split}
&\|u(t)\|_{H^2(\Omega)}+\|\partial_tu(t)\|_{H^1(\Omega)}+\|\partial_\nu\partial_tu\|_{L^2(\Gamma_T)}\\
&\le Ce^{CT}\bigl( \|h_1\|_{H^2(0,T;H^2(\Gamma))}+\|h_2\|_{H^2(0,T;L^2(\Gamma))}\\
&\quad\quad +\|h\|_{L^2(0,T;H^{-1}(\Omega))}+\|\eta_1\|_{H^2(\Omega)}+\|\eta_2\|_{H^1(\Omega)}\bigr).
\end{split}
\end{equation}
On the other hand, assume that $q\in L^\infty(0,T;L^{p_0}(\Omega))$ with $p_0$ satisfying \eqref{cond-for-tilde-q}, $h\in L^2(\Omega_T)$, and the initial boundary data
\begin{equation}\label{cond2-h-h1-h2-u0-u1}
(h_1,h_2,\eta_1,\eta_2)\in H^2(0,T;H^3(\Gamma))\times H^2(0,T;H^1(\Gamma))\times H^3(\Omega)\times H^2(\Omega)
\end{equation}
satisfy the compatibility conditions
\begin{equation}\label{compatibility-initial-lin}
\eta_1=h_1(0),\,\ \Delta \eta_1=h_2(0),\, \eta_2=\partial_th_1(0)\,\, {\rm on}\,\, \Gamma.
\end{equation}
Then the equation \eqref{eq:lin-u-plate} admits a unique solution $u\in E^T_3\cap H^2(0,T;H^1(\Omega))$ with $\partial_\nu\Delta u|_{\Gamma_T}\in L^2(\Gamma_T)$, such that for any $t\in [0,T]$, it holds that
\begin{equation}\label{est:lin-u-H3}
\begin{split}
&\|u(t)\|_{H^3(\Omega)}+\|\partial_tu(t)\|_{H^2(\Omega)}+\|\partial_\nu u\|_{H^1(0,T;L^2(\Gamma))}+\|\partial_\nu\Delta u\|_{L^2(\Gamma_T)}\\
&\le Ce^{CT}\bigl( \|h_1\|_{H^2(0,T;H^3(\Gamma))}+\|h_2\|_{H^2(0,T;H^1(\Gamma))}\\
&\quad\quad +\|h\|_{L^2(\Omega_T)}+\|\eta_1\|_{H^3(\Omega)}+\|\eta_2\|_{H^2(\Omega)}\bigr),
\end{split}
\end{equation}
where the positive constant $C$ depends on $\Omega,\gamma,n,p_0$ and the norm $\|q\|_{L^\infty(0,T;L^{p_0}(\Omega))}$.
\end{theorem}
\begin{proof}
We shall divide the proof into five steps.

{\bf Step 1.} For the moment, we assume that the coefficient $q=0$ in \eqref{eq:lin-u-plate}. For a.e. $t\in [0,T]$, we introduce the Green map that is defined by
\begin{equation}\label{eq:auxi-G}
\begin{cases}
\Delta^2[G(h_1,h_2)]=0 & {\rm in}\, \Omega,\\
G(h_1,h_2)=h_1,\, \Delta G(h_1,h_2)=h_2 & {\rm on}\, \Gamma.
\end{cases}
\end{equation}
By the classical elliptic theory, $G:H^{s+2}(\Gamma)\times H^s(\Gamma)\to H^{s+\frac 52}(\Omega)$ is well-defined for any $s\in\mathbb R$. Since we extend time-dependent boundary data $h_1,h_2$ to the interior of $\Omega$, we apply the Green map $G$ pointwise almost everywhere in time and still denote the resulting operator by $G$, i.e., $(G(h_1,h_2))(t)\vcentcolon=G(h_1(t),h_2(t))$. We note that due to the linearity of $G$, it holds that $\partial_t^j(G(h_1,h_2))(t)=G(\partial_t^jh_1(t),\partial_t^jh_2(t))$ for $j=1,2$. 
Since $h_1\in H^2(0,T;H^k(\Gamma))$, $h_2\in H^2(0,T;H^{k-2}(\Gamma))$ ($k=2,3$), we can get $G(h_1,h_2)\in H^2(0,T;H^{\frac12+k}(\Omega))\subset H^2(0,T;H^{k}(\Omega))$ for $k=2,3$ (e.g., see \cite[Chapter 2]{lions2012nonI}), and $h_1(x,\cdot),h_2(x,\cdot)\in C^1[0,T]$ for a.e. $x\in\Omega$. Let $w=u-G(h_1,h_2)$. We find that
\begin{equation}\label{eq:lin-w-plate}
\begin{cases}
(I-\gamma \Delta)\partial_t^2w+ \Delta^2 w=\tilde h & {\rm in}\,\, \Omega_T,\\
w=\Delta w=0 & {\rm on}\,\, \Gamma_T,\\
w(0)=w_0,\,\partial_tw(0)=w_1 & {\rm in}\,\, \Omega,
\end{cases}
\end{equation}
where $\tilde h=h-(I-\gamma\Delta)\partial_t^2[G(h_1,h_2)]\in L^2(0,T;H^{j}(\Omega))$ for $h\in L^2(0,T;H^j(\Omega))$ with $j=-1,0$, and
\begin{equation}\label{initial-w0-w-1}
w_0=\eta_1-G(h_1(0),h_2(0))\in H^k(\Omega),\, w_1=\eta_2-G(h_{1t}(0),h_{2t}(0))\in H^{k-1}(\Omega)
\end{equation}
for $(\eta_1,\eta_2)\in H^k(\Omega)\times H^{k-1}(\Omega)$ with $k=2,3$.
The compatibility conditions \eqref{cond1-for-u0-u1} and \eqref{compatibility-initial-lin} for $\eta_1,\eta_2$ and the definition of $G(h_1,h_2)$, yield
$$w_0|_\Gamma=\eta_1|_\Gamma-G(h_1(0),h_2(0))|_\Gamma=0,\, w_1|_\Gamma=\eta_2|_\Gamma-G(h_{1t}(0),h_{2t}(0))|_\Gamma=0$$ for $(\eta_1,\eta_2)\in H^k(\Omega)\times H^{k-1}(\Omega)$ with $k=2,3$, and
$$\Delta w_0|_\Gamma=\Delta \eta_1|_\Gamma-\Delta [G(h_1(0),h_2(0))]|_\Gamma=0$$
for $(\eta_1,\eta_2)\in H^3(\Omega)\times H^2(\Omega)$.
Recalling the function spaces $\mathcal H_1$ and $\mathcal H_2$ given in \eqref{def:H1-H2}, we have
$(w_1,w_2)\in\mathcal H_1\times\mathcal H_2$ for $(\eta_1,\eta_2)\in H^3(\Omega)\times H^2(\Omega)$, and $(w_1,w_2)\in \mathcal H_2\times H_0^1(\Omega)$ for $(\eta_1,\eta_2)\in H^2(\Omega)\times H^1(\Omega)$, respectively.

{\bf Step 2.} We next introduce some operators related to the linear problem \eqref{eq:lin-w-plate}. We shall first define an operator $A$ on $L^2(\Omega)$ by 
\begin{equation}
A\psi=\Delta^2\psi,\quad D(A)=\{\psi\in H^4(\Omega): \psi|_\Gamma=\Delta\psi|_\Gamma=0\}.
\end{equation}
By \cite{grisvard1967caracterisation} (see also \cite{LASIECKA199162}), we have
\begin{equation}
A^{\frac 12}\psi=-\Delta\psi,\quad  \psi\in D(A^{\frac 12})=H^2(\Omega)\cap H_0^1(\Omega)=\mathcal H_2,
\end{equation}
and the following fractional identifications (with equivalent norms) of $A$:
\begin{equation}
\begin{split}
D(A^\theta)&=\{\psi\in H^{4\theta}: \psi|_\Gamma=0\},\quad \frac 18<\theta<\frac 58,\\
D(A^\theta)&=\{\psi\in H^{4\theta}: \psi|_\Gamma=\Delta\psi|_\Gamma=0\},\quad \frac 58<\theta\le 1.
\end{split}
\end{equation}
In particular, we have
\begin{equation}
D(A^{\frac 14})=H_0^1(\Omega),\quad D(A^{\frac 12})=\mathcal H_2,\quad D(A^{\frac 34})=\mathcal H_1,
\end{equation}
with equivalent norms
\begin{equation}
\|\psi\|_{D(A^{\frac 14})}=\|A^{\frac14}\psi\|_{L^2(\Omega)}=\|\nabla\psi\|_{L^2(\Omega)}\sim \|\psi\|_{L^2(\Omega)}+\gamma\|\nabla\psi\|_{L^2(\Omega)}\sim\|\psi\|_{H^1(\Omega)},
\end{equation}
\begin{equation}\label{equi-norm-H2}
\|\psi\|_{D(A^{\frac 12})}=\|A^{\frac 12}\psi\|_{L^2(\Omega)}=\|\Delta\psi\|_{L^2(\Omega)}\sim\|\psi\|_{H^2(\Omega)},
\end{equation}
and 
\begin{equation}\label{equi-norms-H2-H3}
\|\psi\|_{D(A^{\frac 34})}=\|A^{\frac 34}\psi\|_{L^2(\Omega)}=\|A^{\frac14}\Delta\psi\|_{L^2(\Omega)}=\|\nabla\Delta\psi\|_{L^2(\Omega)}\sim \|\psi\|_{H^3(\Omega)}.
\end{equation}
Denote by $Z_1\vcentcolon=D(A^{\frac12})\times D(A^{\frac14})=\mathcal H_2\times H_0^1(\Omega)$ and $Z_2\vcentcolon=D(A^{\frac 34})\times D(A^{\frac 12})=\mathcal H_1\times \mathcal H_2$. Using the Green formula, we have
\begin{equation}
(A\psi_1,\psi_2)_{L^2(\Omega)}=(\Delta\psi_1,\Delta\psi_2)_{L^2(\Omega)},\quad {\rm for}\,\, \psi_1,\psi_2\in D(A).
\end{equation}
By the Hahn-Banach theorem, the operator $A$ can be continuously extended from $D(A^{\frac 12})=\mathcal H_2$ to $H^{-2}(\Omega)$. We define
\begin{equation}
\langle A\psi_1,\psi_2\rangle_{(D(A^{\frac12}))', D(A^{\frac12})}=(A^{\frac 12}\psi_1,A^{\frac 12}\psi_2)_{L^2(\Omega)},\quad {\rm for}\,\, \psi_1,\psi_2\in D(A^{\frac12})=\mathcal H_2.
\end{equation}
Let $P_\gamma\vcentcolon=I+\gamma A^{\frac 12}$. The $H_0^1(\Omega)$-elliptic of $P_\gamma$ and Lax-Milgram gives that $P_\gamma$ is bounded invertible from $H^{-1}(\Omega)$ to $H_0^1(\Omega)$. Finally $P_\gamma$ is positive definite, self-adjoint, and  well-defined with $D(P_\gamma^{\frac12})=D(A^{\frac14})=H_0^1(\Omega)$,
and for $\phi_1,\phi_2\in H_0^1(\Omega)$,
\begin{equation}
\begin{split}
\langle(I+\gamma A^{\frac12})\phi_1,\phi_2\rangle_{H^{-1}(\Omega),H_0^1(\Omega)}&=(P_\gamma^\frac12\phi_1,P_\gamma^{\frac12}\phi_2)_{L^2(\Omega)}\\
&=(\phi_1,\phi_2)_{L^2(\Omega)}+\gamma(\nabla\phi_1,\nabla\phi_2)_{L^2(\Omega)}.
\end{split}
\end{equation}
The inner product on $Z_1$ is given by
\begin{equation}
((\psi_1,\phi_1),(\psi_2,\phi_2))_{Z_1}\vcentcolon=(A^{\frac 12}\psi,A^{\frac 12}\psi_2)_{L^2(\Omega)}+(P_\gamma^\frac12\phi_1,P_\gamma^{\frac12}\phi_2)_{L^2(\Omega)}.
\end{equation}

With the above preparations, we introduce operator
\begin{equation}
\mathcal A\vcentcolon=\begin{pmatrix}
0 & I \\
-P_\gamma^{-1} A & 0
\end{pmatrix},\quad\quad Z_1\supset D(\mathcal A)\to Z_1,
\end{equation}
with domain $$D(\mathcal A)=\{(\psi,\phi)\in Z_1: \mathcal A(\psi,\phi)\in Z_1\}=\{(\psi,\phi)\in \mathcal H_1\times \mathcal H_2\}.$$
Indeed, by the definition of $D(\mathcal A)$, we have $P_\gamma^{-1}A\psi\in H_0^1(\Omega)$ and $\phi\in\mathcal H_2$. Let $P_\gamma^{-1}A\psi=\tilde\psi\in H_0^1(\Omega)$. We find that
\begin{equation}
\begin{cases}
\Delta^2\psi=(I-\gamma\Delta)\tilde\psi\in H^{-1}(\Omega) & {\rm in}\, \Omega,\\
\psi=\Delta\psi=0 & {\rm on}\, \Gamma.
\end{cases}
\end{equation}
It follows from the elliptic theory that $\psi\in H^3(\Omega)$ (i.e., $\psi\in\mathcal H_2$).

Let $W(t)=(w(t),\partial_tw(t))^{\rm T}$, $F(t)=(0,P_\gamma^{-1}\tilde h(t))^{\rm T}$, where $\cdot^{\rm T}$ means the transpose of a vector $\cdot$. Therefore, we can re-write the linear problem \eqref{eq:lin-w-plate} as 
\begin{equation}\label{sys:abstract-operator}
\frac{d}{dt}W(t)=\mathcal AW(t)+F(t),\quad W(0)=(w_0,w_1)^{\rm T}\in\mathcal H_2.
\end{equation}
We note that the detailed analysis of the operator $\mathcal A$ and the associated abstract system \eqref{sys:abstract-operator} can be found in \cite{avalos1998exponential}. Observing that, if $\tilde h\in L^1(0,T;H^{-1}(\Omega))$, then we have $\tilde f\vcentcolon=P_\gamma^{-1}\tilde h\in L^1(0,T;H_0^1(\Omega))$ by solving the following elliptic problem
\begin{equation}
-\gamma\Delta\tilde f+\tilde f=\tilde h\in H^{-1}(\Omega),\quad \tilde f=0\,\, {\rm on}\,\, \Gamma.
\end{equation}
Hence, by the classical semigroup theory, $\mathcal A$ generates a strong ${\rm C}_0$-semigroup $e^{t\mathcal A}$ on $Z_1$ if $(w_0,w_1)\in Z_1$ and $F\in L^1(0,T;Z_1)$ (i.e., $\tilde h\in L^1(0,T;H^{-1}(\Omega))$), yielding that
$(w,\partial_tw)\in C([0,T];Z_1)$ is the unique weak solution to the problem \eqref{eq:lin-w-plate}.
Thus, we have $u=w+G(h_1,h_2)\in E^T_2=C([0,T];H^2(\Omega))\cap C^1([0,T];H^1(\Omega))$.
Thanks to \cite[Theorem 1.3]{Irena-99-DCDS}, the boundary trace $\partial_\nu\partial_tw|_{\Gamma_T}\in L^2(\Gamma_T)$ continuously depends on $(w_0,w_1)\in Z_1$ and $\tilde h\in L^1(0,T;H^{-1}(\Omega))$. Hence, we have $\partial_\nu\partial_tu|_{\Gamma_T}=\partial_\nu\partial_tw|_{\Gamma_T}+\partial_\nu[G(h_{1t},h_{2t})]|_{\Gamma_T}\in L^2(\Gamma_T)$ with the estimate
\begin{equation}
\begin{split}
\|\partial_\nu u\|_{H^1(0,T;L^2(\Gamma))}&\le C\bigl( \|h\|_{L^2(0,T;H^{-1}(\Omega))}+\|\eta_1\|_{H^2(\Omega)}+\|\eta_2\|_{H^1(\Omega)}\\
&\quad\quad +\|h_1\|_{H^2(0,T;H^2(\Gamma))}+\|h_2\|_{H^2(0,T;L^2(\Gamma))} \bigr).
\end{split}
\end{equation}

Now using the equation \eqref{eq:lin-w-plate} in the distribution sense and its homogeneous boundary conditions, for any $\varphi\in D(A^{\frac12})=\mathcal H_2$, and for a.e. $t\in [0,T]$, we have 
\begin{equation}
\langle (I-\gamma\Delta)\partial_t^2w,\varphi\rangle =\langle \tilde h-\Delta^2w,\varphi\rangle\lesssim \bigl(\|\tilde h\|_{H^{-1}(\Omega)}+\|w\|_{H^2(\Omega)}\bigr) \|\varphi\|_{H^2(\Omega)}.
\end{equation}
Thus, similar to \cite[Remark 8.2, Chapter 3]{lions2012nonI}, we have $(I-\gamma\Delta)\partial_t^2w\in H^{-2}(\Omega)$, implying that $\partial_t^2u=\partial_t^2w+G(h_{1tt},h_{2tt})\in L^2(\Omega_T)$ by elliptic theory. 

Furthermore, if $(w_0,w_1)\in Z_2=D(\mathcal A)$ and $\tilde h\in L^2(\Omega_T)$, then $F\in L^2(0,T;Z_2)$, yielding that $(w,\partial_tw)\in C([0,T];Z_2)$. Again using the relation that $u=w+G(h_1,h_2)$, we can deduce that $u\in E^T_3=C([0,T];H^3(\Omega))\cap C^1([0,T];H^2(\Omega))$ with $\partial_t^2u\in L^2(0,T;H^1(\Omega))$ and $\Delta\partial_t^2u\in L^2(0,T;H^{-1}(\Omega))$. Moreover, the boundary trace $\partial_\nu\partial_tu\in L^2(\Gamma_T)$ is trivial since $\partial_tu\in C([0,T];H^2(\Omega))$.

{\bf Step 3.} We next prove the estimate \eqref{est:lin-u-H3}. 
Denote by
\begin{equation}
E_1(t)\vcentcolon=\frac12\int_\Omega \bigl(|\partial_tu|^2+|\Delta u|^2+\gamma|\nabla\partial_tu|^2\bigr)dx,\quad t\in [0,T].
\end{equation}
By density, it suffices to prove energy estimates for sufficiently regular solutions (see for example \cite[Section 8, Chapter 3]{lions2012nonI}). We multiply the first equation of \eqref{eq:lin-u-plate} by $\partial_tu$ and integrate in $\Omega\times (0,t)$ to have the following energy identity
\begin{equation}\label{iden:ener-H2-H1}
\begin{split}
E_1(t)&=\gamma\int_0^t\int_\Gamma\partial_tu(\partial_\nu\partial_t^2u) d\Gamma dl+\int_0^t\int_\Gamma\bigl((\partial_\nu\partial_tu)\Delta u-\partial_tu(\partial_\nu\Delta u) \bigr)d\Gamma dl\\
&\quad +\int_0^t\int_\Omega h\partial_tu\, dxdl+\frac12\bigl( \|\eta_2\|^2_{L^2(\Omega)}+\gamma\|\nabla \eta_2\|^2_{L^2(\Omega)}+\|\Delta \eta_1\|^2_{L^2(\Omega)}\bigr).
\end{split}
\end{equation}
Hence, for the problem \eqref{eq:lin-w-plate}, recalling that $w=\Delta w=0$ on $\Gamma$ and $\tilde h=h-(I-\gamma\Delta)\partial_t^2[G(h_1,h_2)]\in L^2(0,T;H^{-1}(\Omega))$ (provided that $h\in L^2(0,T;H^{-1}(\Omega))$, and $(h_1,h_2)\in H^2(0,T;H^2(\Gamma))\times H^2(0,T;L^2(\Gamma))$), we can get
\begin{equation}
\begin{split}
\tilde E_1(t)\vcentcolon&=\frac12\int_\Omega \bigl(|\partial_tw|^2+|\Delta w|^2+\gamma|\nabla\partial_tw|^2\bigr)dx\\
&\le \|\tilde h\|^2_{L^2(0,T;H^{-1}(\Omega))}+\int_0^t\|\partial_tw\|^2_{H^1(\Omega)}dx+C\bigl(\|w_0\|_{H^2(\Omega)}^2+\|w_1\|_{H^1(\Omega)} \bigr)\\
&\le C\bigl(\|h\|^2_{L^2(0,T;H^{-1}(\Omega))}+\|h_1\|^2_{H^2(0,T;H^2(\Gamma))}+\|h_2\|^2_{H^2(0,T;L^2(\Gamma))}\bigr)\\
&\quad +\int_0^t\tilde E_1(l)dl+C\bigl(\|w_0\|_{H^2(\Omega)}^2+\|w_1\|^2_{H^1(\Omega)} \bigr).
\end{split}
\end{equation}
Applying the Gronwall's inequality and using \eqref{initial-w0-w-1}, we can deduce that
\begin{equation}\label{est:E1t}
\begin{split}
\tilde E_1(t)&\le Ce^{CT}\bigl(\|h\|^2_{L^2(0,T;H^{-1}(\Omega))}+ \|h_1\|^2_{H^2(0,T;H^2(\Gamma))}+\|h_2\|^2_{H^2(0,T;L^2(\Gamma))} \\
&\quad\qquad\quad+\|\eta_1\|_{H^2(\Omega)}^2+\|\eta_2\|^2_{H^1(\Omega)}\bigr),\quad \forall t\in [0,T].
\end{split}
\end{equation}
Moreover, we have $\partial_\nu\partial_tu|_{\Gamma_T}=\partial_\nu\partial_tw|_{\Gamma_T}+\partial_\nu[G(h_{1t},h_{2t})]|_{\Gamma_T}\in L^2(\Gamma_T)$.

Recalling that for $w\in C([0,T;H^2(\Omega)\cap H_0^1(\Omega))$, as we have discussed in the second step, we have the equivalent norms 
$$\|w(t)\|_{D(A^{\frac12})}\sim \|\Delta w(t)\|_{L^2(\Omega)}\sim \|w(t)\|_{H^2(\Omega)},$$ and $$\|\partial_tw\|_{D(A^{\frac14})}\sim\|\nabla\partial_tw\|_{L^2(\Omega)}\sim\|\partial_tu(t)\|_{H^1(\Omega)}.$$ 
Thus, using  the relation $u=w+G(h_1,h_2)$, we can get
\begin{equation}\label{ineq:u-equi}
\begin{split}
&\|u(t)\|^2_{H^2(\Omega)}+\|\partial_tu(t)\|^2_{H^1(\Omega)}+\|\partial_tu(t)\|^2_{L^2(\Omega)}\\
&\le C \bigl( \tilde E_1(t)+\|G(h_1,h_2)\|^2_{H^2(\Omega)}+\|\partial_t(G(h_1,h_2))\|^2_{H^1(\Omega)}\bigr)\\
&\le C\bigl( \tilde E_1(t)+\|h_1(t)\|^2_{H^2(\Gamma)}+\|h_2(t)\|^2_{L^2(\Gamma)}+\|h_{1t}(t)\|^2_{H^2(\Gamma)}+\|h_{2t}(t)\|^2_{L^2(\Gamma)}\bigr),
\end{split}
\end{equation}
where we have used the continuity  of the Green map $G:H^{s+2}(\Gamma)\times H^s(\Gamma)\to H^{s+\frac52}(\Omega)$ for any $s\in\mathbb R$.
Taking the supremum of the above inequality \eqref{ineq:u-equi} with respect to $t\in [0,T]$, and using the embedding $h_k(t)\in H^2(0,T)\hookrightarrow C^1([0,T])$ for $k=1,2$, we can obtain
\begin{equation}\label{est:u-E2-middle}
\begin{split}
\|u\|^2_{E_2^T}&=\max_{t\in [0,T]}\bigl( \|u(t)\|^2_{H^2(\Omega)}+\|\partial_tu\|^2_{H^1(\Omega)} \bigr)\\
&\le Ce^{CT}\bigl (\|\eta_1\|^2_{H^2(\Omega)}+\|\eta_2\|^2_{H^1(\Omega)}+\|h\|^2_{L^2(0,T;H^{-1}(\Omega))}\\
&\quad \quad \qquad +\|h_1\|^2_{H^2(0,T;H^2(\Gamma))}+\|h_2\|^2_{H^2(0,T;L^2(\Gamma))} \bigr).
\end{split}
\end{equation}
Here the positive constant $C$ only depends on $\Omega,\gamma,n$.

We next show that, under the regularity of $u\in E^T_2$, the boundary term $\partial_\nu\Delta u|_{\Gamma_T}$ makes sense, at least  in the space $H^{-2}(0,T;H^{-2}(\Gamma))$. To this aim, let $\varphi$ satisfy the backward equation
\begin{equation}
\begin{cases}
(I-\gamma\Delta)\partial_t^2\varphi+\Delta^2\varphi=0 & {\rm in}\, \Omega_T,\\
\varphi=g,\, \Delta\varphi=0 & {\rm on}\, \Gamma_T,\\
\varphi(T)=\varphi_t(T)=0 & {\rm in}\, \Omega.
\end{cases}
\end{equation}
By time reversal $t\mapsto T-t$, the above equation admits a unique weak solution $\varphi\in E^T_2\cap H^2(0,T;L^2(\Omega))$ provided that $g\in H^2(0,T;H^2(\Gamma))$. Moreover,
\begin{equation}\label{est:auxi-backward-var}
\max_{t\in [0,T]}\bigl(\|\varphi\|_{H^2(\Omega)}+\|\partial_t\varphi\|_{H^1(\Omega)}\bigr)+\|\partial_\nu\partial_t\varphi\|_{L^2(\Gamma_T)}\le C\|g\|_{H^2(0,T;H^2(\Gamma))}.
\end{equation}
We invoke the equation \eqref{eq:lin-w-plate} and use integration by parts to have
\begin{equation}
\begin{split}
\int_0^T\int_\Gamma \varphi\partial_\nu\Delta w\,d\Gamma dt&=\int_\Omega[w_1\varphi(0)-w_0\partial_t\varphi(0)+\gamma w_0\Delta\partial_t\varphi(0)-\gamma\varphi(0)\Delta w_1]\, dx\\
&\quad -\gamma\int_0^T\int_\Gamma \partial_t\varphi\partial_\nu\partial_tw\, d\Gamma dt+\int_0^T\int_\Omega\tilde h\varphi\, dxdt,
\end{split}
\end{equation}
where the boundary conditions $w=\Delta w=\Delta\varphi=0$ on $\Gamma_T$ were used.
Hence, applying the estimate \eqref{est:auxi-backward-var}, we can obtain
\begin{equation}
\begin{split}
&\Big| \int_0^T\int_\Gamma \varphi\partial_\nu\Delta w\,d\Gamma dt \Big|\\
&\le C\bigl(\|w_0\|_{H^2(\Omega)}\|+\|w_1\|_{H^1(\Omega)} \bigr)\bigl( \|\varphi(0)\|_{H^2(\Omega)}+\|\partial_t\varphi(0)\|_{H^1(\Omega)} \bigr)\\
&\quad+C\bigl( \|\partial_\nu\varphi\|_{L^2(\Gamma_T)}\|\partial_\nu\partial_\nu w\|_{L^2(\Gamma_T)}+\|\tilde h\|_{L^2(0,T;H^{-1}(\Omega))}\|\varphi\|_{L^2(0,T;H^1(\Omega))}\bigr) \\
&\le C\bigl( \|w_0\|_{H^2(\Omega)}\|+\|w_1\|_{H^1(\Omega)}+ \|\tilde h\|_{L^2(0,T;H^{-1}(\Omega))}\bigr)\|g\|_{H^2(0,T;H^2(\Gamma))}.
\end{split}
\end{equation}
By taking in particular $g\in H_0^2(0,T;H^2(\Gamma))$ (see the definition \eqref{set-function-space-H0m-X}), we can derive that $\partial_\nu\Delta w|_{\Gamma_T}\in H^{-2}(0,T;H^{-2}(\Gamma))$.

For any given $\tilde h_1\in H^2(\Gamma)\subset H^{\frac32}(\Gamma)$, we extend $\tilde h_1$ from $\Gamma$ to $\Omega$ by introducing the following auxiliary equation
\begin{equation}
\begin{cases}
\Delta^2\xi=0 & {\rm in}\,\Omega,\\
\xi=\tilde h_1,\, \Delta\xi=0 & {\rm on}\, \Gamma.
\end{cases}
\end{equation}
Recalling the equation \eqref{eq:auxi-G} for $G(h_1,h_2)$, using Green's formula, we can get
\begin{equation}
\int_\Gamma\xi\partial_\nu[\Delta G(h_1,h_2)]d\Gamma=\int_\Gamma h_2\partial_\nu\xi\, d\Gamma-\int_\Omega \Delta G(h_1,h_2)\Delta\xi\, dx.
\end{equation}
Hence, it follows from (bi-harmonic) elliptic theory (e.g., see \cite[Section 2.4.3]{GGS10}) that
\begin{equation}
\begin{split}
\Big|\int_\Gamma \tilde h_1\partial_\nu[\Delta G(h_1,h_2)]d\Gamma\Big|&\le \|h_2\|_{L^2(\Gamma)}\|\partial_\nu\xi\|_{L^2(\Gamma)}+\|G(h_1,h_2)\|_{H^2(\Omega)}\|\xi\|_{H^2(\Omega)}\\
&\le C\bigl( \|h_1\|_{H^2(\Gamma)}+\|h_2\|_{L^2(\Gamma)} \bigr)\|\tilde h_1\|_{H^{\frac32}(\Gamma)},
\end{split}
\end{equation}
implying that $\partial_\nu[\Delta G(h_1,h_2)]|_{\Gamma}\in H^{-\frac32}(\Gamma)\subset H^{-2}(\Gamma)$.
Therefore, we can conclude that $\partial_\nu\Delta u|_{\Gamma_T}=\partial_\nu\Delta w|_{\Gamma_T}+\partial_\nu[\Delta G(h_1,h_2)]|_{\Gamma_T}\in H^{-2}(0,T;H^{-2}(\Gamma))$.

{\bf Step 4.} We proceed to prove the estimate \eqref{est:lin-u-H3}. 
We define a functional
\begin{equation}
E_2(t)\vcentcolon=\frac12\int_\Omega \bigl(|\partial_tu|^2+|\Delta u|^2+\gamma|\nabla\partial_tu|^2+\gamma|\Delta\partial_tu|^2+|\nabla\Delta u|^2\bigr)dx,\quad t\in [0,T].
\end{equation}
Multiplying equation \eqref{eq:lin-u-plate} by $\Delta\partial_tu$, and integrating over $\Omega\times (0,t)$, we get
\begin{equation}\label{est:energy-H3-first}
\begin{split}
&\frac12\int_\Omega\bigl(|\nabla\partial_tu|^2+\gamma|\Delta\partial_tu|^2+|\nabla\Delta u|^2 \bigr)dx\\
&=\int_0^t\int_\Gamma\bigl(h_{1tt}\partial_\nu\partial_tu+h_{2t}\partial_\nu\Delta u\bigr)d\Gamma dl+\int_0^t\int_\Omega h\Delta\partial_tu\,dxdl\\
&\quad +\frac12\bigl(\|\nabla \eta_2\|^2_{L^2(\Omega)} +\gamma\|\Delta \eta_2\|^2_{L^2(\Omega)}+\|\nabla\Delta \eta_1\|^2_{L^2(\Omega)} \bigr).
\end{split}
\end{equation}
Let $H\in C^2(\overline\Omega;\mathbb R^n)$ be a vector field on $\overline\Omega$. Using the multiplier $H(\Delta u)=\langle H,\nabla\Delta u\rangle$, and by \cite[Appendix A, (A.5)]{LASIECKA199162}, we can obtain
\begin{equation}\label{est:boundary-traces}
\begin{split}
&\int_0^t\int_\Gamma \bigl[(\partial_\nu\Delta u)H(\Delta u)+(\partial_\nu\partial_tu) H(\partial_tu)+(\partial_\nu\partial_tu)(\partial_tu){\rm div} H\\
&\quad +\frac12 \bigl(\gamma|\Delta\partial_tu|^2-|\nabla\Delta u|^2-|\nabla\partial_tu|^2-2(\partial_tu)\Delta\partial_tu \bigr)  \langle H,\nu\rangle\bigr]\,d\Gamma dl\\
&=\int_0^t\int_\Omega\bigl[DH(\nabla\Delta u,\nabla\Delta u)+DH(\nabla\partial_tu,\nabla\partial_tu)\\
&\quad+\frac12\bigl(|\nabla\partial_tu|^2+\gamma|\Delta\partial_tu|^2 -|\nabla\Delta u|^2\bigr){\rm div}H +\langle\nabla({\rm div}H),\nabla\partial_tu\rangle(\partial_tu)\\
&\quad +hH(\Delta u) \bigr]dxdl-\int_\Omega\bigl[\partial_tuH(\Delta u)+\gamma(\Delta\partial_tu)H(\Delta u) \bigr]dx\Big|_0^t,
\end{split}
\end{equation}
where $DH$ is a tensor field of order 2, denoting the covariant derivative of $H$.

Let $H=\nu$ on $\Gamma$. Notice that 
\begin{equation}
\nabla u=(\partial_\nu u)\nu+(\tau u)\tau,\quad \nabla\Delta u=(\partial_\nu\Delta u)\nu+(\tau\Delta u)\tau \quad {\rm on}\,\,\Gamma,
\end{equation}
where $\tau$ is the tangent vector field along $\Gamma$ and $\tau$ is perpendicular to $\nu$.
Recalling that $(u,\Delta u)|_{\Gamma_T}=(h_1,h_2)$, it then follows from \eqref{est:boundary-traces} and Young's inequality that
\begin{equation}\label{est:boundary-trace-first}
\begin{split}
&\frac12\int_0^t\int_\Gamma\bigl[|\partial_\nu\Delta u|^2+|\partial_\nu\partial_tu|^2 +\gamma|\Delta\partial_tu|^2 \bigr]d\Gamma dl\\
&\le \|h_2\|^2_{H^1(0,T;H^1(\Omega))}+\bigl( \max_{\overline\Omega}|{\rm div}H|^2+1\bigr)\|h_1\|^2_{H^1(0,T;H^1(\Gamma))}\\
&\quad +\frac18\int_0^t\int_\Gamma |\partial_\nu\partial_tu|^2\, d\Gamma dl+C\int_0^t E_2(l)dl\\
&\quad +\frac12 \int_\Omega \bigl[\bigl (2\max_{\overline\Omega}|H|\bigr) |\nabla\Delta u|^2 +|\partial_tu|^2+\gamma^2|\Delta\partial_tu|^2  \bigr]dx\\
&\quad +C\bigl(\|\eta_1\|^2_{H^3(\Omega)}+\|\eta_2\|^2_{H^2(\Omega)} +\|h\|^2_{L^2(\Omega_T)} \bigr).
\end{split}
\end{equation}
Let $\alpha\vcentcolon=\max_{\overline\Omega}|H|$ and $\beta\vcentcolon=\alpha+\frac\gamma 4$.
Multiplying $4\beta$ in both sides of \eqref{est:energy-H3-first}, together with \eqref{est:boundary-trace-first}, and using Young's inequality,  we can get
\begin{equation}
\begin{split}
&\int_\Omega\bigl[\alpha|\nabla\Delta u|^2+ 2\alpha\gamma|\Delta\partial_tu|^2+2\beta|\nabla\partial_tu|^2\bigr]dx\\
&\quad +\frac18\int_0^t\int_\Gamma\bigl(2|\partial_\nu\Delta u|^2+|\partial_\nu\partial_tu|^2  \bigr)d\Gamma dl\\
&\le C\int_0^tE_2(l)dl+C\bigl(\|\eta_1\|^2_{H^3(\Omega)}+\|\eta_2\|^2_{H^2(\Omega)}\bigr)\\
&\quad +C\bigl(\|h\|^2_{L^2(\Omega_T)}+ \|h_1\|^2_{H^2(0,T;H^1(\Gamma))}+\|h_2\|^2_{H^1(0,T;H^1(\Gamma))}\bigr),
\end{split}
\end{equation}
where the positive constant $C$ only depends on $\Omega,\gamma,n$ and $H$. Using the estimate \eqref{est:E1t} and applying Gronwall's  inequality, we derive that
\begin{equation}
\begin{split}
&E_2(t)+\int_0^t\int_\Gamma\bigl( |\partial_\nu\Delta u|^2+|\partial_\nu\partial_tu|^2  \bigr)d\Gamma dl\le Ce^{CT}\bigl( \|\eta_1\|^2_{H^3(\Omega)}+\|\eta_2\|^2_{H^2(\Omega)}\\
&+ \|h\|^2_{L^2(\Omega_T)}+ \|h_1\|^2_{H^2(0,T;H^3(\Gamma))}+\|h_2\|^2_{H^1(0,T;H^1(\Gamma))} \bigr),\quad \forall t\in [0,T].
\end{split}
\end{equation}
Here the positive constant $C$ depends on $\Omega,\gamma,n$ and $H$. Since $H$ is an auxiliary vector field on $\overline\Omega$, we will say $C$ depends on $\Omega,\gamma,n$.
Using the equivalent norms of $(w,\partial_tw)\in C([0,T];Z_2)$ on $H^3(\Omega)\times H^2(\Omega)$, and by a similar proof to the estimate \eqref{est:u-E2-middle}, we can finally obtain
\begin{equation}
\begin{split}
&\|u\|_{E^T_3}^2+\|\partial_\nu\Delta u\|^2_{L^2(\Gamma_T)}+\|\partial_\nu\partial_tu\|^2_{L^2(\Gamma_T)}\\
&\le Ce^{CT}\bigl( \|\eta_1\|^2_{H^3(\Omega)}+\|\eta_2\|^2_{H^2(\Omega)}\\
&+ \|h\|^2_{L^2(\Omega_T)}+ \|h_1\|^2_{H^2(0,T;H^3(\Gamma))}+\|h_2\|^2_{H^1(0,T;H^1(\Gamma))} \bigr).
\end{split}
\end{equation}
Hence, the estimate \eqref{est:lin-u-H3} holds.

{\bf Step 5.} We recall that $q=0$ was assumed in the first four steps above. With the non-zero lower-order perturbation $qu$ now included, we prove that the well-posedness result in Theorem \ref{proposition-well-posedness-lin-u} holds. We only prove the first assertion that $u\in E^T_2$ and the another one ($u\in E^T_3)$ can be proved along the same line.

Let $v$ satisfy the following equation:
\begin{equation}\label{eq:lin-v-auxi}
\begin{cases}
(I-\gamma\Delta)\partial_t^2v+\Delta^2v=h & {\rm in}\, \Omega_T,\\
v=h_1,\, \Delta v=h_2 & {\rm on}\, \Gamma_T,\\
v(0)=\eta_1,\, \partial_tv(0)=\eta_2 & {\rm in}\,\Omega.
\end{cases}
\end{equation}
We have proved that $v\in E_2^T$ provided that $(h,h_1,h_2)$ and $(\eta_1,\eta_2)$ satisfy the conditions \eqref{cond1-for-h1-h2-h} and \eqref{cond1-for-u0-u1}, respectively.
Given $v\in E_2^T$ by the equation \eqref{eq:lin-v-auxi}, we define, for any given $\tilde w\in E_2^T$, a map
\begin{equation}
\mathcal K: E_2^T\to E_2^T,\quad \tilde w\mapsto w,
\end{equation}
where $w$ satisfy the problem
\begin{equation}\label{eq:lin-v+w-auxi}
\begin{cases}
(I-\gamma\Delta)\partial_t^2w+\Delta^2w=-q(v+\tilde w) & {\rm in}\, \Omega_T,\\
w=\Delta w=0 & {\rm on}\, \Gamma_T,\\
w(0)=\partial_tw(0)=0 & {\rm in}\,\Omega.
\end{cases}
\end{equation}
Since $v,\tilde w\in E_2^T$, it follows from the estimate \eqref{ineq:qu-H1} in Lemma \ref{lem:ineq-qu-H-1L2} that $q(v+\tilde w)\in L^2(0,T;H^{-1}(\Omega))$. Hence, the linear problem \eqref{eq:lin-v+w-auxi} admits a unique solution $w\in E_2$ implying that the map $\mathcal K$
is well-defined. Now let $\tilde w_1,\tilde w_2\in E_2^T$, and let $w_1,w_2$ be the unique solution of \eqref{eq:lin-v+w-auxi} with $\tilde w$ replaced by $\tilde w_1$ and $\tilde w_2$, respectively. Using the estimates \eqref{est:E1t} and \eqref{ineq:qu-H1}, and the linearity of \eqref{eq:lin-v+w-auxi}, we can get
\begin{equation}
\begin{split}
\|\mathcal K\tilde w_1-\mathcal K\tilde w_2\|_{E^2}&=\|w_1-w_2\|_{E_2^T}\le Ce^{CT}\|q(\tilde w_1-\tilde w_2)\|_{L^2(0,T;H^{-1}(\Omega))}\\
&\le CTe^{CT}\|q\|_{L^\infty(0,T;L^p(\Omega))}\|\tilde w_1-\tilde w_2\|_{L^\infty(0,T;H^2(\Omega))}\\
&\le CTe^{CT}\|q\|_{L^\infty(0,T;L^p(\Omega))}\|\tilde w_1-\tilde w_2\|_{E_2^T}.
\end{split}
\end{equation}
If $T$ is sufficiently small such that $CTe^{CT}\|q\|_{L^\infty(0,T;L^p(\Omega))}<1$, then by the Banach fixed point theorem, 
$\mathcal K$ has a unique fixed point $\tilde w=\mathcal K\tilde w=w$. Thus, $u=v+w\in E_2^T$ is the unique solution of \eqref{eq:lin-u-plate}. Since \eqref{eq:lin-u-plate} is a linear equation, by a rescaling with respect to the time variable, we can get the well-posedness result for any $T>0$. Now, replacing $h$ by $h-qu$ and repeating the arguments in {\bf Step 3}, and using the estimate \eqref{est:u-E2-middle}, we can derive that
\begin{equation}
\begin{split}
&\|u(t)\|_{H^2(\Omega)}^2+\|\partial_tu(t)\|^2_{H^1(\Omega)}\\
&\le C\|h-qu\|^2_{L^2(0,T;H^{-1}(\Omega))}+C\bigl( \|\eta_1\|^2_{H^2(\Omega)}+\|\eta_2\|^2_{H^1(\Omega)} \\
&\quad +\|h_1\|^2_{H^2(0,T;H^2(\Gamma))}+\|h_2\|^2_{H^2(0,T;L^2(\Gamma))} \bigr)\\
&\le C\|h\|^2_{L^2(0,T;H^{-1}(\Omega))}+C\|q\|^2_{L^\infty(0,T;L^p(\Omega))}\int_0^t\|u(l)\|^2_{H^2(\Omega)}dl\\
&\le C\bigl(\|h\|^2_{L^2(0,T;H^{-1}(\Omega))}+ \|\eta_1\|^2_{H^2(\Omega)}+\|\eta_2\|^2_{H^1(\Omega)}+\|h_1\|^2_{H^2(0,T;H^2(\Gamma))}\\
&\quad +\|h_2\|^2_{H^2(0,T;L^2(\Gamma))} \bigr)+C\|q\|^2_{L^\infty(0,T;L^p(\Omega))}\int_0^t\|u(l)\|^2_{H^2(\Omega)}dl.
\end{split}
\end{equation}
Again applying Gronwall's inequality, the estimate \eqref{est:u-H2-H1} holds. Clearly the constant in \eqref{est:u-H2-H1} depends on $\Omega,\gamma,n,p$ and $\|q\|_{L^\infty(0,T;L^p(\Omega))}$. Therefore, the first assertion of Theorem \ref{proposition-well-posedness-lin-u} is obtained. 

With the linear term $qu$ included, using the estimate \eqref{ineq:qu-H2}, the second well-posedness result concerning the $H^3(\Omega)\times H^2(\Omega)$-regularity and estimate \eqref{est:lin-u-H3}, can be proved via a similar manner. We thus omit the details, and complete the proof of Theorem \ref{proposition-well-posedness-lin-u}. 
\end{proof}
\begin{remark}
It is clear that Theorem \ref{proposition-well-posedness-lin-u} remains valid for complex-valued $u$ since the equation \eqref{eq:lin-u-plate} is linear and we can write $u=\Re u+{\rm i}\Im u$ (resp. for other related functions $h,h_1,h_2,\eta_2,u_2$).

We also emphasis that, in \cite{Irena-99-DCDS}, the authors proved that $u\in E_2^T$ provided that the boundary data $u|_{\Gamma_T}=h_1=0$ and $\Delta u|_{\Gamma_T}=h_2\in L^2(\Gamma_T)$. In our setting, for the purpose of cancelling the non-homogeneous hinged boundary data, we assumed that $(h_1,h_2)\in H^2(0,T;H^k(\Gamma))\times H^2(0,T;H^{k-2}(\Gamma))$ ($k=2$ or $k=3$) and applied the Green map. The regularity of the boundary data (at least $h_2$) is not sharp compared to that of \cite{Irena-99-DCDS}. We believe that the regularity condition of $(h_1,h_2)$ can be weakened by utilizing the transport method introduced explicitly by Lions and Magenes \cite{lions2012non}. Nevertheless, we refrain from further discussion of this topic here, because in our inverse problems the boundary data acts as a system input and can be constructed with higher regularity.
\end{remark}
\begin{remark}
We have shown that $u\in C([0,T];H^2(\Omega))$ under conditions \eqref{cond1-for-h1-h2-h}--\eqref{cond1-for-u0-u1},  and the boundary term $\partial_\nu\Delta u|_{\Gamma_T}$ belongs at least to $H^{-2}(0,T;H^{-2}(\Gamma))$, which is not very sharp, since $\Delta u\in L^2(\Omega)$ and one could expect $\partial_\nu\Delta u|_{\Gamma}\in H^{-\frac32}(\Gamma)$. However, such rough boundary regularity result is enough for our purposes.
\end{remark}
\begin{remark}\label{rem:plate-wave-gamma}
We observe that the higher-order term $\gamma\partial_t^2\Delta u$ regularizes the Kirchhoff plate equation and thereby enhances the regularity of solutions. It has been discussed in \cite[Appendix C]{LASIECKA199162} that the abstract form of the homogeneous Kirchhoff plate equation can be formally rewritten as
\begin{equation}
\partial_t^2u=\gamma^{-1}\Delta u+\gamma^{-2}u-(I+\gamma A^{\frac12})^{-1}u+(I+\gamma A^{\frac12})^{-1}f.
\end{equation}
This highlights the hyperbolic character and dispersive behavior of the Kirchhoff plate equation, where the speed of propagation is $\gamma^{-\frac12}$, constituting a significant difference between it and the classical Euler-Bernoulli plate whose principal part is $\partial_t^2+\Delta^2$.
\end{remark}

\subsection{Global well-posedness of nonlinear plate equations}
In this subsection, we study the well-posedness for the nonlinear plate equation \eqref{eq:intro-non-lin-plate} with certain nonlinearities $f(x,u)$. The global well-posedness theory for PDEs, such as semilinear nonlocal wave equations (with fractional Laplacian), was established by Lin, Tyni and Zimmermann in \cite{lin2024well}, motivating us to investigate the global well-posedness of nonlinear Kirchhoff plate equations. 
We next prove a useful inequality that will be frequently used in dealing with the nonlinearity $f$. 
\begin{lemma}\label{lem:est-f-L2-H2}
Let $T>0$ be given. Assume that the nonlinearity $f\in\mathcal M^T_{r,\mathfrak m}$ with $r$ and $\mathfrak m$ satisfying
\eqref{cond-for-r} and \eqref{cond-for-m}, respectively. Then for any $\varphi\in L^\infty(0,T;H^2(\Omega))$, there is a positive constant $C$, which is independent of $T$, such that
\begin{equation}
\begin{split}
\|f(x,\varphi)\|_{L^2(\Omega_T)}&\le CT^{\frac12}\bigl(\|f(\cdot,0)\|_{L^2(\Omega)}\\
&\quad + \|a\|_{L^{\mathfrak m}(\Omega)}\|\varphi\|_{L^\infty(0,T;H^2(\Omega))} +\|\varphi\|^{r+1}_{L^\infty(0,T;H^2(\Omega))} \bigr).
\end{split}
\end{equation}
\end{lemma}
\begin{proof}
Recalling the assumptions {\bf (A.1)} and {\bf (A.2)}, we have that $f(x,\varphi(x,t))$ is measurable for $t\in [0,T]$, and
\begin{equation}
|f(x,s)|\le \Big|\int_0^s\partial_\tau f(x,\tau)\, d\tau \Big|+|f(x,0)|\le C\bigl( |a(x)||s|+|s|^{r+1}\bigr)+|f(x,0)|,
\end{equation}
for a.e. $(x,t)\in\Omega_T$ and $s\in \mathbb R$. This implies 
\begin{equation}
|f(x,\varphi(x,t))|\le C\bigl(|f(x,0)|+|a(x)||\varphi(x,t)|+|\varphi(x,t)|^{r+1} \bigr).
\end{equation}
Hence, applying the inequality \eqref{ineq:qu-L2-H2}, we can get
\begin{equation}
\begin{split}
\|f(x,\varphi)\|_{L^2(\Omega_T)}&\le CT^{\frac12}\bigl( \|f(\cdot,0)\|_{L^2(\Omega)}\\
&\quad+\|a\|_{L^{\mathfrak m}(\Omega)}\|\varphi\|_{L^\infty(0,T;H^2(\Omega))}\bigr)+C\||\varphi|^{r+1}\|_{L^2(\Omega_T)}.
\end{split}
\end{equation}
For the range $n>4$, the conditions \eqref{cond-for-r} yield $2\le 2(r+1)\le \frac{2n}{n-4}$. Applying the Sobolev embedding $H^2(\Omega)\hookrightarrow L^d(\Omega)$ for any $2\le d\le \frac{2n}{n-4}$, we have
\begin{equation}
\||\varphi|^{r+1}\|_{L^2(\Omega_T)}\le T^{\frac12}\|\varphi\|^{r+1}_{L^\infty(0,T;L^{2(r+1))}(\Omega)}\le CT^{\frac12}\|\varphi\|^{r+1}_{L^\infty(0,T;H^2(\Omega))}.
\end{equation}
The range $1\le n\le 4$ is simple, since $r$ satisfies $0\le r<\infty$, and we have the embedding $H^2(\Omega)\hookrightarrow L^d(\Omega)$ for any $2\le d<\infty$. Thus, the proof is finished.
\end{proof}

We are now in a position to state the (global) well-posedness of the nonlinear plate equation \eqref{eq:nonlin-plate-homo-boundary} with homogeneous boundary data $u|_{\Gamma_T}=\Delta u|_{\Gamma_T}=0$.
\begin{theorem}[Global well-posedness of the nonlinear equation \eqref{eq:nonlin-plate-homo-boundary}]\label{thm:global-well-posed-H3}
Let $T>0$ be given. Assume that the nonlinearity $f\in\mathcal M^T_{r,\mathfrak m}$ with $r$ and $\mathfrak m$ satisfying
\eqref{cond-for-r} and \eqref{cond-for-m}, respectively, and the initial data $(\eta_1,\eta_2)\in \mathcal H_1\times\mathcal H_2$ (see \eqref{def:H1-H2}). Then the nonlinear Kirchhoff plate equation \eqref{eq:nonlin-plate-homo-boundary} admits a unique solution $u\in E^T_3$ with boundary traces $\partial_\nu\partial_tu|_{\Gamma_T},\partial_\nu\Delta u|_{\Gamma_T}\in L^2(\Gamma_T)$.
\end{theorem}
\begin{proof}
We divide the proof into three steps.

{\bf Step 1.} Let us set
\begin{equation}
A\vcentcolon=\max\big\{\|\eta_1\|_{H^2(\Omega)},\, \|\eta_2\|_{H^1(\Omega)} \big\},
\end{equation}
and $E_2^{T_0}\vcentcolon=C([0,T_0];H^2(\Omega))\cap C^1([0,T_0];H^1(\Omega))$ for some $0<T_0\le T$ that will be specified later. 
We define the function space
\begin{equation}
X_{T_0,R}\vcentcolon=\{u\in E_2^{T_0}: \|u\|_{T_0}\le R \},
\end{equation}
where $\|\cdot\|_{T_0}$ is given by
$$\|u\|_{T_0}\vcentcolon=\max\big\{\|u\|_{L^\infty(0,T_0;H^2(\Omega))}, \|\partial_tu\|_{L^\infty(0,T_0,H^1(\Omega))} \big\},$$
and $R\ge A$ denotes some constant that will be fixed later.

For given $v\in X_{T_0,R}$, we consider the following linear equation
\begin{equation}\label{eq:nonlin-plate-u-f-v}
\begin{cases}
(I-\gamma\Delta)\partial_t^2u+\Delta^2u=f(x,v) & {\rm in}\,\Omega_T,\\
u=\Delta u=0 & {\rm on}\, \Gamma_T,\\
u(0)=\eta_1,\, \partial_tu(0)=\eta_2 & {\rm in}\, \Omega.
\end{cases}
\end{equation}
By Lemma \ref{lem:est-f-L2-H2}, we have $f(x,v)\in L^2(\Omega_{T_0})\subset L^2(0,T_0;H^{-1}(\Omega))$. Hence it follows from Theorem \ref{proposition-well-posedness-lin-u} that the linear equation \eqref{eq:nonlin-plate-u-f-v} admits a unique solution $u\in E^{T_0}_2$ with $\partial_\nu\partial_tu|_{\Gamma_T}\in L^2(\Gamma_{T_0})$. We define a map
\begin{equation}
\mathcal P: X_{T_0,R}\to E^{T_0}_2,\quad v\mapsto u,
\end{equation}
which maps $v\in X_{T_0,R}$ to the unique solution $u$ of \eqref{eq:nonlin-plate-u-f-v}. We next show that there are proper $T_0$ and $A$ such that $\mathcal P(X_{T_0,R})\subset X_{T_0,R}$. 

Recalling that $u=\Delta u=0$ on $\Gamma_T$, by the energy identity  \eqref{iden:ener-H2-H1} and using the equivalent norms 
\begin{equation}\label{eq:equi-norms-H2}
\|u(t)\|_{H^2(\Omega)}\sim \|\Delta u(t)\|_{L^2(\Omega)},\quad \|\partial_tu(t)\|_{H^1(\Omega)}\sim \|\nabla\partial_t u(t)\|_{L^2(\Omega)},
\end{equation}
we can estimate
\begin{equation}\label{est-energy-:u-T-0}
\begin{split}
&\|u(t)\|^2_{H^2(\Omega)}+\gamma\|\partial_tu(t)\|^2_{H^1(\Omega)}\\
&\le 2\Big|\int_0^{T_0}\int_\Omega f(x,v)\partial_tu\, dxdt\Big|+(1+\gamma)\bigl(\|\eta_1\|^2_{H^2(\Omega)}+\|\eta_2\|^2_{H^1(\Omega)} \bigr)\\
&\le\frac\gamma2 \|\partial_tu(t)\|^2_{L^\infty(0,T_0;L^2(\Omega))}+ 2T_0\gamma^{-1}\|f(x,v)\|^2_{L^2(\Omega_{T_0})}\\
&\quad +(1+\gamma)\bigl(\|\eta_1\|^2_{H^2(\Omega)}+\|\eta_2\|^2_{H^1(\Omega)} \bigr),\quad \forall t\in [0,T_0].
\end{split}
\end{equation}
Taking the supremum of the above inequality over $[0,T_0]$ and applying Lemma \ref{lem:est-f-L2-H2}, we can get
\begin{equation}
\begin{split}
&\|u(t)\|^2_{L^\infty(0,T_0;H^2(\Omega))}+\gamma\|\partial_tu(t)\|^2_{L^\infty(0,T_0;H^1(\Omega))}\\
&\le 4(1+\gamma)A^2+C\gamma^{-1}T_0^2\bigl(\|f(\cdot,0)\|^2_{L^2(\Omega)}+\|a\|^2_{L^{\mathfrak m}(\Omega)}R^2+R^{2(r+1)} \bigr).
\end{split}
\end{equation}
Set $\tilde\gamma\vcentcolon=\min\{\gamma,1\}>0$. We conclude that
\begin{equation}
\|u\|_{T_0}\le 2\tilde\gamma^{-\frac12}(1+\gamma)^{\frac12}A+C_0(\tilde\gamma\gamma)^{-\frac12}T_0(1+R+R^{r+1}),
\end{equation}
where the constant $C_0>0$ only depends on $\Omega,n,p_0,r$ and the norms $\|f(\cdot,0)\|_{L^2(\Omega)}$, $\|a\|_{L^{\mathfrak m}(\Omega)}$ as well as the constants appearing in the assumption {\bf (A.1)}. Choosing $R_0\ge 4\tilde\gamma^{-\frac12}(1+\gamma)^{\frac12}A$, and 
\begin{equation}
0<T_0\le \frac{(\tilde\gamma\gamma)^{\frac12}R}{2C_0(1+R+R^{r+1})},
\end{equation}
we have $\|u\|_{T_0}\le R$, whenever $R\ge R_0$, yielding that $\mathcal P$ maps $X_{T_0,R}$ to itself. 

{\bf Step 2.} We next show that $\mathcal P$ is a strict contraction on $X_{T_0,R}$ provided that $T_0$ is sufficiently small.

We start by observing that
\begin{equation}
\begin{split}
|f(x,\tau_1)-f(x,\tau_2)|&=\Big| \int_0^1\partial_\tau f(x,s\tau_1+(1-s)\tau_2)\, ds\Big|\\
&\le C\bigl(|a(x)||\tau_1-\tau_2|+|\tau_1|^r+|\tau_2|^r  \bigr)|\tau_1-\tau_2|,
\end{split}
\end{equation}
for any $\tau_1,\tau_2\in\mathbb R$ and a.e. $x\in\Omega$. Taking $v_j\in X_{T_0,R}$ and $u_j=\mathcal Pv_j$ for $j=1,2$, again using Lemma \ref{lem:est-f-L2-H2}, we can obtain
\begin{equation}
\begin{split}
&\|f(x,v_1)-f(x,v_2)\|_{L^2(\Omega_{T_0})}\\
&\le CT_0^{\frac12}\bigl( \|a\|_{L^{\mathfrak m}(\Omega)}\|v_1-v_2\|_{L^\infty(0,T_0;H^2(\Omega))}+ \|(|v_1|^r+|v_2|^r)|v_1-v_2|\|_{L^2(\Omega_{T_0})}\bigr).
\end{split}
\end{equation}

For the case $n>4$, noting that $\frac12=\frac{2n}{n-4}+\frac 2n$, by H\"older and Minkowski's inequality, we have
\begin{equation}
\begin{split}
I_{v_1,v_2}\vcentcolon&=\|(|v_1(t)|^r+|v_2(t)|^r)|v_1(t)-v_2(t)|\|_{L^2(\Omega)}\\
&\le \bigl( \||v_1(t)|^r\|_{L^{\frac n2}(\Omega)}+\||v_2(t)|^r\|_{L^{\frac n2}(\Omega)}  \bigr)\|v_1(t)-v_2(t)\|_{L^{\frac{2n}{n-4}}(\Omega)}\\
&\le C\bigl(\||v_1(t)|^r\|_{L^{\frac n2}(\Omega)}+\||v_2(t)|^r\|_{L^{\frac n2}(\Omega)} \bigr)\|v_1(t)-v_2(t)\|_{H^2(\Omega)}.
\end{split}
\end{equation}
If $\frac 2n\le r\le \frac{4}{n-4}$, then $1\le \frac{rn}{2}\le \frac{2n}{n-4}$, then it follows that
\begin{equation}\label{est-I-v1-v2}
I_{v_1,v_2}\le C\bigl(\|v_1(t)\|^r_{H^2(\Omega)} +\|v_2(t)\|^r_{H^2(\Omega)} \bigr)\|v_1(t)-v_2(t)\|_{H^2(\Omega)}.
\end{equation}
The case $r=0$ is clear and we consider the range $0< r<\frac{2}{n}$. We choose $z\ge 1$ such that $1\le rz\le \frac{2n}{n-4}$, implying that $\frac n2<r^{-1}\le z$. Hence,
\begin{equation}
\begin{split}
I_{v_1,v_2}&\le C\bigl(\| |v_1(t)|^r\|_{L^z(\Omega)} +\||v_2(t)|^r\|_{L^z(\Omega)} \bigr)\|v_1(t)-v_2(t)\|_{H^2(\Omega)}\\
&\le C\bigl(\|v_1(t)\|^r_{L^{rz}(\Omega)} +\|v_2(t)\|^r_{L^{rz}(\Omega)} \bigr)\|v_1(t)-v_2(t)\|_{H^2(\Omega)}\\
&\le C\bigl(\|v_1(t)\|^r_{H^2(\Omega)} +\|v_2(t)\|^r_{H^2(\Omega)} \bigr)\|v_1(t)-v_2(t)\|_{H^2(\Omega)}.
\end{split}
\end{equation}

For the case $n<4$, the Sobolev embedding $H^2(\Omega)\hookrightarrow L^\infty(\Omega)$ implies \eqref{est-I-v1-v2} immediately. For the critical case $n=4$, if $r=0$, then the estimate is immediate. If $r>0$, then we choose $z_0>2$ and $2<z_1<\infty$, such that $rz_0>2$ and $\frac12=\frac{1}{z_0}+\frac{1}{z_1}$. Using the embedding $H^2(\Omega)\hookrightarrow L^d(\Omega)$ for $rz_0\le d<\infty$, we can get
\begin{equation}
\begin{split}
I_{v_1,v_2}&\le C\bigl(\|v_1(t)\|^r_{L^{rz_0}(\Omega)} +\|v_2(t)\|^r_{L^{rz_0}(\Omega)} \bigr)\|v_1(t)-v_2(t)\|_{L^{z_1}(\Omega)}\\
&\le C\bigl(\|v_1(t)\|^r_{L^\infty(0,T_0;H^2(\Omega))} +\|v_2(t)\|^r_{H^2(\Omega)} \bigr)\|v_1(t)-v_2(t)\|_{H^2(\Omega)}.
\end{split}
\end{equation}
Thus, we conclude that
\begin{equation}\label{est:fv1-v2-L2}
\begin{split}
&\|f(x,v_1)-f(x,v_2)\|_{L^2(\Omega_{T_0})}\\
&\le CT_0^{\frac12}\|a\|_{L^{\mathfrak m}(\Omega)}\|v_1-v_2\|_{L^\infty(0,T_0;H^2(\Omega))}\\
&\quad +CT_0^{\frac12}\sum_{j=1}^2\|v_j(t)\|^r_{L^\infty(0,T_0;H^2(\Omega))} \|v_1(t)-v_2(t)\|_{L^\infty(0,T_0;H^2(\Omega))}.
\end{split}
\end{equation}

Let $w=u_1-u_2$. We find that $w$ satisfies
\begin{equation}\label{eq:w=u1-u2}
\begin{cases}
(I-\gamma\Delta)\partial_t^2w+\Delta^2w=f(x,v_1)-f(x,v_2) & {\rm in}\, \Omega_{T_0},\\
w=\Delta w=0 & {\rm on}\, \Gamma_{T_0},\\
w(0)=\partial_tw(0)=0 &  {\rm in}\, \Omega.
\end{cases}
\end{equation}
It follows from Theorem \ref{proposition-well-posedness-lin-u} that the linear equation \eqref{eq:w=u1-u2} admits a unique solution $w\in E_2^{T_0}$. Applying the estimates \eqref{est-energy-:u-T-0} and \eqref{est:fv1-v2-L2} and again using the equivalent norms \eqref{eq:equi-norms-H2}, we can deduce that
\begin{equation}
\begin{split}
&\|w(t)\|^2_{H^2(\Omega)}+\gamma\|\partial_tw(t)\|^2_{H^1(\Omega)}\\
&\le \frac\gamma 2\|\partial_tw(t)\|^2_{L^\infty(0,T_0,H^1(\Omega))}+2T_0\gamma^{-1}\|f(x,v_1)-f(x,v_2)\|^2_{L^2(\Omega_{T_0})}\\
&\le \frac\gamma 2\|\partial_tw(t)\|^2_{L^\infty(0,T_0,H^1(\Omega))}+CT_0^2\gamma^{-1}\bigl(\|a\|^2_{L^{\mathfrak m}(\Omega)}+\|v_1\|^{2r}_{L^\infty(0,T_0;H^2(\Omega))} \\
&\quad +\|v_2\|^{2r}_{L^\infty(0,T_0,H^2(\Omega))}\bigr)\|v_1-v_2\|^2_{L^\infty(0,T_0;H^2(\Omega))}.
\end{split}
\end{equation}
Taking the supremum over $[0,T_0]$, we have
\begin{equation}
\begin{split}
&\|w(t)\|^2_{L^\infty(0,T_0;H^2(\Omega))}+\gamma\|\partial_tw(t)\|^2_{L^\infty(0,T_0;H^1(\Omega))}\\
&\le CT_0^2\bigl(\|a\|^2_{L^{\mathfrak m}(\Omega)}+\sum_{j=1}^2\|v_j(t)\|^r_{L^\infty(0,T_0;H^2(\Omega))}\bigr)\|v_1-v_2\|^2_{L^\infty(0,T_0;H^2(\Omega))},
\end{split}
\end{equation}
implying that
\begin{equation}
\|\mathcal Pv_1-\mathcal Pv_2\|_{T_0}\le \widetilde C_0(\tilde\gamma\gamma)^{-\frac12}T_0(1+R^r)\|v_1-v_2\|_{T_0}.
\end{equation}
The positive constant $\widetilde C_0$ only depends on $\Omega,n,r,\mathfrak m$ and the norm $\|a\|_{L^{\mathfrak m}(\Omega)}$, as well as the constants in the assumption {\bf (A.1)}. If we choose 
\begin{equation}
T_0\le\min\left\{\frac{(\tilde\gamma\gamma)^{\frac12}R}{2C_0(1+R+R^{r+1})},\, \frac{(\tilde\gamma\gamma)^{\frac12}}{2\widetilde C_0(1+R^r)} \right\},
\end{equation}
then we see that
\begin{equation}
\|\mathcal Pv_1-\mathcal Pv_2\|_{T_0}\le\frac12\|v_1-v_2\|_{T_0},
\end{equation}
yielding that $\mathcal P:X_{T_0,R}\to X_{T_0,R}$ is a contraction. Hence, applying the Banach fixed point theorem, there is a unique fixed point $u\in X_{T_0,R}$ that is the unique solution to the nonlinear problem \eqref{eq:nonlin-plate-homo-boundary}.

{\bf Step 3.}  We next show that the local solution $u\in X_{T_0,R}$  can be extended beyond $T_0$, thereby yielding a unique global solution. Note that we can extend the local solution to a map \(u:\Omega\times [0,T_*)\to\mathbb R\), where \(T_*\le T\) denotes the maximal existence time, such that for any \(T'<T_*\), the function \(u:\Omega\times [0,T']\to\mathbb R\) solves the nonlinear equation \eqref{eq:nonlin-plate-homo-boundary}.

Now, we fix $0<T'<T_*$ to have $u\in H^1(0,T';H^1(\Omega))$. Let us fix $x\in\Omega$ such that $u(x,\cdot)\in H^1(0,T')\hookrightarrow C([0,T'])$. Recalling the definition of $F$ and condition \eqref{cond:for-partial-tau-f} in {\bf (A.1)}, we see that $F(x,\cdot)\in C_{\rm loc}^{0,1}(\mathbb R)$. Let us modify $\tau\mapsto F(x,\tau)$ outside a neighborhood of the compact set $\bigl[-|u(x,\cdot)|{L^\infty(0,T')},|u(x,\cdot)|{L^\infty(0,T')}\bigr]$ so that $F(x,\cdot)\in C^{0,1}(\mathbb R)$. Then \cite[Theorem 2.1.11]{ziemer2012weakly} guarantees $F(x,u(x,\cdot))\in H^1(0,T')$ and it holds that
\begin{equation}
\partial_tF(x,u(x,t))=f(x,u(x,t))\partial_tu(x,t).
\end{equation}
Hence, we have
\begin{equation}
F(x,u(x,t))=F(x,\eta_1(x))+\int_0^tf(x,u(x,s))\partial_tu(s)\, ds,\quad \forall t\in [0,T'].
\end{equation}
Using the equation \eqref{eq:nonlin-plate-homo-boundary} and the energy identity \eqref{iden:ener-H2-H1}, we can get
\begin{equation}
\begin{split}
&\|\partial_tu(t)\|^2_{L^2(\Omega)}+\|\Delta u(t)\|^2_{L^2(\Omega)}+\gamma\|\nabla\partial_tu(t)\|^2_{L^2(\Omega)}\\
&=\|\Delta\eta_1\|^2_{L^2(\Omega)}+\gamma\|\nabla\eta_2\|^2_{L^2(\Omega)}+\|\eta_2\|^2_{L^2(\Omega)}+2\int_0^t\int_\Omega f(x,u(x,s))\partial_tu(x,s)\, dxds\\
&=\|\Delta\eta_1\|^2_{L^2(\Omega)}+\gamma\|\nabla\eta_2\|^2_{L^2(\Omega)}+\|\eta_2\|^2_{L^2(\Omega)}+2\int_\Omega[F(x,u(x,t))-F(x,\eta_1(x))]\, dx.
\end{split}
\end{equation}
Recalling the condition that $\sup_{\tau\in\mathbb R} F(x,\tau)\le M_0$, and again using the equivalent norms \eqref{equi-norms-H2-H3}, we can derive that
\begin{equation}\label{est:well-posedness-u-H2}
\begin{split}
\|\partial_tu(t)\|^2_{H^1(\Omega)}+\|u(t)\|^2_{H^2(\Omega)}\le C\bigl(\|\eta_1\|^2_{H^2(\Omega)}+\|\eta_2\|^2_{H^1(\Omega)}\bigr)+4CM_0,
\end{split}
\end{equation}
for all $t\in [0,T']$, where $C>0$ is independent of $T'$. Since the extension time $T_0$ depends only on the size of the initial data as well as $\Omega,n,r,p_0,a,f$, we can choose $T'$ sufficiently close to $T_*$, and thereby extend the solution beyond $T_*$. This contradicts the assumption that $T_*<T$ is maximal. Thus, we can conclude that $T_*=T$ and the estimate \eqref{est:well-posedness-u-H2} holds for any $t\in [0,T]$, and $u\in E_2^T$ is the unique solution to the nonlinear equation \eqref{eq:nonlin-plate-homo-boundary}. 

We proceed to show that $u\in E^T_3$ with the aid of the uniqueness of solutions and $u\in E_2^T$ for any $T>0$. First recalling that $(\eta_1,\eta_2)\in\mathcal H_1\times \mathcal H_2$, by some arguments similar to those in {\bf Step 1} and {\bf Step 2}, we can prove that there is a sufficiently $T_0>0$ such that the nonlinear equation \eqref{eq:nonlin-plate-homo-boundary} admits a unique solution $u\in E^{T_0}_3=C([0,T_0];H^3(\Omega))\cap C^1([0,T_0];H^2(\Omega))$. Indeed, we can apply Theorem \ref{proposition-well-posedness-lin-u}, the energy identity \eqref{est:energy-H3-first}, the equivalent norms \eqref{equi-norm-H2} and \eqref{equi-norms-H2-H3} to show that $\mathcal P:X_{T_0,R}\to X_{T_0,R}$ is a strict contraction provided that $T_0$ is small enough. Here the function space $X_{T_0,R}$ is replaced by
\begin{equation}
X_{T_0,R}\vcentcolon=\{u\in E_3^{T_0}: \|u\|_{T_0}\le R \},
\end{equation}
where $\|\cdot\|_{T_0}$ is given by
$$\|u\|_{T_0}\vcentcolon=\max\big\{\|u\|_{L^\infty(0,T_0;H^3(\Omega))}, \|\partial_tu\|_{L^\infty(0,T_0,H^2(\Omega))} \big\}.$$
If $T=T_0$, then the proof is finished. Hence, it remains to show that the maximal existence time $T_*$ can be extended to $T>T_*$.

To this end, we use the energy inequality \eqref{est:energy-H3-first} to have
\begin{equation}
\begin{split}
&\|\nabla\Delta u(t)\|^2_{L^2(\Omega)}+\gamma\|\Delta\partial_tu(t)\|^2_{L^2(\Omega)}+\|\nabla\partial_tu(t)\|^2_{L^2(\Omega)}\\
&\le (1+\gamma)\bigl(\|\eta_1\|^2_{H^3(\Omega)}+\|\eta_2\|^2_{H^2(\Omega)} \bigr)+\|f(x,u)\|^2_{L^2(\Omega_T)}+\int_0^t\|\Delta\partial_tu(l)\|^2_{L^2(\Omega)}\, dl
\end{split}
\end{equation}
for all $t\in [0,T']$. Since $u\in E_2^T$ is the unique global solution to \eqref{eq:nonlin-plate-homo-boundary} for any $T>0$, it follows from \eqref{est:well-posedness-u-H2} that
\begin{equation}
\begin{split}
&\|\nabla\Delta u(t)\|^2_{L^2(\Omega)}+\gamma\|\Delta\partial_tu(t)\|^2_{L^2(\Omega)}+\|\nabla\partial_tu(t)\|^2_{L^2(\Omega)}\\
&\le C\bigl(\|\eta_1\|^2_{H^3(\Omega)}+\|\eta_2\|^2_{H^2(\Omega)} \bigr)+CT\bigl(\|f(\cdot,0)\|^2_{L^2(\Omega)}+\|a\|^2_{L^{\mathfrak m}(\Omega)}\|u\|^2_{L^\infty(0,T;H^2(\Omega))}\\
&\quad +\|u\|^{2(r+1)}_{L^\infty(0,T;H^2(\Omega))}\bigr)+\int_0^t\|\Delta\partial_tu(l)\|^2_{L^2(\Omega)}\, dl\\
&\le C\bigl(\|\eta_1\|^2_{H^3(\Omega)}+\|\eta_2\|^2_{H^2(\Omega)} \bigr)+CT\bigl(\|f(\cdot,0)\|^2_{L^2(\Omega)}+M_0+M_0^{r+1}\\
&\quad +\|\eta_1\|^2_{H^2(\Omega)}+\|\eta_2\|^2_{H^1(\Omega)} +\|\eta_1\|^{2(r+1)}_{H^2(\Omega)}+\|\eta_2\|^{2(r+1)}_{H^1(\Omega)}\bigr)+\int_0^t\|\Delta\partial_tu(l)\|^2_{L^2(\Omega)}\, dl,
\end{split}
\end{equation}
where the positive constant $C$ is independent of $T'$.
Finally, the equivalent norms of $H^3(\Omega)\times H^2(\Omega)$ and Gronwall's inequality imply 
\begin{equation}
\|u(t)\|^2_{H^3(\Omega)}+\|\partial_tu(t)\|^2_{H^2(\Omega)}\le C(T+1)e^{Ct}.
\end{equation}
Hence, there exists $C>0$ such that
\begin{equation}
\|u(t)\|^2_{H^3(\Omega)}+\|\partial_tu(t)\|^2_{H^2(\Omega)}\le C,\quad \forall\,  0\le t<T_*,
\end{equation}
as desired. Therefore, the proof of Theorem \ref{thm:global-well-posed-H3} is finished.
\end{proof}

\begin{remark}
Let us consider the following free Kirchhoff plate equation with focusing cubic nonlinearity,
\begin{equation}\label{eq:blow-up-remark}
(I-\gamma\Delta)\partial_t^2u+\Delta^2u-u^3=0.
\end{equation}
It is clear that $f(x,u)=u^3$ fails to satisfy the assumption {\bf (A.2)}. Assume that $u(x,t)=\phi(t)$ is a space-independent solution to this equation. We can check that $\phi(t)=\frac{\sqrt2}{t-1}$ solves \eqref{eq:blow-up-remark}, which blows up as $t\to 1$.
\end{remark}

\subsection{Local well-posedness of nonlinear plate equations}\label{subsec:well-posedness-non-small}
For the local well-posedness, we restrict ourselves to consider the most practical cases $n=1,2,3$ only. The reason is simply that we can directly use the $E_3^T$-regularity of solutions, which can be continuously embedded into $C(\overline\Omega_T)$, for linear Kirchhoff plate equations established in Theorem \ref{proposition-well-posedness-lin-u}. 

The (local) well-posedness result for the nonlinear plate equation \eqref{eq:intro-non-lin-plate} under small initial boundary data is now stated as follows.
\begin{theorem}[Local well-posedness of the nonlinear equation \eqref{eq:intro-non-lin-plate-g}]\label{thm-local-well-posedness}
Let the dimension $n\in\{1,2,3\}$, and let $T>0$ be given. Assume that the nonlinearity $g:\overline\Omega\times[0,T]\times\mathbb C\to\mathbb C$ is analytic on $z\in\mathbb C$ with values in $L^\infty(0,T;L^{p_0}(\Omega))$ with $p_0$ satisfying \eqref{cond-for-tilde-q}, and $g(\cdot,\cdot,0)=0$ in $\overline\Omega_T$. Then there exists a small constant $\epsilon>0$, such that for any initial boundary data $(h_1,h_2,\eta_1,\eta_2)\in\mathcal S_\epsilon(\Gamma_T)\times\mathcal S_\epsilon(\Omega)$ (see \eqref{set:admi-boundary-data} and \eqref{set:initial-data-small}) that satisfy
\begin{equation}
\|(h_1,h_2)\|_{\mathcal H_\Gamma^T}+\|(\eta_1,\eta_2)\|_{\mathcal H^0_\Omega}<\epsilon,
\end{equation}
the nonlinear equation \eqref{eq:intro-non-lin-plate-g} admits a unique solution $u\in E^T_3$ with boundary traces $\partial_\nu\partial_tu|_{\Gamma_T},\partial_\nu\Delta u|_{\Gamma_T}\in L^2(\Gamma_T)$.
\end{theorem}
\begin{proof}
Let us first apply the Sobolev embedding to get $E_3^T\hookrightarrow C^1([0,T];C(\overline\Omega))$. Suppose that $v$ satisfy the following linear equation
\begin{equation}\label{eq:lin-v-prove-local-well-pose}
\begin{cases}
(I-\gamma\Delta)\partial_t^2v+\Delta^2v-\partial_zg(x,t,0)v=0 & {\rm in}\, \Omega_T,\\
v=h_1,\, \Delta v=h_2 & {\rm on}\, \Gamma_T,\\
v(0)=\eta_1,\, \partial_tv(0)=\eta_2 & {\rm in}\, \Omega.
\end{cases}
\end{equation}
Noting that $(\eta_1,\eta_2)\in\mathcal H^0_\Omega=\mathcal H_1\times\mathcal H_2$ (see \eqref{def:H1-H2}), and $(h_1,h_2)\in\mathcal H_\Gamma^T$ (see \eqref{function-space-initial-boundary}), by the assumption that $\partial_zg(\cdot,\cdot,z)\in L^\infty(0,T;L^{p_0}(\Omega))$ for any $z\in\mathbb C$, it then follows from Theorem \ref{proposition-well-posedness-lin-u} that the equation \eqref{eq:lin-v-prove-local-well-pose} admits a unique solution $v\in E^T_3$, such that
\begin{equation}
\|v\|_{E^T_3}\le C\bigl(\|(h_1,h_2)\|_{\mathcal H_\Gamma^T}+\|(\eta_1,\eta_2\|_{\mathcal H_\Omega^0}  \bigr)\le C\epsilon.
\end{equation}
For given $\delta>0$, we define a set
\begin{equation}
Z_\delta\vcentcolon=\{\varphi\in E^T_3: \|\varphi\|_{E^T_3}\le \delta\}.
\end{equation}
Now, taking $\tilde v\in Z_\delta$, for the solution $v$, we define a map
\begin{equation}
\mathcal A: Z_\delta\to E^T_3,\quad \tilde v\mapsto w,
\end{equation}
where $w$ satisfies the following equation
\begin{equation}\label{eq:lin-w-prove-local-well-pose}
\begin{cases}
(I-\gamma\Delta)\partial_t^2w+\Delta^2w-\partial_zg(x,t,0)w=\sum_{k=2}^\infty\partial_z^kg(x,t,0)\frac{(v+\tilde v)^k}{k!} & {\rm in}\, \Omega_T,\\
w=\Delta w=0 & {\rm on}\, \Gamma_T,\\
w(0)=\partial_tw(0)=0 & {\rm in}\, \Omega.
\end{cases}
\end{equation}

Since $g(\cdot,\cdot,z)$ is analytic in $\mathbb C$, for any given constant $R>0$, we can get
\begin{equation}
\|\partial_z^kg(x,t,0)\|_{L^2(\Omega_T)}\le \frac{k!}{R^k}\sup_{|z|=R}\|g(x,t,z)\|_{L^2(\Omega_T)}.
\end{equation}
Hence, using the estimate \eqref{est:lin-u-H3} in Theorem \ref{proposition-well-posedness-lin-u}, for any $R>2\max\{\delta,\epsilon\}$, we can obtain 
\begin{equation}
\begin{split}
\|w\|_{E^T_3}&=\|\mathcal A\tilde v\|_{E^T_3}\le C\sum_{k=2}^\infty\frac{1}{k!}\|(v+\tilde v)^k\partial_z^kg(\cdot,\cdot,0)\|_{L^2(\Omega_T)}\\
&\le C\sum_{k=2}^\infty\frac{2^{k-1}}{R^k}\sup_{|z|=R}\|g(\cdot,\cdot,z)\|_{L^2(\Omega_T)}\bigl(\|v\|_{E^T_3}^k+\|\tilde v\|^k_{E^T_3} \bigr)\\
&\le \frac{C}{R}\sup_{|z|=R}\|g(\cdot,\cdot,z)\|_{L^2(\Omega_T)}\sum_{k=1}^\infty\frac{2^k}{R^k}\bigl(\delta^{k+1}+\epsilon^{k+1}\bigr)\\
&=\frac{2C}{R}\sup_{|z|=R}\|g(\cdot,\cdot,z)\|_{L^2(\Omega_T)}\Bigl(\frac{\delta^2}{R-2\delta}+\frac{\epsilon^2}{R-2\epsilon}\Bigr).
\end{split}
\end{equation}
Here, the positive constant $C$ depends on $\Omega,T,\gamma,n,p_0$ and $\partial_zg(x,t,0)$,  but is independent of $\epsilon$ and $\delta$.
We choose $R=2$ and sufficiently small $0<\epsilon_0=\delta_0<1/2$, such that
\begin{equation}
C\delta_0\sup_{|z|=R}\|g(\cdot,\cdot,z)\|_{L^2(\Omega_T)}\le \frac12.
\end{equation}
This implies $\|\mathcal A\tilde v\|_{E^T_3}\le \delta$ for any $0<\delta<\delta_0$, and thus $\mathcal A:Z_\delta\to Z_\delta$ is well-defined. We next show that $\mathcal A$ is a contraction.

Indeed, observe that, for any $R>3\delta$,
\begin{equation}
\begin{split}
\|\mathcal A\tilde v_1-\mathcal A\tilde v_2\|_{E^T_3}&\le C\sum_{k=2}^\infty\frac{1}{k!}\|\partial_z^kg(\cdot,\cdot,0)\|_{L^2(\Omega_T)}\|(v+\tilde v_1)^k-(v+\tilde v_2)^k\|_{L^\infty(\Omega_T)}\\
&\le \widetilde C\sum_{k=2}^\infty\frac{k}{R^k}\sup_{|z|=R}\|g(\cdot,\cdot,z)\|_{L^2(\Omega_T)}(3\delta)^{k-1}\|\tilde v_1-\tilde v_2\|_{E^T_3}\\
&=\widetilde C\frac{3\delta(2R-3\delta)}{R(R-3\delta)^2}\sup_{|z|=R}\|g(\cdot,\cdot,z)\|_{L^2(\Omega_T)}\|\tilde v_1-\tilde v_2\|_{E^T_3}.
\end{split}
\end{equation}
We still choose $R=2$ and $\epsilon_2=\delta_2<1/2$, such that
\begin{equation}
15\widetilde C\delta_2\sup_{|z|=2}\|g(\cdot,\cdot,z)\|_{L^2(\Omega_T)}\le \frac12,
\end{equation}
for any $0<\delta<\delta_2$. Now, choosing $\epsilon=\delta=\min\{\delta_0,\delta_1\}$, we see that $\mathcal A:Z_\delta\to Z_\delta$ is a contraction. It then follows from Banach fixed point theorem that $\mathcal A$ has a unique fixed point $\tilde v$ such that $u=v+\tilde v=v+w\in E^T_3$ is the unique solution to the nonlinear equation \eqref{eq:intro-non-lin-plate-g}. Therefore, the proof of Theorem \ref{thm-local-well-posedness} is finished.
\end{proof}

\begin{remark}
If the initial boundary data $\eta_1,\eta_2,h_1,h_2$ and the source term $h$ are sufficiently regular and satisfy necessary compatibility conditions, then the linear equation \eqref{eq:lin-u-plate} admits more regular solutions $u\in E_m^T$ for some $m\ge 3$. Consequently, as $E^T_m$ is a Banach algebra for $m$ large enough, the study of the local well-posedness theory for dimensions $n>3$ presents no essential difficulties.
\end{remark}

\section{Construction of geometric optics solutions}\label{Construction-GO-solution}
In this section, we will construct geometric optics solutions for the linear plate equation 
\begin{equation}
\begin{cases}
(I-\gamma\Delta)\partial_t^2v+\Delta^2v+qv=0 & {\rm in}\, \Omega_T,\\
v=h_1,\, \Delta v=h_2 & {\rm on}\, \Gamma_T,\\
v(0)=v_0,\, \partial_tv(0)=v_1  & {\rm in}\, \Omega,
\end{cases}
\end{equation}
of the following form
\begin{equation}\label{expression-GO}
v(x,t)=e^{{\rm i}\sigma\varphi(x,t)}a(x,t)+r_\sigma(x,t),
\end{equation}
where ${\rm i}=\sqrt{-1}$ is the imaginary unit, $\sigma>1$ is a fixed constant, $\varphi$ and $a$ denote phase and amplitude functions, respectively. $r_\sigma$ is the remainder term that satisfies certain asymptotic property with respect to $\sigma$.

For the purpose of notational convenience, we set 
$$\mathcal L_q\vcentcolon=(I-\gamma\Delta)\partial_t^2+\Delta^2+q.$$
Direct calculation shows
\begin{equation}
\begin{split}
e^{-{\rm i}\sigma\varphi}\mathcal L_q(e^{{\rm i}\sigma\varphi}a)&=\sum_{k=1}^4({\rm i}\sigma)^k\widetilde L_k(a,\varphi)+\mathcal L_qa.
\end{split}
\end{equation}
Here, the terms $\widetilde L_k$ ($k=1,2,3$) are respectively given by
\begin{equation}
\begin{split}
\widetilde L_1(a,\varphi)&=2a_t\varphi_t+a\varphi_{tt}-\gamma L_1(a,\varphi),\\
\widetilde L_2(a,\varphi)&=a\varphi_t^2-\gamma L_2(a,\varphi)+P_2(a,\varphi),\\
\widetilde L_3(a,\varphi)&=-\gamma L_3(a,\varphi)+P_3(a,\varphi),\\
\widetilde L_4(a,\varphi)&=(|\nabla\varphi|^4-\gamma\varphi_t^2|\nabla\varphi|^2)a,
\end{split}
\end{equation}
where 
\begin{equation}
\begin{split}
L_1(a,\varphi)&=2\nabla a_{tt}\cdot\nabla\varphi + 4\nabla a_t\cdot\nabla\varphi_t + 2\nabla a\cdot\nabla\varphi_{tt} + a_{tt}\Delta\varphi + 2a_t\Delta\varphi_t\\
&\qquad +a\Delta\varphi_{tt} + 2\varphi_t\Delta a_t + \varphi_{tt}\Delta a\\
L_2(a,\varphi) &= a_{tt}|\nabla\varphi|^2 + 4a_t\nabla\varphi\cdot\nabla\varphi_t + 2a|\nabla\varphi_t|^2 + 2a\nabla\varphi\cdot\nabla\varphi_{tt} \\
&\quad + 2\varphi_t\bigl( 2\nabla a_t\cdot\nabla\varphi + 2\nabla a\cdot\nabla\varphi_t + a_t\Delta\varphi + a\Delta\varphi_t \bigr) \\
&\quad + \varphi_t^2\Delta a + 2\varphi_{tt}\nabla a\cdot\nabla\varphi + a\varphi_{tt}\Delta\varphi \\
L_3(a,\varphi)&=2\varphi_t( a_t|\nabla\varphi|^2 + 2a\nabla\varphi\cdot\nabla\varphi_t)+\varphi_{tt}\, a|\nabla\varphi|^2 +\varphi_t^2(2\nabla a\cdot\nabla\varphi + a\Delta\varphi),
\end{split}
\end{equation}
and 
\begin{equation}
\begin{split}
P_1(a,\varphi) &=\Delta(2\nabla a\cdot\nabla\varphi + a\Delta\varphi) + 2\nabla(\Delta a)\cdot\nabla\varphi + (\Delta a)\Delta\varphi,\\
P_2(a,\varphi) &=\Delta(a|\nabla\varphi|^2)+2\nabla G\cdot\nabla\varphi+G\Delta\varphi+(\Delta a)|\nabla\varphi|^2,\\
P_3(a,\varphi) &=2\nabla(a|\nabla\varphi|^2)\cdot\nabla\varphi + a|\nabla\varphi|^2\Delta\varphi + G|\nabla\varphi|^2,
\end{split}
\end{equation}
where $G=2\nabla a\cdot\nabla\varphi+a\Delta\varphi$.

Let $a=a_0+\sigma^{-1}a_1+\sigma^{-2}a_2$. Noting that $\widetilde L_k(\cdot,\varphi)$ is linear with respect to the variable $a$, we can compute
\begin{equation}
\begin{split}
e^{-i\sigma\varphi}\mathcal L_q(e^{{\rm i}\sigma\varphi}a)=\sigma^{-2}\mathcal L_qa_2+\sum_{k=-1}^3\sigma^k\widetilde P_k+\sigma^4\widetilde L_4(a_0,\varphi),   
\end{split}
\end{equation}
where 
\begin{equation}
\begin{split}
\widetilde P_{-1}(a_0,a_1,a_2,\varphi)&={\rm i}[\widetilde L_1(a_0,\varphi)-\widetilde L_3(a_2,\varphi)]+\mathcal L_qa_1\\
\widetilde P_0(a_0,a_1,a_2,\varphi)&=\mathcal L_qa_0+{\rm i}\widetilde L_1(a_1,\varphi)-\widetilde L_2(a_2,\varphi),\\
\widetilde P_1(a_0,a_1,a_2,\varphi)&={\rm i}\widetilde L_1(a_0,\varphi)-\widetilde L_2(a_1,\varphi)-{\rm i}\widetilde L_3(a_2,\varphi),\\
\widetilde P_2(a_0,a_1,a_2,\varphi)&=\widetilde L_2(a_0,\varphi)-{\rm i}\widetilde L_3(a_1,\varphi)+\widetilde L_4(a_2,\varphi),\\
\widetilde P_3(a_0,a_1,\varphi)&=-{\rm i}\widetilde L_3(a_0,\varphi)+\widetilde L_4(a_1,\varphi).
\end{split}
\end{equation}
Considering the powers of $\sigma$, we are required to solve the equations 
\begin{equation}
\widetilde P_j(a_0,a_1,a_2,\varphi)=\widetilde P_3(a_0,a_1,\varphi)=\widetilde L_4(a_0,\varphi)=0,\quad j=1,2.
\end{equation}
For $\widetilde L_4(a_0,\varphi)=0$, it suffices to solve the following eikonal equation
\begin{equation}\label{eq-eikonal-var}
|\nabla\varphi|^4-\gamma\varphi_t^2|\nabla\varphi|^2=0.
\end{equation}
It is clear that 
\begin{equation}\label{expre-phase-var}
\varphi=x\cdot\theta+t,\quad \theta\in \mathbb S^{n-1}(\sqrt\gamma),
\end{equation}
satisfies the above equation \eqref{eq-eikonal-var}. With such $\varphi$, we have $\widetilde L_4(a_k,\varphi)=0$ for $k=1,2$. By $\widetilde P_3(a_0,a_1,\varphi)=0$, and recalling the expression of $\widetilde L_3$ (i.e., $L_3,P_3$), we can derive the following transport equation for $a_0$,
\begin{equation}\label{eq-trans-a0}
\widetilde L_3(a_0,\varphi)=-2\gamma(\gamma a_{0t}-\nabla a_0\cdot\theta)=0\Leftrightarrow (\gamma a_{0t}-\nabla a_0\cdot\theta)=0.
\end{equation}
We choose $\zeta\in\mathbb R^{n+1}$ such that $\zeta\cdot (\gamma,-\theta)=\gamma\zeta\cdot(1,-\gamma^{-1}\theta)=0$. It is clear that
\begin{equation}\label{expre-ampli-a0}
a_0(x,t)=e^{-{\rm i}(x,t)\cdot \zeta}
\end{equation}
solves the transport equation \eqref{eq-trans-a0}. With $a_0$ at hand, we can solve the inhomogeneous transport equation
\begin{equation}
\begin{split}
\widetilde P_2(a_0,a_1,a_2,\varphi)=0&\Leftrightarrow\widetilde L_3(a_1,\varphi)=-{\rm i}\widetilde L_2(a_0,\varphi)\\
&\Leftrightarrow\gamma a_{1t}-\nabla a_1\cdot\theta=\frac{{\rm i}}{2\gamma}\widetilde L_2(a_0,\varphi)
\end{split}
\end{equation}
to get 
\begin{equation}\label{expre-ampli-a1}
a_1(x,t)=a_0(x,t)+\frac{{\rm i}}{2\gamma^2}\int_0^t \widetilde L_2(a_0,\varphi)(x+\gamma^{-1}\tau\theta,t-\tau)\,d\tau.
\end{equation}
Successively‌, we can solve the transport equation $\widetilde P_1(a_0,a_1,a_2,\varphi)=0$ to get 
\begin{equation}\label{expre-ampli-a2}
a_2(x,t)=a_0(x,t)+\frac{{\rm i}}{2\gamma^2}\int_0^t[{\rm i}\widetilde L_1(a_0,\varphi)-\widetilde L_2(a_1,\varphi)](x+\gamma^{-1}\tau\theta,t-\tau)\,d\tau,
\end{equation}
since $\varphi,a_0,a_1$ have been constructed. 

Lastly, we conclude that the remainder term $r_\sigma$  satisfies the following equation
\begin{equation}\label{eq:remainder-r}
\begin{cases}
\mathcal L_qr_\sigma= -e^{{\rm i}\sigma t}F_\sigma&  {\rm in}\,\, \Omega_T,\\
r_\sigma=\Delta r_\sigma=0  & {\rm on}\,\, \Gamma_T,\\
r_\sigma(0)=\partial_tr_\sigma(0)=0 & {\rm in}\,\, \Omega,
\end{cases}
\end{equation}
where 
\begin{equation}
\begin{split}
F_\sigma(x,t)\vcentcolon&=F_\sigma(a_0,a_1,a_2,\varphi)\\
&=[\sigma^{-2}\mathcal L_qa_2+\sigma^{-1}\widetilde P_{-1}(a_0,a_1,a_2,\varphi)+\widetilde P_0(a_0,a_1,a_2,\varphi)]e^{{\rm i}\sigma x\cdot\theta}.
\end{split}
\end{equation}
We have the following useful lemma for $r_\sigma$, which reveals that $r_\sigma$ vanishes as $\sigma\to \infty$.                        
\begin{lemma}\label{lem:asympotic-r-sigma}
Let $\sigma>1$, and let $q\in L^\infty(0,T;L^{p_0}(\Omega))$ with $p_0$ satisfying \eqref{cond-for-tilde-q}. Then there exists a solution $v\in E^T_3\cap H^2(0,T;H^1(\Omega))$ to $\mathcal L_qv=0$ of the form \eqref{expression-GO}, where the remainder term $r_\sigma$ satisfies \eqref{eq:remainder-r} with boundary traces $(\partial_\nu r_\sigma,\partial_\nu\Delta r_\sigma)|_{\Gamma_T}\in H^1(0,T;L^2(\Gamma))\times L^2(\Gamma_T)$. Moreover, there is a constant $C>0$, depending on $\Omega,T,\gamma,n,p_0,\varphi$ and the norm $\|q\|_{L^\infty(0,T;L^{p_0}(\Omega))}$, such that
\begin{equation}\label{est-r-sigma-asy}
\sigma\|r_\sigma\|_{L^2(0,T;H^2(\Omega))}+\|r_\sigma\|_{L^2(0,T;H^3(\Omega))}+\|\partial_tr_\sigma\|_{L^2(0,T;H^2(\Omega))}\le C\|a_0\|_{C^5(\overline\Omega_T)}.
\end{equation}
\end{lemma}
\begin{proof}
The phase and amplitude functions $\varphi$ and $a=a_0+\sigma^{-1}a_1+\sigma^{-2}a_2$ have been respectively constructed (see \eqref{expre-phase-var}, \eqref{expre-ampli-a0}, \eqref{expre-ampli-a1} and \eqref{expre-ampli-a2}). From the expressions of $\varphi,a_0,a_1,a_2$, it is not difficult to check that $F_\sigma\in L^2(\Omega_T)$. Hence, it follows from Theorem \ref{proposition-well-posedness-lin-u} that $r_\sigma\in E^T_3\cap H^2(0,T;H^1(\Omega))$ satisfying the estimate 
\begin{equation}
\begin{split}
&\|r_\sigma\|_{L^\infty(0,T;H^3(\Omega))}+\|\partial_tr_\sigma\|_{L^\infty(0,T;H^2(\Omega))}+\|\partial_\nu\partial_tr_\sigma\|_{L^2(\Gamma_T)}+\|\partial_\nu\Delta r_\sigma\|_{L^2(\Gamma_T)}\\
&\le C\|e^{{\rm i}\sigma t}F_\sigma\|_{L^2(\Omega_T)}\le C\|a_0\|_{C^4(\overline\Omega_T)}.
\end{split}
\end{equation}
It remains to show the estimate
\begin{equation}
\sigma\|r_\sigma\|_{L^2(0,T;H^2(\Omega))}\le C\|a_0\|_{C^5(\overline\Omega_T)}.
\end{equation}
To this aim, we introduce the time integral transform 
\begin{equation}\label{eq:change-v-w}
w_\sigma(x,t)=\int_0^tr_\sigma(x,l)dl
\end{equation}
of $r_\sigma$. We see that $w_\sigma$ satisfies the following equation
\begin{equation}\label{eq:w-sig}
\begin{cases}
(I-\gamma\Delta)\partial_t^2w_\sigma+\Delta^2w_\sigma=H & {\rm in}\, \Omega_T,\\
w_\sigma=\Delta w_\sigma=0 & {\rm on}\, \Gamma_T,\\
w_\sigma(0)=\partial_tw_\sigma(0)=0 & {\rm in}\, \Omega,
\end{cases}
\end{equation}
where 
\begin{equation}
H(x,t)=-\int_0^t(qr_\sigma)(x,l)dl+\int_0^te^{{\rm i}\sigma l}F_\sigma(x,l)dl.
\end{equation}
We can check that the source term $H\in H^1(0,T;L^2(\Omega))$ satisfying $H(\cdot,0)=0$ in $\Omega$. Differentiating $w_\sigma$ with respect to $t$ and applying Theorem \ref{proposition-well-posedness-lin-u}, the problem \eqref{eq:w-sig} admits a unique solution  
$w_\sigma\in H^2(0,T;H^2(\Omega))\cap L^2(0,T;H^3(\Omega)).$
Equation \eqref{eq:w-sig} together with elliptic regularity for the bi-harmonic operator $\Delta^2w_\sigma\in L^2(\Omega)$ imply $w_\sigma\in L^2(0,T;H^4(\Omega))$.

Multiplying the first equation of \eqref{eq:w-sig} by $\partial_t\bar w_\sigma$ and $\Delta\partial_t\bar w_\sigma$ respectively, taking the real parts and integrating over $\Omega\times (0,t)$ for $t\in (0,T]$, we can get 
\begin{equation}
\begin{split}
N(t)\vcentcolon&=\frac12\int_\Omega\bigl[|\partial_tw_\sigma(x,t)|^2+(\gamma+1)|\nabla\partial_tw_\sigma(x,t)|^2+|\Delta w_\sigma(x,t)|^2\\
&\quad\quad \qquad+\gamma|\Delta\partial_tw_\sigma(x,t)|^2+|\nabla\Delta w_\sigma(x,t)|^2\bigr]dx\\
&=\Re \int_0^t\int_\Omega H(\partial_t\bar w+\Delta\partial_t\bar w_\sigma)(x,l)dxdl.
\end{split}
\end{equation}
Noting that $w_\sigma=\Delta w_\sigma=0$ on $\Gamma_T$, we can invoke the equivalent norms of $H^3(\Omega)$ and $H^2(\Omega)$ (see \eqref{equi-norms-H2-H3} in the proof of Theorem \ref{proposition-well-posedness-lin-u}). Hence, we can consider
\begin{equation}\label{equi-norm-H3}
\widetilde N(t)\vcentcolon=\|w_\sigma(t)\|^2_{H^3(\Omega)}+\|\partial_tw_\sigma(t)\|^2_{H^2(\Omega)}\sim N(t).
\end{equation}

Let $\beta_\sigma(x,t)=\int_0^te^{{\rm i}\sigma l}F_\sigma(x,l)dl$. Using Fubini's theorem, we can derive
\begin{equation}
\begin{split}
& \Re \int_0^t\int_\Omega H(\partial_t\bar w+\Delta\partial_t\bar w_\sigma)(x,l)\, dxdl\\
&=-\Re \int_0^t\int_\Omega\Bigl(\int_0^l(qr_\sigma)(x,s)ds\Bigr)(\partial_t\bar w+\Delta\partial_t\bar w_\sigma)(x,l)\, dxdl\\
&\quad +\Re \int_0^t\int_\Omega\Bigl(\int_0^le^{{\rm i}\sigma s}F_\sigma (x,s)ds\Bigr)(\partial_t\bar w+\Delta\partial_t\bar w_\sigma)(x,l)\, dxdl\\
&=-\Re \int_0^t\int_\Omega (qr_\sigma)(x,s)\Bigl(\int_s^t(\partial_t\bar w_\sigma+\Delta\partial_t\bar w_\sigma)(x,l)dl\Bigr)dxds\\
&\quad +\Re \int_0^t\int_\Omega \beta_\sigma(\partial_t\bar w_\sigma+\Delta\partial_t\bar w_\sigma)(x,l)\, dxdl\vcentcolon=\Re  I_1+\Re I_2.
\end{split}
\end{equation}

For the term $I_1$, applying the inequality \eqref{ineq:qu-H2}, we can obtain
\begin{equation}
\begin{split}
|I_1|&=\Big|\int_0^t\int_\Omega(qr_\sigma)(x,s)[\bar w_\sigma(x,s)-\bar w_\sigma(x,t)+\Delta\bar w_\sigma(x,s)-\Delta\bar w_\sigma(x,t)]\, dxds\Big|\\
&\le \int_0^t\|q(s)\partial_tw_\sigma(s)\|_{H^{-1}(\Omega)}\bigl(\|w_\sigma(t)\|_{H^1(\Omega)}+\|\Delta w_\sigma(t)\|_{H^1(\Omega)}\bigr)\,ds\\
&\quad +\int_0^t\|q(s)\partial_tw_\sigma(s)\|_{H^{-1}(\Omega)}\bigl( \|w_\sigma(s)\|_{H^1(\Omega)}+ \|\Delta w_\sigma(s)\|_{H^1(\Omega)} \bigr)\,ds\\
&\le C\|w_\sigma(t)\|_{H^3(\Omega)}\int_0^t\|\partial_tw_\sigma(s)\|_{H^2(\Omega)}\|q(s)\|_{L^p(\Omega)}\,ds \\
&\quad +C\int_0^t\|\partial_tw_\sigma(s)\|_{H^2(\Omega)}\|w_\sigma(s)\|_{H^3(\Omega)}\|q(s)\|_{L^p(\Omega)}\,ds
\end{split}
\end{equation}
Using Young's inequality, we deduce that
\begin{equation}\label{est:I-1}
\begin{split}
|I_1|&\le C\widetilde N^{\frac12}(t)\int_0^t\tilde N^{\frac12}(s)\|q(s)\|_{L^p(\Omega)}\, ds+C\int_0^t\widetilde N(s)\|q(s)\|_{L^p(\Omega)}\,ds\\
&\le \frac12\widetilde N(t)+(C^2+C)\bigl(\|q\|_{L^\infty(0,T;L^p(\Omega)}^2+\|q\|_{L^\infty(0,T;L^p(\Omega))}\bigr)\int_0^t\widetilde N(s)\,ds.
\end{split}
\end{equation}
Since the conditions \eqref{cond-for-tilde-q} on $p_0$ are more restrictive than \eqref{cond-for-q} on $p$, we can conclude that the above estimate \eqref{est:I-1} also hold for $q\in L^\infty(0,T;L^{p_0}(\Omega))$. In fact, we can choose $p=2$ for $1\le n\le 6$ and $p=\frac n3$ for $n>6$. Thus, we have $p_0\ge p$.

For the term $I_2$, we have
\begin{equation}\label{est:I-2}
|I_2|\le \|\beta_\sigma\|_{L^2(\Omega_T)}\int_0^t\|\partial_t\bar w_\sigma(s)\|_{L^2(\Omega)}\, ds\le \|\beta_\sigma\|^2_{L^2(\Omega_T)}+\int_0^t \widetilde N(s)\, ds.
\end{equation}
Combining \eqref{est:I-1} with \eqref{est:I-2} and using \eqref{equi-norm-H3}, we can obtain
\begin{equation}
\widetilde N(t)\le C\int_0^t\widetilde N(s)\, ds+C\|\beta_\sigma\|^2_{L^2(\Omega_T)}.
\end{equation}
The Gronwall's inequality yields
\begin{equation}\label{est:widetilde-Nt}
\widetilde N(t)\le C\|\beta_\sigma\|^2_{L^2(\Omega_T)}e^{CT},
\end{equation}
where $C$ depends on $\Omega,T,\gamma,n,p_0$ and the norm $\|q\|_{L^\infty(0,T;L^{p_0}(\Omega)}$. 

We turn our attention to $\beta_\sigma$ to find 
\begin{equation}
\begin{split}
\beta_\sigma(x,t)&=-{\rm i}\sigma^{-1}\int_0^t[\partial_l(e^{{\rm i}\sigma l}F_\sigma)-e^{{\rm i}\sigma l}\partial_lF_\sigma](x,l)\, dl\\
&={\rm i}\sigma^{-1}\int_0^te^{{\rm i}\sigma l}\partial_lF_\sigma(x,l)\, dl-{\rm i}\sigma^{-1}[e^{{\rm i}\sigma t}F_\sigma(x,t)-F_\sigma(x,0)].
\end{split}
\end{equation}
Hence, using the embedding result $H^1(0,T)\hookrightarrow C([0,T])$ and H\"older's inequality, we can estimate
\begin{equation}
\begin{split}
\|\beta_\sigma\|^2_{L^2(\Omega_T)}&\le \sigma^{-2}\int_0^T\int_\Omega\Big| \int_0^t e^{{\rm i}\sigma l}\partial_lF_\sigma(x,l)dl \Big|\, dxdt\\
&\quad +\sigma^{-2}\int_0^T\int_\Omega \big|e^{{\rm i}\sigma t}[F_\sigma(x,t)-F_\sigma(x,0)]\big|^2\, dxdt\\
&\le \sigma^{-2}\bigl(T^2\|\partial_tF_\sigma\|^2_{L^2(\Omega_T)}+\|F_\sigma\|^2_{L^2(\Omega_T)}+T\|F_\sigma(\cdot,0)\|^2_{L^2(\Omega)}\bigr)\\
&\le C\sigma^{-2}(T^2+T+1)\|F_\sigma\|^2_{H^1(0,T;L^2(\Omega))}.
\end{split}
\end{equation}
Again recalling the expression of $F_\sigma$, given that $a_0,a_1,a_2,\varphi$ are sufficiently smooth functions on $\overline\Omega_T$, and $\phi\in C_c^\infty(\mathbb R^n)$, we can derive 
\begin{equation}
\|\beta_\sigma\|_{L^2(\Omega_T)}\le C\sigma^{-1}\|a_0\|_{C^5(\overline\Omega_T)}.
\end{equation}
By \eqref{eq:change-v-w} and \eqref{est:widetilde-Nt}, we can conclude that
\begin{equation}
\begin{split}
\|r_\sigma\|_{L^2(0,T;H^2(\Omega))}=\|\partial_tw_\sigma\|_{L^2(0,T;H^2(\Omega))}\le C\sigma^{-1}\|a_0\|_{C^5(\overline\Omega_T)}.
\end{split}
\end{equation}
Hence, the proof of Lemma \ref{lem:asympotic-r-sigma} is complete.
\end{proof}

\noindent {\bf GO solutions vanishing at $t=T$.} Noting that $\varphi(x,t)=2T-t+x\cdot\theta$ also solves the eikonal equation \eqref{eq-eikonal-var},  we can thus similarly construct GO solutions to the backward equation $\mathcal L_qy=0$ of the form
\begin{equation}
y(x,t)=e^{{-\rm i}\sigma(t+x\cdot\theta)}b(x,t)+e^{{-\rm i}\sigma(2T-t+x\cdot\theta)}d(x,t)+\tilde r_\sigma(x,t),
\end{equation}
where $b=b_0+\sigma^{-1}b_1+\sigma^{-2}b_2$ and $d=d_0+\sigma^{-1}d_1+\sigma^{-2}d_2$. The remainder term $\tilde r_\sigma$ satisfies the backward equation
\begin{equation}
\begin{cases}
\mathcal L_q\tilde r_\sigma=-e^{{-\rm i}\sigma t}\widetilde F_\sigma & {\rm in}\, \Omega_T,\\
\tilde r_\sigma=\Delta\tilde r_\sigma=0 & {\rm on}\, \Gamma_T,\\
\tilde r_\sigma(T)=\partial_t\tilde r_\sigma(T)=0 & {\rm in}\, \Omega.
\end{cases}
\end{equation}
where $\widetilde F_\sigma$ depends on $b_0,d_0,b_1,d_1,b_2,d_2$ and $\gamma,\theta,T$.
By some similar arguments to the construction of $a_0$, it is sufficient to solve the following transport equations
\begin{equation}\label{eq:transport-b-0}
\gamma b_{0t}+\nabla b_0\cdot\theta=0,\quad \gamma d_{0t}-\nabla d_0\cdot\theta=0,\quad \theta\in \mathbb S^{n-1}(\sqrt\gamma).
\end{equation} 
to get the leading terms $b_0$ and $d_0$, respectively. 

For our purpose, we choose $b_0=1$ and $d_0=-1$.
Similar to $a_1,a_2$, the functions $b_1$ and $d_1$ can be constructed by solving the following inhomogeneous transport equations
\begin{equation}
\gamma b_{1t}-\nabla b_1\cdot\theta=\frac{{\rm i}}{2\gamma},\quad \gamma d_{1t}+\nabla d_1\cdot\theta=\frac{-{\rm i}}{2\gamma},
\end{equation}
giving solutions $b_1=\frac{{\rm i}}{2\gamma^2}(t-T)$ and $d_1=\frac{{\rm i}}{2\gamma^2}(T-t)$. Likewise, $b_2$ and $d_2$ are obtained from the following inhomogeneous transport equations 
\begin{equation}
\gamma b_{2t}-\nabla b_2\cdot\theta=\frac{T-t}{4\gamma^3},\quad \gamma d_{2t}+\nabla d_2\cdot\theta=\frac{t-T}{4\gamma^3},
\end{equation}
giving solutions $b_2=-\frac{(t-T)^2}{8\gamma^4}$ and $d_2=\frac{(t-T)^2}{8\gamma^4}$.

Lastly, similar to Lemma \ref{lem:asympotic-r-sigma}, we can also prove the estimate
\begin{equation}\label{est:asy-tilde-r-sigma}
\begin{split}
&\sigma\|\tilde r_\sigma\|_{L^2(0,T;H^2(\Omega))}+\|\tilde r_\sigma\|_{L^2(0,T;H^3(\Omega))}+\|\partial_t\tilde r_\sigma\|_{L^2(0,T;H^2(\Omega))}\\
&\le C\bigl(\|b_0\|_{C^5(\overline\Omega_T)}+\|d_0\|_{C^5(\overline\Omega_T)} \bigr)
\end{split}
\end{equation}
for the remainder term $\tilde r_\sigma$. Therefore, we have constructed GO solutions
\begin{equation}\label{GO-solutions-vanish-T}
\begin{split}
y(x,t)&=e^{{-\rm i}\sigma(t+x\cdot\theta)}\bigl[1+\sigma^{-1}{\rm i}(2\gamma^2)^{-1}(t-T)-\sigma^{-2}(8\gamma^4)^{-1}(t-T)^2\bigr]\\
&\quad -e^{{-\rm i}\sigma(2T-t+x\cdot\theta)}\bigl[1+\sigma^{-1}{\rm i}(2\gamma^2)^{-1}(t-T)-\sigma^{-2}(8\gamma^4)^{-1}(t-T)^2\bigr]+\tilde r_\sigma,
\end{split}
\end{equation}
which vanish at $t=T$ (i.e., $y(\cdot,T)=0$ in $\Omega$).

\section{Proof of the main theorems}\label{sec:proof-thms}
This section mainly contains the proofs of the main theorems. Before proceeding, we present two important results from Section \ref{sec-obser-runge} that play a key role in recovering initial data for Kirchhoff plate equations.

\subsection{Observability inequality and Runge approximation}\label{sec-obser-runge}

We first recall an observability inequality associated with the exact controllability of the linear Kirchhoff plate equation.
Invoking \cite[Theorem 3]{zhang2006sharp} and the well-posedness results established in Theorem \ref{proposition-well-posedness-lin-u}, we state the following result.
\begin{lemma}[Observability inequality]\label{lem:controllability-ineq}
Let $T>T_0$ with $T_0$ as in \eqref{observe-time-T0}, and $\Gamma_0$ be given by \eqref{cond-for-Gamma-0}. Suppose that $u$ satisfies the following linear Kirchhoff plate equation:
\begin{equation}
\begin{cases}
(I-\gamma\Delta)\partial_t^2u+\Delta^2u+bu=0 & {\rm in}\, \Omega_T,\\
u=\Delta u=0 & {\rm on}\, \Gamma_T,\\
u(0)=\eta_1,\, \partial_tu(0)=\eta_2 & {\rm in}\, \Omega,
\end{cases}
\end{equation}
with 
$(\eta_1,\eta_2)\in \mathcal H_1\times \mathcal H_2$ (see \eqref{def:H1-H2}), and $b\in L^\infty(0,T;L^{\mathfrak p}(\Omega))$ with $\mathfrak p\in [\frac{5n}{2},\infty]$. Then there exists a positive constant $C$ that is independent of $\eta_1,\eta_2$, such that
\begin{equation}\label{est:observability}
\begin{split}
&\|\eta_1\|_{H^3(\Omega)}+\|\eta_2\|_{H^2(\Omega)}\\
&\le C\exp\bigl(C\|b\|^{\frac{2\mathfrak p}{6\mathfrak p-5n}}_{L^\infty(0,T;L^{\mathfrak p}(\Omega))}\bigr)\bigl( \|\partial_\nu u\|_{H^1(0,T;L^2(\Gamma_0))}+\|\partial_\nu\Delta u\|_{L^2(\Gamma_{0T})}\bigr).
\end{split}
\end{equation}
\end{lemma}
\begin{remark}
It follows from Theorem \ref{proposition-well-posedness-lin-u} that the right-hand side of \eqref{est:observability} is finite. We note that the space-regularity condition imposed on $b$ in the above lemma is more restrictive than the one required in the well-posedness arguments (see conditions \eqref{cond-for-tilde-q} for $q$ in Theorem \ref{proposition-well-posedness-lin-u}). 
\end{remark}
Based on the above observability inequality, by adopting some similar arguments to \cite[Theorem 5.1]{Lin-JLMS}, we next prove two generalized Runge approximation results in a smaller function space (compared to $L^2$-space) that have their independent interest, for Kirchhoff plate equations.
\begin{proposition}\label{proposi-Runge-Appro}
Let $T_0<t_1<t_2<T$, where $T_0$ is given by \eqref{observe-time-T0}. Assume that $b\in L^\infty(0,T;L^{\mathfrak p}(\Omega))$ with $\mathfrak p\in [\frac{5n}{2},\infty]$. 
Then for any solution 
$$v\in E_2^{t_1,t_2}\vcentcolon=C([t_1,t_2];H^2(\Omega))\cap C^1([t_1,t_2];H^1(\Omega))\cap H^2(t_1,t_2;L^2(\Omega))$$ to the equation 
\begin{equation}
(I-\gamma\Delta)\partial_t^2v+\Delta^2v+bv=0\quad {\rm in}\,\, \Omega_T,
\end{equation}
and any $\varepsilon>0$, there exists a solution $W\in E_2^T\cap H^2(0,T;L^2(\Omega))$ (the function space $E_2^T$ was introduced in Section \ref{subsec:well-posed-prelimi}) to 
\begin{equation}\label{eq:W-without-boundary-dta}
\begin{cases}
(I-\gamma\Delta)\partial_t^2W+\Delta^2W+bW=0 & {\rm in}\, \Omega_T,\\
W(0)=\partial_tW(0)=0 & {\rm in}\, \Omega,
\end{cases}
\end{equation}
such that 
\begin{equation}
\|W-v\|_{L^2(t_1,t_2;H^1(\Omega))}<\varepsilon.
\end{equation}
Moreover, let $T>2T_0$ and let $T_0<t_1<t_2<T-T_0$. Then there exists a solution $\widetilde W\in E_2^T\cap H^2(0,T;L^2(\Omega))$ to the backward equation
\begin{equation}\label{eq:back-W-without-boundary-data}
\begin{cases}
(I-\gamma\Delta)\partial_t^2\widetilde W+\Delta^2\widetilde W+b\widetilde W=0 & {\rm in}\, \Omega_T,\\
\widetilde W(T)=\partial_t\widetilde W(T)=0 & {\rm in}\, \Omega,
\end{cases}
\end{equation}
such that 
\begin{equation}
\|\widetilde W-v\|_{L^2(t_1,t_2;H^1(\Omega))}<\varepsilon.
\end{equation}
\end{proposition}
\begin{proof}
It suffices to prove the first approximation result, and the another one can be proved similarly.  It is equivalent to show the set
\begin{equation}
X\vcentcolon=\{w=W|_{\Omega\times (t_1,t_2)}: W\in E_2^T\cap H^2(0,T;L^2(\Omega))\, \text{is a solution to}\, \eqref{eq:W-without-boundary-dta} \}
\end{equation}
is dense in the set
\begin{equation}
Y\vcentcolon=\{v\in E_2^{t_1,t_2}: (I-\gamma\Delta)\partial_t^2v+\Delta^2v+bv=0\,\, \text{in}\,\, \Omega\times(t_1,t_2)\}
\end{equation}
in terms of $L^2(t_1,t_2;H^1(\Omega))$.
By the Hahn-Banach theorem, it suffices to prove, if $h\in L^2(t_1,t_2;(H^1(\Omega))')$ satisfies 
\begin{equation}
\int_{t_1}^{t_2} \langle h,w\rangle_{(H^1(\Omega))',H^1(\Omega)}\, dt=0,\quad \forall w\in X,
\end{equation}
then it hold that 
\begin{equation}
\int_{t_1}^{t_2}\langle h,v\rangle_{(H^1(\Omega))',H^1(\Omega)}\, dt=0,\quad \forall v\in Y.
\end{equation}

To this aim, we take a function 
\begin{equation}
H(x,t)=
\begin{cases}
h(x,t),\, & (x,t)\in\Omega\times (t_1,t_2),\\
0,\, & (x,t)\in \Omega\times ((0,t_1]\cup [t_2,T)).
\end{cases}
\end{equation}
Assume that $\tilde v$ satisfies the backward equation
\begin{equation}
\begin{cases}
(I-\gamma\Delta)\partial_t^2\tilde v+\Delta^2\tilde v+b\tilde v=H & {\rm in}\, \Omega_T,\\
\tilde v=\Delta\tilde v=0 & {\rm on}\, \Gamma_T,\\
\tilde v(T)=\partial_t\tilde v(T)=0 & {\rm in}\, \Omega.
\end{cases}
\end{equation}
Observing that $H\in L^2(0,T;H^{-1}(\Omega))$ and $b\in L^\infty(0,T;L^{\mathfrak p}(\Omega))$ with $\mathfrak p\in [\frac{5n}{2},\infty]$, by the time reversal $t\to T-t$ symmetry of the Kirchhoff plate equation, and Theorem \ref{proposition-well-posedness-lin-u}, we can obtain $\tilde v\in E^T_2\cap H^2(0,T;L^2(\Omega))$ with boundary traces 
$$\partial_\nu\partial_t\tilde v|_{\Gamma_T}\in L^2(\Gamma_T),\quad \partial_\nu\Delta\tilde v|_{\Gamma_T}\in H^{-2}(0,T;H^{-2}(\Gamma)).$$ 
Invoking the equation \eqref{eq:W-without-boundary-dta} and noting that $w=W|_{\Omega\times (t_1,t_2)}\in E_2^{t_1,t_2}$, we can get
\begin{equation}
\begin{split}
0&=\int_{t_1}^{t_2}\langle h,w\rangle_{(H^1(\Omega))',H^1(\Omega)}\, dt=\int_0^T\langle H,W\rangle_{(H^1(\Omega))',H^1(\Omega)}\, dt\\
&=\int_0^T\int_\Omega ((I-\gamma\Delta)\partial_t^2\tilde v+\Delta^2\tilde v+b\tilde v)W\, dxdt\\
&=\int_0^T\int_\Gamma[W\partial_\nu\Delta\tilde v+(\Delta W)(\partial_\nu\tilde v)+\gamma(\partial_tW)(\partial_\nu\partial_t\tilde v)]d\Gamma dt.
\end{split}
\end{equation}
If we choose boundary data $W|_{\Gamma_T}\in H^2_0(0,T;H^2(\Gamma))$, $\Delta W|_{\Gamma_T}\in H^2(0,T;L^2(\Gamma))$, then by Theorem \ref{proposition-well-posedness-lin-u}, we have $W\in E^T_2\cap H^2(0,T;L^2(0,T))$. Hence, the boundary term $\int_0^T\int_\Gamma W\partial_\nu\Delta\tilde v\, d\Gamma dt$ makes sense in the following way 
$$\int_0^T\int_\Gamma W\partial_\nu\Delta\tilde v\, d\Gamma dt=\langle \partial_\nu\Delta\tilde v, W\rangle_{H^{-2}(0,T;H^{-2}(\Gamma)),H_0^2(0,T;H^2(\Gamma))}.$$  
Since $W|_{\Gamma_T}$ and $\Delta W|_{\Gamma_T}$ can be arbitrary functions in $C_c^\infty(0,T;C^\infty(\Gamma))$, we can conclude that $\tilde v\in E^2$ satisfies 
\begin{equation}\label{eq:v-full-boundary-zero}
\begin{cases}
(I-\gamma\Delta)\partial_t^2\tilde v+\Delta^2\tilde v+b\tilde v=0 & {\rm in}\, \Omega\times ((0,t_1)\cup (t_2,T)),\\
\tilde v=\partial_\nu\tilde v=\Delta\tilde v=\partial_\nu\Delta\tilde v=0 & {\rm on}\, \Gamma_T,\\
\tilde v(T)=\partial_t\tilde v(T)=0 & {\rm in}\, \Omega.
\end{cases}
\end{equation}
By the observability inequality \eqref{est:observability}, we have $\tilde v=0$ in $\Omega\times (0,t_1)$. The uniqueness of solution to the equation \eqref{eq:v-full-boundary-zero} yields that $\tilde v=0$ in $\Omega\times (t_2,T)$. This implies 
\begin{equation}
\partial_t^k\tilde v(\cdot,l)=0\quad \text{in}\,\, \Omega,\quad \text{for}\,\, l=t_1,t_2,\,\, \text{and}\,\, k=0,1.
\end{equation}
Thus, it holds that
\begin{equation}
\int_{t_1}^{t_2}\langle h, v\rangle_{(H^1(\Omega))',H^1(\Omega)}dt=\int_{t_1}^{t_2}\int_\Omega((I-\gamma\Delta)\partial_t^2\tilde v+\Delta^2\tilde v+b\tilde v)v\, dxdt=0
\end{equation}
for any $v\in E^{t_1,t_2}_2$ satisfying the equation $(I-\gamma\Delta)\partial_t^2 v+\Delta^2v+bv=0$ in $\Omega\times (t_1,t_2)$ as desired. Therefore, the proof is finished.
\end{proof}

\subsection{Recovery of the initial data}\label{subsection-deter-initial}
In this subsection, we are concerned with the proof of Theorem \ref{thm:deter-initial-data} and giving some counterexamples related to the non-uniqueness.

\noindent {\bf Proof of Theorem \ref{thm:deter-initial-data}.} Recalling that $r$ and $r_0$ respectively satisfy \eqref{cond-for-r} and \eqref{cond-for-r0-IP}, for $n>6$, we have  $0<\frac{4}{5(n-6)}<\frac{4}{n-4}$. Thus, 
$f\in\mathcal M^T_{r_0,5n/2}$ implies that $f\in\mathcal M^T_{r,\mathfrak m}$ for some suitable $r$ and $\mathfrak m$ satisfying \eqref{cond-for-r} and \eqref{cond-for-m}. Since $(\eta_{1j},\eta_{2j})\in\mathcal H_1\times\mathcal H_2$, it then follows from Theorem \ref{thm:global-well-posed-H3} that the nonlinear equation \eqref{eq:nonlin-plate-homo-boundary} admits a unique solution $u_j\in E_3^T$ for each $j=1,2$. 
Let $w=u_1-u_2$. We see that
\begin{equation}\label{eq:proof-w=u1-u2}
\begin{cases}
(I-\gamma\Delta)\partial_t^2w+\Delta^2w-q_fw=0 & {\rm in}\, \Omega_T,\\
w=\Delta w=0 & {\rm on}\, \Gamma_T,\\
w(0)=\eta_{11}-\eta_{12},\, \partial_tw(0)=\eta_{21}-\eta_{22} & {\rm in}\, \Omega,
\end{cases}
\end{equation}
where 
\begin{equation}
q_f(x,t)=\int_0^1\partial_\tau f(x,su_1(x,t)+(1-s)u_2(x,t))\, ds.
\end{equation}
We are required to show that $q_f\in L^\infty(0,T;L^{\frac{5n}{2}}(\Omega))$.   

Recalling the condition \eqref{cond:for-partial-tau-f} for $f\in \mathcal M^T_{r_0,5n/2}$, we have
\begin{equation}
|q_f|\le C\bigl(|a(x)|+|u_1|^{r_0}+|u_2|^{r_0}\bigr)\quad \text{in}\,\, \Omega_T.
\end{equation}
Using Minkowski's inequality, we can get
\begin{equation}
\|q_f\|_{L^{5n/2}(\Omega)}\le C\bigl(\|a\|_{L^{5n/2}(\Omega)}+\||u_1|^{r_0}+|u_2|^{r_0}\|_{L^{5n/2}(\Omega)} \bigr).
\end{equation}
For the case $n>6$ and $r_0\ge \frac{2}{5n}$, we have $1\le \frac{5n}{2}r_0\le \frac{2n}{n-6}$ and
\begin{equation}\label{est:I-u1-u2}
\begin{split}
I_{u_1,u_2}\vcentcolon&=\||u_1|^{r_0}+|u_2|^{r_0}\|_{L^{5n/2}(\Omega)}\le \|u_1\|^{r_0}_{L^{\frac{5n}{2}r_0}(\Omega)}+\|u_2\|^{r_0}_{L^{\frac{5n}{2}r_0}(\Omega)}\\
&\le C\bigl( \|u_1\|^{r_0}_{H^3(\Omega)}+\|u_2\|^{r_0}_{H^3(\Omega)}\bigr).
\end{split}
\end{equation}
If $r_0=0$, then it is clear that the above inequality still holds. So we consider $0<r_0<2/5n$, we can choose $z\ge 1$ such that $1\le r_0z\le 2$. This implies $5n/2<1/r_0\le z$ and thus, we can obtain
\begin{equation}
I_{u_1,u_2}\le \|u_1\|^{r_0}_{L^{r_0z}(\Omega)}+\|u_2\|^{r_0}_{L^{r_0z}(\Omega)}\le C\bigl( \|u_1\|^{r_0}_{L^2(\Omega)}+\|u_2\|^{r_0}_{L^2(\Omega)} \bigr).
\end{equation}
For the critical cases $n=6$, since $r_0\in [0,4/(n-4)]\subset[0,\infty)$ and $H^3(\Omega)\hookrightarrow L^d(\Omega)$ for any $2\le  d<\infty$, we can argue similarly to achieve \eqref{est:I-u1-u2}.
For the remainder cases $1\le n\le 5$ simple since we have the Sobolev embedding $H^3(\Omega)\hookrightarrow L^\infty(\Omega)$.
Hence, we deduce that $q_f\in L^\infty(0,T;L^{5n/2}(\Omega))\subset L^\infty(0,T;L^{p_0}(\Omega))$ for some $p_0$ satisfying \eqref{cond-for-tilde-q}, and by applying the second assertion of Theorem \ref{proposition-well-posedness-lin-u} to the equation \eqref{eq:proof-w=u1-u2}, we have 
$$\partial_\nu w|_{\Gamma_T}\in H^1(0,T;L^2(\Gamma)),\quad \partial_\nu\Delta w|_{\Gamma_T}\in L^2(\Gamma_T).$$
Finally, using Lemma \ref{lem:controllability-ineq}, we can obtain
\begin{equation}
\begin{split}
&\|\eta_{11}-\eta_{12}\|_{H^3(\Omega)}+\|\eta_{21}-\eta_{22}\|_{H^2(\Omega)}\\
&\le C\exp\bigl(C\|q_f\|^{\frac{2\mathfrak p}{6\mathfrak p-5n}}_{L^\infty(0,T;L^{\mathfrak p}(\Omega))}\bigr)\bigl( \|\partial_\nu w\|_{H^1(0,T;L^2(\Gamma_0))}+\|\partial_\nu\Delta w\|_{L^2(\Gamma_{0T})}\bigr).
\end{split}
\end{equation}
Therefore, the  estimate \eqref{est:thm1-deter-initial-data} follows, and the proof of Theorem \ref{thm:deter-initial-data} is finished.  \hfill $\square$
\medskip

\noindent {\bf Non-uniqueness of recovering the initial data.} Let $u$ solve the following linear Kirchhoff plate equation with sources $h(x,t)$:
\begin{equation}\label{eq:non-uniqueness-sources}
\begin{cases}
(I-\gamma\Delta)\partial_t^2u+\Delta^2u+qu=h & {\rm in}\, \Omega_T,\\
u=\Delta u=0 & {\rm on}\, \Gamma_T,\\
u(0)=\eta_1,\, \partial_tu(0)=\eta_2 & {\rm in}\, \Omega.
\end{cases}
\end{equation}
We next show that, if $h$ is unknown, then the passive measurement $\Lambda^{0,\Gamma_0}_{\eta_1,\eta_2,h}$ is unable to recover the initial data $\eta_1,\eta_2$. Hence, with unknown sources, a natural obstacle arises when one attempts to recover initial conditions using only boundary passive measurements. To see this, let $u_j$ denote the solution to \eqref{eq:non-uniqueness-sources} with respect to $h_j$ and $\eta_{1j},\eta_{2j}$ for $j=1,2$. We choose $u_j\in C_c^\infty([0,T);C_c^\infty(\Omega))$ such that
$u_1\ne u_2$ in $(\Omega\backslash\Omega_\epsilon)\times [0,T]$, and $u_1=u_2$ in $\Omega_\epsilon\times [0,T]$ for $k=0,1$, where $$\Omega_\epsilon\vcentcolon=\{x\in\Omega: \text{dist}(x,\Gamma)<\epsilon\}\subset\Omega$$
for some sufficiently small constant $\epsilon>0$.
Then, we find that $u_j=\Delta u_j=0$ on $\Gamma_T$, and $\partial_t^ku_1(\cdot,0)\ne \partial_t^ku_2(\cdot,0)$ in $\Omega\backslash\Omega_\epsilon$ for $k=0,1$.
We set
\begin{equation}
h_j(x,t)=(I-\gamma\Delta)\partial_t^2u_j+\Delta^2u_j+qu_j\quad \text{in}\,\, \Omega_T,\,\, j=1,2.
\end{equation}
Thus, it is clear that $(\eta_{11},\eta_{21},h_1)\ne (\eta_{12},\eta_{22},h_2)$, but $\Lambda^{0,\Gamma_0}_{\eta_{11},\eta_{12},h_1}=\Lambda^{0,\Gamma_0}_{\eta_{12},\eta_{22},h_2}=0$.

Let $m\in\mathbb N_{\ge 1}$. We turn our attention to considering the nonlinear equation
\begin{equation}\label{eq:with-qu-h-et}
(I-\gamma\Delta)\partial_t^2u+\Delta^2u=e^{(1-m)t}\alpha u^m+e^th(x)\quad \text{in}\,\, \Omega_T.
\end{equation}
We choose a non-negative function $\varphi\in C^\infty(\overline\Omega)$ satisfying 
\begin{equation}\label{cond-for-varphi}
\begin{cases}
\varphi>0 & {\rm in}\, \Omega\backslash\Omega_{2\epsilon},\\
\varphi=0 & {\rm in}\, \overline\Omega_\epsilon.
\end{cases}
\end{equation}
The functions $\alpha$ and $h$ are respectively given by
\begin{equation}
\begin{cases}
\alpha=\frac{\varphi-\gamma\Delta\varphi+\Delta^2\varphi}{\varphi^m} & {\rm in}\, \Omega\backslash\Omega_{2\epsilon},\\
\alpha=0 & {\rm in}\, \Omega_{2\epsilon},
\end{cases}
\end{equation}
\begin{equation}
\begin{cases}
h=\varphi-\gamma\Delta\varphi+\Delta^2\varphi & {\rm in}\, \Omega_{2\epsilon},\\
h=0 & {\rm in}\,  (\Omega\backslash \Omega_{2\epsilon})\cup\overline\Omega_\epsilon.
\end{cases}
\end{equation}
Let $u=e^t\varphi$. Then we find that $u$ satisfies the equation \eqref{eq:with-qu-h-et} with $u=\Delta u=0$ on $\Gamma_T$, and $\eta_1=u(0)=\varphi$, $\eta_2=\partial_tu(0)=\varphi$ in $\Omega$. 

Next, we choose non-negative functions $\varphi_1,\varphi_2\in C^\infty(\overline\Omega)$ such that they satisfy \eqref{cond-for-varphi} and $\varphi_1=\varphi_2$ in $\overline\Omega_{2\epsilon}$, $\varphi_1\ne\varphi_2$ in $\Omega\backslash\overline\Omega_{2\epsilon}$. Thus, we have $h_1=h_2$ in $\Omega$. Let $u_j(x,t)=e^t\varphi_j(x)$ be solutions to \eqref{eq:with-qu-h-et} with respect to $\alpha_j$, $\eta_{1j},\eta_{2j}$ and $h$, and let $\Lambda^{0,\Gamma_0}_{\eta_{1j},\eta_{2j},\alpha_j,h}$ be the corresponding passive measurement for $j=1,2$. Hence, we can get $\Lambda^{0,\Gamma_0}_{\eta_{11},\eta_{12},\alpha_1,h}=\Lambda^{0,\Gamma_0}_{\eta_{12},\eta_{22},\alpha_2,h}=0$, but
\begin{equation}
(\eta_{11},\eta_{21},\alpha_1)\ne (\eta_{12},\eta_{22},\alpha_2)\quad \text{in}\,\, \Omega.
\end{equation}
This example implies that, even with known external source terms, different potential coefficients precludes the passive measurement $\Lambda^{0,\Gamma_0}_{\eta_1,\eta_2,\alpha,h}$ from recovering the initial data.
So it remains an open question of simultaneously recovering unknown sources (or coefficients) and initial data for linear Kirchhoff plate equations by passive measurement. However, as we have stated in Theorem \ref{thm:deter-ini-da-nonlin}, for the nonlinear Kirchhoff plate equation \eqref{eq:intro-non-lin-plate-g} with nonlinearities $g$, we can simultaneously recover initial data and certain information of $g$ by using the active measurement $\Lambda^{\Gamma}_{\eta_1,\eta_2,g}$. 
\begin{remark}
It is clear that the simultaneous recovery of general initial data and coefficients for evolutionary PDEs using a single passive boundary measurement is generally considered unlikely, even in the one-dimensional case \cite{feizmohammadi2025reconstruction}, unless some \emph{a priori} conditions are imposed on either the source terms or the coefficients.
\end{remark}

\subsection{Recovery of the linear coefficient $q$}
Based on the GO solutions for linear Kirchhoff plate equations constructed in Section \ref{Construction-GO-solution}, we give the proof of the first assertion in Theorem \ref{thm:deter-potential-coeffi}.

\noindent {\bf Proof of Theorem \ref{thm:deter-potential-coeffi}.} Let $\tilde q=q_2-q_1$ in $\Omega_T$. We extend $\tilde q$ from $\Omega_T$ to $\mathbb R^{n+1}$ by zero and still denote it by $\tilde q$. For $j=1,2$, let $u_j$ satisfy the following linear Kirchhoff plate equation
\begin{equation}
\begin{cases}
(I-\gamma\Delta)\partial_t^2u_j+\Delta^2u_j+q_ju_j=0 & {\rm in}\, \Omega_T\\
u_j=h_1,\, \Delta u_j=h_2 & {\rm on}\, \Gamma_T,\\
u_j(0)=\eta_1,\, \partial_tu_j(0)=\eta_2 & {\rm in}\, \Omega.
\end{cases}
\end{equation}
The initial boundary data $\eta_j$ and $h_j$ ($j=1,2$) are known functions giving by
\begin{equation}
\eta_j(x)=\partial_t^{j-1}[e^{{\rm i}\sigma\varphi}(a_0+\sigma^{-1}a_1+\sigma^{-2}a_2)](x,0)\quad {\rm in}\, \Omega,
\end{equation}
\begin{equation}
h_1=e^{{\rm i}\sigma\varphi}[a_0+\sigma^{-1}a_1+\sigma^{-2}a_2],\,\, h_2=\Delta[e^{{\rm i}\sigma\varphi}(a_0+\sigma^{-1}a_1+\sigma^{-2}a_2)]\quad {\rm on}\, \Gamma_T,
\end{equation}
where the functions $\varphi,a_0,a_1,a_2$ were constructed in Section \ref{Construction-GO-solution} (see \eqref{expre-phase-var}, \eqref{expre-ampli-a0}--\eqref{expre-ampli-a2}).
Recalling that $q_j\in L^\infty(0,T;L^{p_0}(\Omega))$ with $p_0$ satisfying \eqref{cond-for-tilde-q}, it then follows from Theorem \ref{proposition-well-posedness-lin-u} that $u_j\in E_3^T\cap H^2(0,T;H^1(\Omega))$.
Let $w=u_1-u_2$. We see that $w$ satisfies the equation
\begin{equation}\label{eq:sys-lin-w=u1-u2-Thm2}
\begin{cases}
(I-\gamma\Delta)\partial_t^2w+\Delta^2w+q_2w=\tilde qu_1 & {\rm in}\, \Omega_T\\
w=\Delta w=0 & {\rm on}\, \Gamma_T,\\
w(0)=\partial_tw(0)=0  & {\rm in}\, \Omega.
\end{cases}
\end{equation}
Let $y$ satisfy the following backward equation
\begin{equation}\label{eq:auxi-back-syst-y-T}
\begin{cases}
(I-\gamma\Delta)\partial_t^2y+\Delta^2y+q_2y=0 & {\rm in}\, \Omega_T\\
y=\tilde h_1,\, \Delta u_j=\tilde h_2 & {\rm on}\, \Gamma_T,\\
y(T)=0,\, \partial_ty(T)=y_T & {\rm in}\, \Omega.
\end{cases}
\end{equation}
As we have discussed in Section \ref{Construction-GO-solution}, we can construct GO solutions \eqref{GO-solutions-vanish-T} to the backward equation \eqref{eq:auxi-back-syst-y-T}. Hence, $\tilde h_1,\tilde h_2$ and $y_T$ can be constructed correspondingly.

Using the condition $\Lambda^{\Gamma,T}_{q_1}(h_1,h_2,\eta_1,\eta_2)=\Lambda^{\Gamma,T}_{q_2}(h_1,h_2,\eta_1,\eta_2)$, we have
\begin{equation}
w=\partial_\nu w=\Delta w=\partial_\nu\Delta w=0\quad \text{on}\,\, \Gamma_T\quad  \text{and}\quad w(T)=0\quad \text{in}\,\, \Omega.
\end{equation}
Multiplying $y$ to the first equation in \eqref{eq:sys-lin-w=u1-u2-Thm2} and applying Green's formula, we can obtain
\begin{equation}\label{int:q1-q2=u1y=0}
\int_0^T\int_\Omega\tilde qu_1y\, dxdt=0.
\end{equation}
We insert the GO solution \eqref{GO-solutions-vanish-T} and
\begin{equation}\label{GO-solution-u-1-thm2.1}
u_1(x,t)=e^{{\rm i}\sigma\varphi}[a_0(x,t)+\sigma^{-1}a_1(x,t)+\sigma^{-2}a_2(x,t)]+r_\sigma
\end{equation}
into \eqref{int:q1-q2=u1y=0} to get
\begin{equation}
\int_0^T\int_\Omega\tilde q(x,t)a_0b_0\, dxdt+I(x,t)+\mathcal O(\sigma^{-1})=0,
\end{equation}
where 
\begin{equation}
\begin{split}
I(x,t)\vcentcolon&=\int_0^T\int_\Omega\tilde q\bigl[e^{{\rm i}\sigma(x\cdot\theta +t)}a_0\tilde r_\sigma+ e^{-{\rm i}\sigma(x\cdot\theta +t)}d_0r_\sigma +e^{2{\rm i}\sigma(t-T)}a_0d_0\\
&\qquad\qquad\quad + e^{-{\rm i}\sigma(2T-t+x\cdot\theta)}d_0r_\sigma+r_\sigma\tilde r_\sigma \bigr]\, dxdt
\end{split}
\end{equation}
for all $\sigma>1$. Recalling that $a_0=e^{-{\rm i}(x,t)\cdot\zeta}$, $b_0=1,d_0=-1$, and $a_1,a_2,b_1,b_2,d_1,d_2$ are all smooth function on $\overline\Omega_T$, applying the Riemann-Lebesgue lemma, we can get
\begin{equation}
\Big|\int_0^T\int_\Omega\tilde qe^{2{\rm i}\sigma(t-T)}a_0d_0\, dxdt\Big|\to 0
\end{equation}
for $\tilde q\in L^\infty(0,T;L^{p_0}(\Omega))\subset L^2(\Omega_T)$ as $\sigma\to \infty$. Moreover, by the estimates \eqref{est-r-sigma-asy} and \eqref{est:asy-tilde-r-sigma} for $r_\sigma$ and $\tilde r_\sigma$, we have
\begin{equation}
\Big|\int_0^T\int_\Omega\tilde q \bigl[e^{{\rm i}\sigma(x\cdot\theta +t)}a_0\tilde r_\sigma+ e^{-{\rm i}\sigma(x\cdot\theta +t)}d_0r_\sigma \bigr]\, dxdt\Big|\to 0,
\end{equation}
and 
\begin{equation}
\begin{split}
\Big| \int_0^T\int_\Omega\tilde qr_\sigma\tilde r_\sigma\, dxdt\Big|&\le \|qr_\sigma\|_{L^2(\Omega_T)}\|\tilde r_\sigma\|_{L^2(\Omega_T)}\\
&\le C\|\tilde q\|_{L^\infty(0,T;L^{p_0}(\Omega))}\|r_\sigma\|_{L^2(0,T;H^3(\Omega))}\|\tilde r_\sigma\|_{L^2(\Omega_T)}\to 0
\end{split}
\end{equation}
as $\sigma\to\infty$. It then follows that 
\begin{equation}
\int_0^T\int_\Omega\tilde q(x,t)e^{-{\rm i}(x,t)\cdot\zeta}\, dxdt=0.
\end{equation}
Since $\tilde q=0$ in $\mathbb R^{n+1}\backslash\Omega_T$, this implies that the Fourier transform 
$$\mathcal F\tilde q(\zeta)\vcentcolon=\int_{\mathbb R^{n+1}}\tilde q^{-{\rm i}(x,t)\cdot\zeta}\, dxdt$$
of $\tilde q\in L^1(\mathbb R^{n+1})$ vanishes for all $\zeta$ lying in the hyperplane $\{\omega\in\mathbb R^{n+1}: \omega\cdot (1,-\tilde\theta)=0\}$, where $\tilde\theta=\gamma^{-1}\theta\in\mathbb S^{n-1}(\gamma^{-\frac12})$ is chosen arbitrarily. On the other hand, since $\tilde q$ is compactly supported in $\overline\Omega_T$, we know that $\mathcal F\tilde q$ is a complex valued analytic function and it follows that $\mathcal F\tilde q=0$. By inverse Fourier transform, this implies $\tilde q=0$, and thus $q_1=q_2$ in $\Omega_T$. Hence, the \emph{first assertion} in Theorem \ref{thm:deter-potential-coeffi} is proved.
 

We proceed to  the second result in Theorem \ref{thm:deter-potential-coeffi} by first considering the following equation
\begin{equation}
\begin{cases}
(I-\gamma\Delta)\partial_t^2u_j+\Delta^2u_j+q_ju_j=0 & {\rm in}\, \Omega_T\\
u_j=h_1,\, \Delta u_j=h_2 & {\rm on}\, \Gamma_T,\\
u_j(0)=\eta_{1j},\, \partial_tu_j(0)=\eta_{2j} & {\rm in}\, \Omega,
\end{cases}
\end{equation}
for $j=1,2$.
Let $h_1=h_2=0$, and let $\tilde u_j$ satisfy the equation
\begin{equation}\label{eq:sys-lin-tilde-u-j-qj}
\begin{cases}
(I-\gamma\Delta)\partial_t^2\tilde u_j+\Delta^2\tilde u_j+q_j\tilde u_j=0 & {\rm in}\, \Omega_T\\
\tilde u_j=\Delta\tilde u_j=0 & {\rm on}\, \Gamma_T,\\
\tilde u_j(0)=\eta_{1j},\, \tilde\partial_tu_j(0)=\eta_{2j} & {\rm in}\, \Omega.
\end{cases}
\end{equation}
Let $w_j=u_j-\tilde u_j$ for $j=1,2$. We see that
\begin{equation}\label{eq:w-j-zero-initial-data}
\begin{cases}
(I-\gamma\Delta)\partial_t^2w_j+\Delta^2w_j+q_jw_j=0 & {\rm in}\, \Omega_T\\
w_j=h_1,\, \Delta w_j=h_2 & {\rm on}\, \Gamma_T,\\
w_j(0)=\partial_tw_j(0)=0 & {\rm in}\, \Omega.
\end{cases}
\end{equation}
By the condition 
\begin{equation}
\Lambda^{\Gamma,T}_{q_1,\eta_{11},\eta_{21}}(h_1,h_2)=\Lambda^{\Gamma,T}_{q_2,\eta_{12},\eta_{22}}(h_1,h_2),\quad \forall (h_1,h_2)\in \mathcal H(\Gamma_T),
\end{equation}
we have $u_1(T)=u_2(T)$, $\tilde u_1(T)=\tilde u_2(T)$, and 
$$(\partial_\nu u_1,\partial_\nu\Delta u)|_{\Gamma_T}=(\partial_\nu u_2,\partial_\nu\Delta u_2)|_{\Gamma_T},\, (\partial_\nu \tilde u_1,\partial_\nu\Delta \tilde u_1)|_{\Gamma_T}=(\partial_\nu \tilde u_2,\partial_\nu\Delta \tilde u_2)|_{\Gamma_T}.$$

Let $\tilde w=w_1-w_2=(u_1-\tilde u_1)-(u_2-\tilde u_2)$. We have
\begin{equation}\label{eq:tildew=w1-w2-thm2.2}
\begin{cases}
(I-\gamma\Delta)\partial_t^2\tilde w+\Delta^2\tilde w+q_2\tilde w=\tilde qw_1 & {\rm in}\, \Omega_T\\
\tilde w=\Delta \tilde w=\partial_\nu\tilde w=\partial_\nu\Delta\tilde w=0 & {\rm on}\, \Gamma_T,\\
\tilde w(0)=\partial_t\tilde w(0)=\tilde w(T)=0 & {\rm in}\, \Omega.
\end{cases}
\end{equation}
Again multiplying $y$ to the equation \eqref{eq:tildew=w1-w2-thm2.2} and recalling that $q_1,q_2\in\mathcal U_{q_0}^{t_1,t_2}$ (see \eqref{set:for-q-q-0} for the set $\mathcal U_{q_0}^{t_1,t_2}$), we have ${\rm supp}\, \tilde q\subset\overline\Omega\times[t_1,t_2]$, and thus
\begin{equation}\label{int:equa-qw1-y}
\int_{t_1}^{t_2}\int_\Omega\tilde qw_1y\, dxdt=\int_0^T\int_\Omega\tilde qw_1y\, dxdt=0,
\end{equation}
where $y$ is the solution to \eqref{eq:auxi-back-syst-y-T}.
Noting that $w_1$ and $\partial_tw_1$ vanish at $t=0$, we can not directly use GO solutions to the equation $\mathcal L_{q_1}w_1=0$ of the form \eqref{GO-solution-u-1-thm2.1}. So we can apply the Runge approximation in Proposition \ref{proposi-Runge-Appro}. Indeed, we construct GO solutions to the free equation
\begin{equation}
\mathcal L_{q_1}{\bf w_1}=0\quad \text{in}\,\, \Omega\times (t_1,t_2)
\end{equation}
of the form 
\begin{equation}
{\bf w_1}(x,t)=e^{{\rm i}\sigma\varphi}[a_0(x,t)+\sigma^{-1}a_1(x,t)+\sigma^{-2}a_2(x,t)]+r_\sigma,
\end{equation}
where $\varphi,a_0,a_1,a_2$ are the same as in \eqref{GO-solution-u-1-thm2.1}.
Recalling that $q_j\in\mathcal U_{q_0}^{t_1,t_2}$ (i.e., ${\rm supp}\,(q_1-q_2)\subset\overline\Omega\times[t_1,t_2]$), by Proposition \ref{proposi-Runge-Appro}, there exists a sequence of complex-valued functions $\{w_k^1\}_{k\in\mathbb N}$ such that, for each $k\in\mathbb N$, $w_k^1\in E^T_2\cap H^2(0,T;L^2(\Omega))$ is the solution to
\begin{equation}
\begin{cases}
(I-\gamma\Delta)\partial_t^2 w_k^1+\Delta^2 w_k^1+q_1 w_k^1=0 & {\rm in}\, \Omega_T\\
w_k^1(0)=\partial_tw_k^1(0)=0 & {\rm in}\, \Omega,
\end{cases}
\end{equation}
such that $w_k^1\to {\bf w_1}$ in $L^2(t_1,t_2;H^1(\Omega))$ as $k\to \infty$.

Choosing $w_k^1=w_1$ in \eqref{int:equa-qw1-y} and noting that ${\rm supp}\,\tilde q\subset\overline\Omega\times [t_1,t_2]$, we can get
\begin{equation}
\begin{split}
\int_{t_1}^{t_2}\int_\Omega\tilde q{\bf w_1}y\, dxdt&=\int_{t_1}^{t_2}\int_\Omega\tilde q({\bf w_1}-w_k^1)y\, dxdt+\int_{t_1}^{t_2}\int_\Omega\tilde qw_k^1y\, dxdt\\
&\le \|\tilde q\|_{L^\infty(0,T;L^{p_0}(\Omega))}\|y\|_{L^2(0,T;H^3(\Omega))}\|{\bf w_1}-w_k^1\|_{L^2(\Omega\times(t_1,t_2))}\\
&\to 0,\quad \text{as}\,\, k\to \infty.
\end{split}
\end{equation}
Taking GO solutions \eqref{GO-solutions-vanish-T} for \eqref{eq:auxi-back-syst-y-T}, we can argue similarly to show that $\tilde q=0$. This together with the equation \eqref{eq:sys-lin-tilde-u-j-qj} and Theorem \ref{thm:deter-initial-data} imply that $\eta_{11}-\eta_{12}=\eta_{21}-\eta_{22}=0$ in $\Omega$ as desired. 

To prove the last assertion in Theorem \ref{thm:deter-potential-coeffi}, we choose GO solutions to the free equation
\begin{equation}
\mathcal L_{q_2}{\bf y}=0\quad \text{in}\,\, \Omega\times(t_1,t_2)
\end{equation}
of the form 
\begin{equation}
{\bf y}(x,t)=e^{-{\rm i}\sigma\varphi}[b_0(x,t)+\sigma^{-1}b_1(x,t)+\sigma^{-2}b_2(x,t)]+\tilde r_\sigma(x,t).
\end{equation}
Here the phase function $\varphi=x\cdot\theta+t$ is the same to that in ${\bf w_1}$. Recalling the amplitude function $b_0$ satisfies \eqref{eq:transport-b-0}, we can choose $b_0=1$. For the moment, using the assumption $T>2T_0$ and $T^*<t_1<t_2<T-T_0$, again applying Proposition \ref{proposi-Runge-Appro}, there exists a sequence of complex-valued functions $\{y_k\}_{k\in\mathbb N}$ such that, for each $k\in\mathbb N$, $y_k\in E^T_2\cap H^2(0,T;L^2(\Omega))$ is the solution to
\begin{equation}
\begin{cases}
(I-\gamma\Delta)\partial_t^2 y_k+\Delta^2 y_k+q_2y_k=0 & {\rm in}\, \Omega_T\\
y_k(T)=\partial_ty_k(T)=0 & {\rm in}\, \Omega,
\end{cases}
\end{equation}
such that $y_k\to {\bf y}$ in $L^2(t_1,t_2;H^1(\Omega))$ as $k\to \infty$. Applying the approximation property, we can derive that
\begin{equation}
\begin{split}
\int_{t_1}^{t_2}\int_\Omega\tilde q{\bf w_1}{\bf y}\, dxdt=0.
\end{split}
\end{equation}
Using GO solutions ${\bf w_1}$ and ${\bf y}$, we can argue similarly to conclude that $\tilde q=0$.
Therefore, the proof of Theorem \ref{thm:deter-potential-coeffi} is finished. \hfill $\square$

\subsection{Simultaneous recovery of initial data and nonlinearities}
In this section, we give a proof of Theorem \ref{thm:deter-ini-da-nonlin}.  The main approach is the usage of higher-order linearization around non-zero solutions together with Runge approximation and suitable GO solutions. Our arguments are inspired by the work \cite{Lin-JLMS} for inverse nonlinear wave equations. 

Let $M\in\mathbb N$, and let $\varepsilon=(\varepsilon_1,\cdots,\varepsilon_M)\in\mathbb R^m$ with $|\varepsilon|=|\varepsilon_1|+\cdots+|\varepsilon_M|$. Assume that $(h_{1k},h_{2k})\in\mathcal H_\Gamma^T$ (see \eqref{function-space-initial-boundary} for the definition of $\mathcal H_\Gamma^T$) for each $k=1,\cdots, M$.
For the nonlinear equation \eqref{eq:intro-non-lin-plate-g}, we introduce the following boundary data
\begin{equation}\label{boundary-input-h-varep}
h_1(x,t;\varepsilon)=\sum_{k=1}^M\varepsilon_kh_{1k}(x,t),\quad h_2(x,t;\varepsilon)=\sum_{k=1}^M\varepsilon_kh_{2k}(x,t),\quad {\rm on}\,\,\Gamma_T.
\end{equation}
We choose $|\varepsilon|>0$ sufficiently small such that
\begin{equation}
\|(h_1,h_2)\|_{\mathcal H_\Gamma^T}\le \sum_{k=1}^M|\varepsilon_k|\|(h_{1k},h_{2k})\|_{\mathcal H_\Gamma^T}<\frac\epsilon2.
\end{equation}
Let $T>T_0$ and let $u_j=u_j(x,t;\varepsilon)$ satisfy 
\begin{equation}\label{eq:non-prove-thm-non-initial}
\begin{cases}
(I-\gamma\Delta)\partial_t^2u_j+\Delta^2u_j=g_j(x,t,u_j) & {\rm in}\, \Omega_T,\\
u_j=\sum_{k=1}^M\varepsilon_kh_{1k}(x,t),\, \Delta u_j=\sum_{k=1}^M\varepsilon_kh_{2k}(x,t) & {\rm on}\, \Gamma_T,\\
u_j(0)=\eta_{1j},\, \partial_tu_j(0)=\eta_{2j} & {\rm in}\, \Omega,
\end{cases}
\end{equation}
where $g_j\in\mathcal G^{t_1,t_2}_{g_0}$ (see \eqref{set:admi-nonlin-g} for the set $\mathcal G^{t_1,t_2}_{g_0}$),  and 
the initial data $(\eta_{1j},\eta_{2j})\in\mathcal S_{\epsilon}(\Omega)$ (see \eqref{set:initial-data-small} for the definition of $\mathcal S_\epsilon(\Omega)$) for $j=1,2$. Applying Theorem \ref{thm-local-well-posedness}, the nonlinear equation \eqref{eq:non-prove-thm-non-initial} is locally well-posed such that $u_j\in E^T_3$ for each $j=1,2$. 

We are now  a position to prove Theorem \ref{thm:deter-ini-da-nonlin}.

\noindent{\bf Proof of Theorem \ref{thm:deter-ini-da-nonlin}.} We divide the proof into four steps.

{\bf Step 1.} \emph{First-order linearization $(M=1)$.} When $\varepsilon=0$, let $w_j(x,t)=u_j(x,t;0)$ be solutions to the following equations
\begin{equation}\label{eq:u-varep=0-w}
\begin{cases}
(I-\gamma\Delta)\partial_t^2w_j+\Delta^2w_j=g_j(x,t,w_j) & {\rm in}\, \Omega_T,\\
w_j=\Delta w_j=0 & {\rm on}\, \Gamma_T,\\
w_j(0)=\eta_{1j},\, \partial_tw_j(0)=\eta_{2j} & {\rm in}\, \Omega,
\end{cases}
\end{equation}
for $j=1,2$. We linearize \eqref{eq:non-prove-thm-non-initial} around $w_j$ for $j=1,2$. We see that
\begin{equation}
v_j^{(k)}(x,t)\vcentcolon=\partial_{\varepsilon_k}\big|_{\varepsilon=0}u_j(x,t;\varepsilon)=\lim_{\varepsilon_k\to 0}\frac{u_j(x,t;\varepsilon)-w_j(x,t)}{\varepsilon_k}\quad {\rm in}\,\, \Omega_T,
\end{equation}
satisfies the following linear equation
\begin{equation}\label{eq:v-j-first-order-lin}
\begin{cases}
(I-\gamma\Delta)\partial_t^2v_j^{(k)}+\Delta^2v_j^{(k)}-\tilde q_jv_j^{(k)}=0 & {\rm in}\, \Omega_T,\vspace{0.8ex}\\
v_j^{(k)}=h_{1k},\, \Delta v_j^{(k)}=h_{2k}& {\rm on}\, \Gamma_T,\vspace{0.8ex}\\
v_j^{(k)}(0)=\partial_tv_j^{(k)}(0)=0 & {\rm in}\, \Omega,
\end{cases}
\end{equation}
for each $j=1,2$, where $\tilde q_j=\partial_zg_j(x,t,w_j)\in L^\infty(0,T;L^{\frac{5n}{2}}(\Omega))$. Let $\tilde v^{(k)}=v_1^{(k)}-v_2^{(k)}$. We find that $v^{(k)}$ satisfies 
\begin{equation}\label{eq:lin-v(k)}
\begin{cases}
(I-\gamma\Delta)\partial_t^2\tilde v^{(k)}+\Delta^2\tilde v^{(k)}-\tilde q_2\tilde v^{(k)}=(\tilde q_1-\tilde q_2)v_1^{(k)} & {\rm in}\, \Omega_T,\\
\tilde v^{(k)}=\Delta \tilde v^{(k)}=0 & {\rm on}\, \Gamma_T,\\
\tilde v^{(k)}(0)=\partial_t\tilde v^{(k)}(0)=0 & {\rm in}\, \Omega.
\end{cases}
\end{equation}

Let $y$ be the solution to the following backward equation
\begin{equation}\label{eq:back-y-prove-thm3}
\begin{cases}
(I-\gamma\Delta)\partial_t^2y+\Delta^2y-\tilde q_2y=0 & {\rm in}\, \Omega_T\\
y=\tilde h_1,\, \Delta u_j=\tilde h_2 & {\rm on}\, \Gamma_T,\\
y(T)=0,\, \partial_ty(0)=y_T & {\rm in}\, \Omega,
\end{cases}
\end{equation}
where the boundary data $\tilde h_1,\tilde h_2$ and the terminal value $y_T$ can be determined by constructing GO solutions for the above $y$-equation (see \eqref{GO-solutions-vanish-T} for such constructed GO solutions). Recalling the assumption that
\begin{equation}\label{Lam-g-initial-data}
\Lambda^{\Gamma,T}_{g_1,\eta_{11},\eta_{21}}(h_1,h_2)=\Lambda^{\Gamma,T}_{g_2,\eta_{12},\eta_{22}}(h_1,h_2),\quad \text{for all}\,\, (h_1,h_2)\in\mathcal S_\epsilon(\Gamma_T),
\end{equation}
we can get by first-order linearization, 
\begin{equation}
\tilde v^{(k)}=\Delta \tilde v^{(k)}=\partial_\nu \tilde v^{(k)}=\partial_\nu\Delta \tilde v^{(k)}=0\quad {\rm on}\,\, \Gamma_T,\quad \tilde v^{(k)}(T)=0\quad {\rm in}\,\, \Omega.
\end{equation}
Multiplying the first equation in \eqref{eq:lin-v(k)} by $y$, using integration by parts, we have
\begin{equation}
\int_0^T\int_\Omega(\tilde q_1-\tilde q_2)v_1^{(k)}y\, dxdt=0.
\end{equation}
Recalling that $g_1,g_2\in\mathcal G^{t_1,t_2}_{g_0}$, which implies ${\rm supp}\,(\tilde q_1-\tilde q_2)\subset\overline\Omega\times[t_1,t_2]$, we have
\begin{equation}
\int_{t_1}^{t_2}\int_\Omega(\tilde q_1-\tilde q_2)v_1^{(k)}y\, dxdt=0.
\end{equation}
By some similar arguments to the proof of the second assertion (2) in Theorem \ref{thm:deter-potential-coeffi}, we can derive that $q_0\vcentcolon=\tilde q_1=\tilde q_2$ in $\Omega_T$. Hence, we have $v^{(k)}\vcentcolon=v_1^{(k)}=v_2^{(k)}$ and $\partial_zg_1(x,t,w_1)=\partial_zg_2(x,t,w_2)$ in $\Omega_T$. It is worth noting that the arguments are valid for all $k=1,\cdots,M$ for $M\ge 1$ and will be applied in later steps, even though the analysis is presented for the case $M=1$ in the first-order linearization.

{\bf Step 2.} \emph{Second-order linearization $(M=2)$}. We differentiate the nonlinear equation \eqref{eq:non-prove-thm-non-initial} with respect to the parameters $\varepsilon_1$ and $\varepsilon_2$ to see that $v_j^{(1,2)}(x,t)=\partial_{\varepsilon_1,\varepsilon_2}^2\big|_{\varepsilon=0}u_j(x,t;\varepsilon)$
satisfies
\begin{equation}
\begin{cases}
(I-\gamma\Delta)\partial_t^2v_j^{(1,2)}+\Delta^2v_j^{(1,2)}-q_0v_j^{(1,2)}=\partial_z^2g_j(x,t,w_j)v^{(1)}v^{(2)} & {\rm in}\, \Omega_T\vspace{0.8ex}\\
v_j^{(1,2)}=\Delta v_j^{(1,2)}=0 & {\rm on}\, \Gamma_T,\vspace{0.8ex}\\
v_j^{(1,2)}(0)=\partial_tv_j^{(1,2)}(0)=0 & {\rm in}\, \Omega,
\end{cases}
\end{equation}
for each $j=1,2$, where $v^{(k)}\in E^T_3$ ($k=1,2$) satisfy
\begin{equation}\label{eq:for-v-k-k=1-2}
\begin{cases}
(I-\gamma\Delta)\partial_t^2v^{(k)}+\Delta^2v^{(k)}-q_0v^{(k)}=0 & {\rm in}\, \Omega_T,\vspace{0.8ex}\\
v^{(k)}=h_{1k},\, \Delta v^{(k)}=h_{2k}& {\rm on}\, \Gamma_T,\vspace{0.8ex}\\
v^{(k)}(0)=\partial_tv^{(k)}(0)=0 & {\rm in}\, \Omega.
\end{cases}
\end{equation}
Let $\tilde v^{(1,2)}\vcentcolon=v_1^{(1,2)}-v_2^{(1,2)}$. Using the assumption \eqref{Lam-g-initial-data}, we can obtain
\begin{equation}
\tilde v^{(1,2)}=\Delta \tilde v^{(1,2)}=\partial_\nu \tilde v^{(1,2)}=\partial_\nu\Delta \tilde v^{(1,2)}=0\quad {\rm on}\,\, \Gamma_T,\quad \tilde v^{(1,2)}(T)=0\quad {\rm in}\,\, \Omega.
\end{equation}
Again multiplying $y$ to the equation of $v^{(1,2)}$, we can get
\begin{equation}
\int_0^T\int_\Omega[\partial_z^2g_1(x,t,w_1)-\partial_z^2g_2(x,t,w_2)]v^{(1)}v^{(2)}y\, dxdt=0.
\end{equation}
Using suitable GO solutions for $y$ and the Runge approximation for $v^{(2)}$, and recalling the condition that $g_j\in\mathcal G_{g_0}^{t_1,t_2}$ for $j=1,2$, we can argue similarly to get
\begin{equation}
[\partial_z^2g_1(x,t,w_1)-\partial_z^2g_2(x,t,w_2)]v^{(1)}=0\quad {\rm in}\,\, \Omega\times (t_1,t_2).
\end{equation}
Hence, for any $v_0$ satisfying the equation
\begin{equation}
(I-\gamma\Delta)\partial_t^2v_0+\Delta^2v_0-q_0v_0=0\quad {\rm in}\,\, \Omega\times (t_1,t_2),
\end{equation}
it holds that
\begin{equation}
\int_{t_1}^{t_2}\int_\Omega[\partial_z^2g_1(x,t,w_1)-\partial_z^2g_2(x,t,w_2)]v^{(1)}v_0\, dxdt=0.
\end{equation}
Again by some similar arguments to the proof of the second assertion (2) in Theorem \ref{thm:deter-potential-coeffi}, we can get $\partial_z^2g_1(x,t,w_1)=\partial_z^2g_2(x,t,w_2)$ in $\Omega_T$ as desired. Moreover, by uniqueness of solutions, we have $v^{(1,2)}\vcentcolon=v_1^{(1,2)}=v_2^{(1,2)}$ in $\Omega_T$.

{\bf Step 3.} \emph{Higher-order linearization $(M\ge 3)$.} By induction, we assume that 
\begin{equation}
\partial_z^kg_1(x,t,w_1(x,t))=\partial_z^kg_2(x,t,w_2(x,t))\quad {\rm in}\,\,\Omega_T
\end{equation}
for $k=1,\cdots, M-1.$
We are required to show that
\begin{equation}
\partial_z^Mg_1(x,t,w_1(x,t))=\partial_z^Mg_2(x,t,w_2(x,t))\quad {\rm in}\,\,\Omega_T.
\end{equation}
To this end, we differentiate the equation \eqref{eq:non-prove-thm-non-initial} with respect to $\varepsilon_1,\cdots,\varepsilon_{M-1}$ and $\varepsilon_M$ to get $v_j^{(1,2,\cdots, M)}=\partial^M_{\varepsilon_1,\cdots,\varepsilon_M}\big|_{\varepsilon=0}u(x,t;\varepsilon)$ satisfying the equation
\begin{equation}\label{eq:M-th-linearition}
\begin{split}
&(I-\gamma\Delta)\partial_t^2v_j^{(1,2,\cdots, M)}+\Delta^2v_j^{(1,2,\cdots, M)}-q_0v_j^{(1,2,\cdots, M)}\\
&=\partial_z^Mg_j(x,t,w_j(x,t))v^{(1)}\cdots v^{(M)}\quad {\rm in}\,\, \Omega_T,\quad j=1,2,
\end{split}
\end{equation}
where for each $k=1,\cdots, M$, $v^{(k)}$ is the solution to the following equation
\begin{equation}
\begin{cases}
(I-\gamma\Delta)\partial_t^2v^{(k)}+\Delta^2v^{(k)}-q_0v^{(k)}=0 & {\rm in}\, \Omega_T,\vspace{0.8ex}\\
v^{(k)}=h_{1k},\, \Delta v^{(k)}=h_{2k}& {\rm on}\, \Gamma_T,\vspace{0.8ex}\\
v^{(k)}(0)=\partial_tv^{(k)}(0)=0 & {\rm in}\, \Omega.
\end{cases}
\end{equation}
Since $v^{(k)}\in E^T_3\hookrightarrow L^\infty(\Omega_T)$ for $k=1,\cdots, M$, and $g_j\in\mathcal G^{t_1,t_2}_{g_0}$, we know that the right-hand side of \eqref{eq:M-th-linearition} belong to $L^2(\Omega_T)$ for $j=1,2$.
Similar to the previous steps, using the backward equation \eqref{eq:back-y-prove-thm3}, we can conclude that
\begin{equation}
\int_{t_1}^{t_2}\int_\Omega[\partial_z^Mg_1(x,t,w_1(x,t))-\partial_z^Mg_2(x,t,w_2(x,t))]yv^{(1)}\cdots v^{(M)}\, dxdt=0.
\end{equation}
Applying Runge approximation for $v^{(1)}$ and choosing suitable GO solutions, we can derive that
\begin{equation}\label{ineq:odd-M-even-M}
[\partial_z^Mg_1(x,t,w_1(x,t))-\partial_z^Mg_2(x,t,w_2(x,t))]v^{(2)}\cdots v^{(M)}=0\quad {\rm in}\,\, \Omega\times (t_1,t_2).
\end{equation}
If $M$ is odd, then we can take the pairs $(v^{(2)},v^{(3)}),\cdots, (v^{(M-1)},v^{(M)})$, and choose successively GO solutions for these pairs by approximation. Otherwise, we add a GO solution to \eqref{ineq:odd-M-even-M}, in order to guarantee even solutions to be multiplied together. In both cases, we can finally deduce that
\begin{equation}
\partial_z^Mg_1(x,t,w_1(x,t))=\partial_z^Mg_2(x,t,w_2(x,t))\quad {\rm in}\,\, \Omega_T.
\end{equation}

{\bf Step 4.} \emph{Recovery of initial data.}
Using the assumption $g_j\in\mathcal G^{t_1,t_2}_{g_0}$, we have $g_1(x,t,0)=g_2(x,t,0)$. Applying $\partial_z^kg_1(x,t,w_1(x,t))=\partial_z^kg_2(x,t,w_2(x,t))$ for any $k\in\mathbb N$, and the analyticity of $g_j$ for $j=1,2$,  we can obtain
\begin{equation}\label{ineq:g-1-g-2}
\begin{split}
&g_1(x,t,w_1(x,t))-g_2(x,t,w_2(x,t))\\
&=\sum_{k=1}^\infty\frac{(-1)^k}{k!}\bigl[\partial_z^kg_2(x,t,w_2(x,t))[w_2(x,t)]^k-\partial_z^kg_1(x,t,w_1(x,t))[w_1(x,t)]^k\bigr]\\
&=\sum_{k=1}^\infty\frac{(-1)^k}{k!}\partial_z^kg_2(x,t,w_2(x,t))\bigl[[w_2(x,t)]^k-[w_1(x,t)]^k \bigr].
\end{split}
\end{equation}
Hence, for any $L>0$, we can estimate
\begin{equation}
\begin{split}
&\Big|G(x,t)\vcentcolon=\frac{g_1(x,t,w_1(x,t))-g_2(x,t,w_2(x,t))}{w_1(x,t)-w_2(x,t)}\Big|\\
&=\Big|\sum_{k=1}^\infty\frac{(-1)^k}{k!}\partial_z^kg_2(x,t,w_2(x,t))\sum_{l=0}^{k-1}[w_2(x,t)]^{k-1-l}[w_1(x,t)]^l\Big|\\
&\le \sum_{k=1}^\infty\big| \partial_z^kg_2(x,t,w_2(x,t)) \big|\frac{R^{k-1}}{(k-1)!}\le \sum_{k=1}^\infty\frac{kR^{k-1}}{L^k}\sup_{|z-w_2(x,t)|=L}|g_2(x,t,z)|,
\end{split}
\end{equation}
where we have set $R=\|w_1\|_{L^\infty(\Omega_T)}+\|w_2\|_{L^\infty(\Omega_T)}$. Now, choosing $L=2(R+1)$, we derive that $G\in L^\infty(\Omega_T)$. Let $\tilde w=w_1-w_2$. We see that
\begin{equation}
\begin{cases}
(I-\gamma\Delta)\partial_t^2\tilde w+\Delta^2\tilde w-G\tilde w=0 & {\rm in}\, \Omega_T,\\
\tilde w=\Delta \tilde w=0& {\rm on}\, \Gamma_T,\\
\tilde w(0)=\eta_{11}-\eta_{12},\, \partial_t\tilde w(0)=\eta_{21}-\eta_{22} & {\rm in}\, \Omega.
\end{cases}
\end{equation}
By the observability \eqref{est:observability} for the above equation, and the condition \eqref{Lam-g-initial-data}, we can get $\eta_{11}=\eta_{12}$, $\eta_{21}=\eta_{22}$ in $\Omega$, and $w_0\vcentcolon=w_1=w_2$ in $\Omega_T$. Hence, it follows from \eqref{ineq:g-1-g-2} that
$g_1(x,t,w_0(x,t))=g_2(x,t,w_0(x,t))$ in $\Omega_T$. This finally implies $g_1(x,t,z)=g_2(x,t,z)$ for $(x,t,z)\in\Omega_T\times\mathbb C$ by observing the following equality
\begin{equation}
g_j(x,t,z)=g_j(x,t,w_0(x,t))+\sum_{k=1}^\infty\partial_zg_j^k(x,t,w_0(x,t))\frac{(z-w_0(x,t))^k}{k!},\quad j=1,2.
\end{equation}
Therefore, the proof of Theorem \ref{thm:deter-ini-da-nonlin} is finished. \hfill $\square$

\section{Conclusions}\label{sec:conclusions}
In this paper, we are devoted to studying well-posedness and inverse problems for both linear and nonlinear Kirchhoff plate equations. We use passive measurements to recover initial data (without the smallness condition) for semilinear equations with certain nonlinearities in a stable way. By using various active measurements, the simultaneous recovery of (possibly unbounded) linear and nonlinear coefficients and initial data are obtained. The methodology of this paper mainly consists of the construction of suitable GO solutions vanishing at $t=T$, together with two generalized Runge approximation results for linear Kirchhoff plate equations. We believe that all the results obtained in this paper can be extended to Kirchhoff plate equations with Dirichlet boundary conditions $(u,\partial_\nu u)|_{\Gamma_T}$. It is also expected that these results carry over to Euler-Bernoulli plate equations with unbounded potentials. We are also interested in the simultaneous recovery of multiple parameters for Kirchhoff plate equations by using a single (pair) passive measurement, and the stable recovery results, but we leave this topic in future research.

Roughly speaking, (exact) controllability is a favorable feature of PDEs that facilitates Runge approximation, the solvability of inverse problems, and unique continuation. Based on the observability inequality, we have established a Runge approximation theory that is essential for handling the vanishing initial (or terminal) data when applying GO solutions to Kirchhoff plate equations. Moreover, we note that the Runge approximation theory is of independent interest and has important applications in inverse problems involving local and nonlocal PDEs (see, for instance, \cite{TGMS-APDE,lin2024well,LTZ1-CVPDE}). This property and its quantitative refinement deserve deeper study, especially regarding the stability estimates for inverse plate equations.  Nevertheless, in the context of our analysis of inverse Kirchhoff plate equations with zero initial (or terminal) conditions, as an alternative way, suitable cut-off functions can be chosen so that the corresponding GO solutions vanish at $t=0$ or $t=T$. 
Below, we present some details in the end of this section.

\subsection*{A direct cut-off procedure}
Let us first recall the GO solution
\begin{equation}
v(x,t)=e^{{\rm i}\sigma(x\cdot\theta+t)}[a_0(x,t)+\sigma^{-1}a_1(x,t)+\sigma^{-2}a_2(x,t)]+r_\sigma(x,t),
\end{equation}
to the equation $\mathcal L_{q_1}v=0$, where $a_0$ satisfies the transport equation
\begin{equation}
\gamma a_{0t}-\nabla a_0\cdot\theta=0.
\end{equation}
For any given $\phi\in C^\infty(\mathbb R^n)$,  we can check that 
\begin{equation}
a_0=e^{-{\rm i}(x,t)\cdot\zeta}\phi(x+t\tilde\theta),\quad \tilde\theta\in\mathbb S^{n-1}(\gamma^{-\frac12}),
\end{equation}
is a solution to the above transport equation, where $\zeta\in\mathbb R^{n+1}$ satisfying $\zeta\cdot(1,-\tilde\theta)=0$. With $a_0$ at hand, similar to the arguments as that in Section \ref{Construction-GO-solution}, we can solve certain inhomogeneous transport equations to construct $a_1$ and $a_2$. 

Let $T^*\vcentcolon=\max_{\overline\Omega}|x|$, and let $\delta>0$ be an arbitrarily small constant. We choose $\phi\in C^\infty(\mathbb R^n)$ be a cut-off function satisfying
\begin{equation}
\begin{cases}
\phi(x)=0, & |x|\in [0,T^*+\delta],\\
0<\phi(x)<1, & |x|\in (T^*+\delta,T^*+2\delta),\\
\phi(x)=1 & |x|\in [T^*+2\delta,\infty).
\end{cases}
\end{equation}
Denote by
\begin{equation}
\widetilde T^*\vcentcolon=2\gamma^{\frac12}\bigl(T^*+\delta  \bigr)<T.
\end{equation}
This implies $a_0(\cdot,0)=\partial_ta_0(\cdot,0)=0$. Using the relations \eqref{expre-ampli-a1} and \eqref{expre-ampli-a2}, we can also get $a_1(\cdot,0)=\partial_ta_1(\cdot,0)=a_2(\cdot,0)=\partial_ta_2(\cdot,0)=0$ in $\overline\Omega$, yielding that  $v(\cdot,0)=\partial_tv(\cdot,0)=0$ in $\overline\Omega$.
Moreover, if $t\in [\widetilde T^*,T]$, then
\begin{equation}
|x+t\tilde\theta|\ge 2T^*+2\delta-\max_{\overline\Omega}|x|= T^*+2\delta,
\end{equation}
implying that $\phi(x+t\tilde\theta)=1$ in $\overline\Omega\times [\widetilde T^*,T]$. Hence, we can construct suitable GO solutions for the equation \eqref{eq:w-j-zero-initial-data} that vanish at $t=0$. If we assume that $q_j=q_0$ in $\Omega\times [0,\widetilde T^*+2\delta)$ for $j=1,2$, where $q_0\in L^\infty(0,T;L^{p_0}(\Omega))$ is some known function, then ${\supp}\,(q_1-q_2)\subset\overline\Omega\times[\widetilde T^*+2\delta,T]$. By some similar arguments to the proof of the second assertion (2) in Theorem \ref{thm:deter-potential-coeffi}, we can derive the following identity
\begin{equation}
\int_0^T\int_\Omega(q_1-q_2)e^{-{\rm i}(x,t)\cdot\zeta}\,dxdt=0,
\end{equation}
yielding $q_1=q_2$ in $\Omega_T$.

Similarly, for the GO solution
\begin{equation}
y(x,t)=e^{-{\rm i}\sigma(x\cdot\theta+t)}[b_0(x,t)+\sigma^{-1}b_1(x,t)+\sigma^{-2}b_2(x,t)]+\tilde r_\sigma(x,t),
\end{equation}
we can choose $b_0(x,t)=\tilde\phi(x+t\tilde\theta)$, which is a solution to the transport equation
\begin{equation}
\gamma b_{0t}-\nabla b_0\cdot\theta=0.
\end{equation}
Here, $\tilde\phi\in C^\infty(\mathbb R^n)$ is given by
\begin{equation}
\begin{cases}
\tilde\phi(x)=0, & |x|\in [\gamma^{-\frac12}T-T^*-\delta,\infty),\\
0<\tilde\phi(x)<1, & |x|\in [\gamma^{-\frac12}-T^*-2\delta,\gamma^{-\frac12}T-T^*-\delta),\\
\tilde\phi(x)=1, & |x|\in [0,\gamma^{-\frac12}T-T^*-2\delta].
\end{cases}
\end{equation}
Observing that 
$$|x+T\tilde\theta|\ge \gamma^{-\frac12}T-\max_{\overline\Omega}|x|=\gamma^{-\frac12}T-T^*>\gamma^{-\frac12}T-T^*-\delta,$$   
for any arbitrary small $\delta>0$. This implies that $b_0(x,T)=\tilde\phi(x+T\tilde\theta)=0$, and $\partial_tb_0(\cdot,T)=0$ for any $x\in\overline\Omega$. Moreover, using the inhomogeneous transport equations for $b_1$ and $b_2$, we can obtain $\partial_t^kb_1(\cdot,T)=\partial_t^kb_2(\cdot,T)=0$ for $k=0,1$, yielding that $y(\cdot,T)=\partial_ty(\cdot,T)=0$ in $\overline\Omega$.
Let $T>\widetilde T^*$. Whenever $t\in [0,T-\widetilde T^*]$, we can check
\begin{equation}
\begin{split}
0\le |x+t\tilde\theta|&\le \max_{\overline\Omega}|x|+\gamma^{-\frac12}(T-\widetilde T^*)=T^*+\gamma^{-\frac12}\bigl[T-2\gamma^{\frac12}(\widetilde T^*+\delta)\bigr]\\
&=\gamma^{-\frac12}T-T^*-2\delta.
\end{split}
\end{equation}
Hence, we conclude that $\tilde\phi(x+t\tilde\theta)=1$ for any $(x,t)\in\overline\Omega\times[0,T-\widetilde T^*]$. 

Now, let $T>2(\widetilde T^*+\delta)$. Suppose that $q_j=q_0$ in $\Omega\times\bigl([0,\widetilde T^*+2\delta)\cup (T-\widetilde T^*,T]\bigr)$ for $j=1,2$. Then we see that ${\rm supp}\,(q_1-q_2)\subset\overline\Omega\times[\widetilde T^*+2\delta,T-\widetilde T^*]$. Similar to the arguments to the proof of the third assertion (3) in Theorem \ref{thm:deter-potential-coeffi}, we can derive
\begin{equation}
\int_{t_1}^{t_2}\int_\Omega(q_1-q_2)e^{-{\rm i}(x,t)\cdot\zeta}\,dxdt=\int_0^T\int_\Omega(q_1-q_2)a_0b_0\, dxdt=0,
\end{equation}
where $t_1=\widetilde T^*+2\delta$ and $t_2=T-\widetilde T^*$. Hence, it follows that $q_1=q_2$ in $\Omega_T$.

\medskip

\subsection*{ Acknowledgments}  

S. Fu is supported by the  Fundamental Research Funds for the Central Universities, NPU, under grant number D5000250416, and the key program of the National Natural Science Foundation of China under grant number 62433020.
H. Liu is supported by the Hong Kong RGC General Research Funds (projects 11311122, 11300821, and 11303125), the NSFC/RGC Joint Research Fund (project  N\_CityU101/21), the France-Hong Kong ANR/RGC Joint Research Grant, A-CityU203/19. 
Y. Yu is supported by the National Natural Science Foundation of China under grant number 12401579.

\subsection*{Statements and Declarations} No datasets were generated or analyzed during the current study. 

	
\subsection*{Conflict of Interests} Hereby we declare there are no conflict of interests.
		
\bibliography{refs} 
	
\bibliographystyle{alpha}

\end{document}